\documentclass[11pt,reqno]{amsart}
\usepackage{color}
\usepackage{amsmath,amssymb,amsfonts,amsthm,bm}
\usepackage{mathrsfs}
\usepackage{graphicx}
\usepackage{enumerate}
\usepackage[numbers, sort&compress]{natbib}
\usepackage[shortlabels]{enumitem}
\usepackage{a4wide}
\usepackage[colorlinks]{hyperref}
 \newtheorem{theorem}{Theorem}[section]
 \newtheorem{corollary}[theorem]{Corollary}
 \newtheorem{lemma}[theorem]{Lemma}
 \newtheorem{proposition}[theorem]{Proposition}
\theoremstyle{definition}
\newtheorem{definition}[theorem]{Definition}
 \theoremstyle{remark}

\numberwithin{equation}{section}

 \newcommand{\eps}{\varepsilon} 
\newcommand{\abs}[1]{\left\vert#1\right\vert}
\newcommand{\set}[1]{\left\{#1\right\}}
\newcommand{\norm}[1]{\Vert#1\Vert}
\newcommand{\normm}[1]{| \! | \! | #1| \! | \! | }
\newcommand{\inner}[1]{\left(#1\right)}
\newcommand{\comi}[1]{\left<#1\right>}
\newcommand{\com}[1]{\left[#1\right]}
\newcommand{\CG}{\mathcal{G}}
\newcommand{\CT}{\mathcal{T}}

\makeatletter
\@namedef{subjclassname@2020}{%
  \textup{2020} Mathematics Subject Classification}
\makeatother
\begin{document}

\title[Gevrey regularity  for the  non-cutoff  Boltzmann equation]
 {Gevrey regularity of mild solutions to the non-cutoff Boltzmann equation}

%----------Author 1
\author[R.J. Duan ]{Renjun Duan}

\address[R.J. Duan ]{Department of Mathematics
	The Chinese University of Hong Kong
	Shatin, Hong Kong,
	P.R. CHINA}

\email{rjduan@math.cuhk.edu.hk}

%----------Author 2
\author[W.-X. Li]{Wei-Xi Li}
\address[W.-X. Li]{School of Mathematics and Statistics, and Computational Science Hubei Key Laboratory, Wuhan
	University, 430072 Wuhan, China}
\email{wei-xi.li@whu.edu.cn}

\author[L.Q. Liu]{Lvqiao Liu}
\address[L.Q. Liu]{School of Mathematics and Statistics,   Anhui Normal University,  241000 Wuhu, China}
\email{lvqiaoliu@whu.edu.cn}



\date{}

\begin{abstract}
In the paper, for the Cauchy problem on the non-cutoff Boltzmann equation in torus, we establish the global-in-time Gevrey smoothness in velocity and space variables for a class of low-regularity mild solutions near Maxwellians with the Gevrey index depending only on the angular singularity. This together with  \cite{MR4230064} provides a self-contained well-posedness theory for both existence and regularity of global solutions for initial data of low regularity in the framework of perturbations. For the proof we treat in a subtle way  the commutator between the regularization operators and the Boltzmann collision operator involving rough coefficients, and this enables us to combine the classical  H\"ormander's hypoelliptic techniques together with the global symbolic calculus established for the linearized Boltzmann operator so as to improve the regularity of solutions at positive time.
\end{abstract}

%%% ----------------------------------------------------------------------
\maketitle

%%% ----------------------------------------------------------------------
\tableofcontents

\section{Introduction}
The global well-posedness, such as existence, uniqueness and regularity, for the nonlinear Boltzmann equation in the spatially inhomogeneous setting, is a fundamental mathematical problem in collisional kinetic theory. On one hand, for general initial data with finite mass, energy and entropy, it is well known that the global existence of appropriate renormalized weak solutions in $L^1$ was established first by DiPerna-Lions \cite{MR1014927} under the Grad's angular cutoff assumption and later by Alexandre-Villani \cite{MR1857879} for the non-cutoff collision kernel with physically realistic long-range interactions even including the Coulomb potential, while both uniqueness and regularity of such general global solutions have remained largely open. 

On the other hand, whenever initial data are  close enough to Maxwellians in a certain sense, the theory of unique existence of global classical solutions is well investigated from different aspects. In the cutoff case, it was  started first by Ukai \cite{MR363332} using the spectral analysis and later by Liu-Yu \cite{MR2044894}, Liu-Yang-Yu \cite{LYY} and Guo \cite{MR2095473, MR1946444} using the energy method. In the non-cutoff case,  AMUXY \cite{MR2793203,MR2795331,MR2863853} and Gressman-Strain \cite{MR2784329} independently constructed the unique global classical solutions near global Maxwellians in the whole space and in the torus domain, respectively. In those aforementioned works, in order to treat  the nonlinearity in space variables, the high-order Sobolev space $H^{\ell}$ with $\ell$ large enough, for instance $\ell>3/2$ in case of three space dimensions, was used to control the $L^\infty$ norm of the solutions through the Sobolev embedding. In the meantime, those theories are focused on the perturbation with the Gaussian tail in large velocity.  Thus it has been a challenging task to look for an enlarged function space of  the perturbed solutions either with the lower regularity in space and velocity variables or with the slower large-velocity decay such that the existence  with the uniqueness principle still can be achieved. Indeed, a lot of great progresses have been made in these two directions for the cutoff collision kernel. In particular, we would mention that 
\begin{itemize}
	\item Guo \cite{MR2679358} developed a robust $L^2-L^\infty$ interplay approach for the cutoff Boltzmann equation in general bounded domains. Since then, there have existed extensive studies of formulation of singularity and regularity of such solutions induced by the boundary, see \cite{MR3592757} and references therein. Note that an $L^2\cap L^\infty_\beta$ solution for the case of the whole space was also constructed in \cite{MR2239407} by the multiple Duhamel iterations.
	\item Gualdani-Mischler-Mouhot \cite{MR3779780} developed a new general perturbation theory with the perturbation having the only polynomial tail in large velocity for the cutoff Boltzmann equation in torus, and it also produced many applications in kinetic theory, for instance, \cite{MR4201411, MR3625186,MR3483898} and references therein.
	%  \item 
\end{itemize}
%\noindent 

It is extremely hard to extend those results in \cite{MR3779780,MR2679358} to the non-cutoff Boltzmann equation due to the difficulty arising from the angular singularity of the Boltzmann collision integral, whereas we may refer to \cite{MR4076068,MR4171912} and \cite{MR3625186,MR3483898} for the corresponding development in the context of the Landau equation with the explicit velocity diffusion property. Recently, for the non-cutoff Boltzmann equation in torus, Alonso-Morimoto-Sun-Yang \cite{MR4201411} extent  the result in \cite{MR3779780} to construct the classical solutions with polynomial tails; see also the independent works \cite{2017arXiv171000315H} and \cite{MR4107942}. Moreover, the same authors \cite{2020arXiv201010065A} established the unique existence of solutions in the $L^2\cap L^\infty$ setting via the De Giorgi type argument (cf.~\cite{MR3923847}) with the help of a strong averaging lemma. 

In the current work, we are devoted to studying the smoothness of a class of unique low-regularity solutions to the non-cutoff Boltzmann equation near global Maxwellians in torus. In fact, instead of directly using the $L^\infty$ space, motivated by the early works \cite{MR3461360,MR3815145,MR3532066}, the construction of solutions can be carried out in the Wiener algebra that is the space of all integrable functions on the torus whose Fourier series are absolutely convergent, cf.~\cite{MR4230064}. The goal of this work is to further establish the Gevrey smoothness of the solution in both space and velocity variables uniformly for all positive time with the Gevrey index depending only on the angular singularity. This then provides a complete well-posedness Boltzmann theory for existence, unqueness and regularity of global solutions with low-regularity initial data. The further review with focus on the regularity issue for the non-cutoff Boltzmann equation will be provided later on. 

\subsection{Boltzmann equation} 
The spatially inhomogeneous Boltzmann equation in torus reads as 
 \begin{equation}\label{1}
 \begin{array}{ll}
 \partial_{t}F+v\cdot\nabla_{x}F=Q(F,F). 
 %\quad F|_{t=0}=F_0,
 \end{array}
 \end{equation}
Here, the unknown $F(t,x,v)\geq 0$ stands for the density distribution function of gas particles with position $x=(x_1,x_2,x_3) \in \mathbb{T}^3$  and velocity $v=(v_1, v_2, v_3)\in \mathbb{R}^3$ at time $t>0$. The Boltzmann collision operator on the right hand side of \eqref{1} is bilinear and acts only on velocity variables, taking the form of
 \begin{equation*}%\label{collis}
 Q(G,F)(v)=\int_{\mathbb{R}^3}\int_{\mathbb S^2}B(v-v_*,\sigma)[G(v_\ast')F(v')-G(v_\ast)F(v)] \,d \sigma d v_\ast.
 \end{equation*}
 In the above integrand the velocity pairs  $(v^{\prime},v^{\prime}_\ast) $ and $(v,v_\ast)$  are given by the relation
 \begin{equation*}
 \left\{
 \begin{aligned}
 &v'  = \frac{v+v_*}{2} + \frac{ |v-v_*|}{2} \sigma, \\
 &v'_* = \frac{v+v_*}{2} - \frac{ |v-v_*|}{2} \sigma,
 \end{aligned}
 \right. 
 \end{equation*}
 with $\sigma\in \mathbb{S}^2$, according to conservations of molecular momentum and energy before and after an elastic collision
 \begin{equation*}
 v^{\prime}+v^{\prime}_{*}=v+v_{*} , ~~|v^{\prime}|^{2}+|v^{\prime}_{*}|^{2}=|v|^{2}+|v_{*}|^{2}.
 \end{equation*}
 Moreover, the cross section $B(v-v_*,\sigma)$ depends only on the relative speed $\abs{v-v_*}$ and the  deviation angle $\theta$  with 
% \begin{equation*}
 $\cos \theta =   \langle \frac{ v-v_*}{|v-v_*|},  \sigma\rangle$.
 %\end{equation*}
 We assume
that $B(v-v_\ast,\sigma)$ is supported without loss of generality on $0\leq \theta\leq \pi/2$ such that $\cos\theta\geq 0$ and also assume that it takes the specific form
 \begin{equation}\label{kern}
 { B}(v-v_*, \sigma) =  |v-v_*|^\gamma  b (\cos \theta),
 \end{equation}
where $|v-v_*|^\gamma$ is called the kinetic part with   $-3<\gamma\leq 1$, and $b (\cos \theta)$ is called the angular part satisfying that there are $C_b>1$ and $0<s<1$ such that 
\begin{equation}\label{angu}
 \frac{1}{C_b \theta^{1+2s}}\leq \sin \theta   b(\cos \theta) \leq \frac{C_b}{ \theta^{1+2s}}, \qquad \forall\,\theta \in (0, \frac{\pi}{2}],
\end{equation}
%For the angular non-cutoff Boltzmann operator, it is also customary to use the terms
%“hard potentials” when $\gamma\geq 0$, “moderately soft potentials” when  $-2s \leq \gamma <0$, and “very soft potentials” when $ -3\leq \gamma <-2s$.

 %  Motivated by  the lower and upper bounds of the non-isotropic norm introduced  by    Alexandre-Morimoto-Ukai-Xu -Yang  \cite{AMUXY1},   the third author \cite{MR3193940} considered the global model 
%  \begin{eqnarray*}
%  	\partial_t+v\partial_x+c(v) \big(-\tilde \Delta_v\big)^s,
%  \end{eqnarray*}
%with $\comi{v}^\gamma\lesssim c(v) \lesssim \comi{v}^{\gamma+2s}$ for  $\gamma\in ]-3, 1]$  so that $c(v)$ may go to 0 as $\abs{v}\rightarrow +\infty.$   Here and throughout the paper
%\begin{eqnarray*}
%	\comi{\cdot}\stackrel{\textrm{def}}{ =} \inner{1+\abs{\cdot}^2}^{1/2}.
%\end{eqnarray*} 
%The maximal regularity was given in \cite{MR3193940}, using the global pseudo-differential calculus.  
 
We are concerned with the solution to the Boltzmann equation \eqref{1} around  the normalized global Maxwellian
%\begin{equation*}
      $\mu=\mu(v)=(2\pi)^{-3/2}e^{-|v|^{2}/2}$.
%\end{equation*}
Thus, let $F(t,x,v)=\mu+\sqrt{\mu}f(t,x,v)$, then the reformulated unknown $f=f(t,x,v)$ satisfies that
\begin{eqnarray}\label{eqforper}
 \partial_{t}f+v\cdot\nabla_{x}f-\mathcal{L} f
      =\Gamma(f,f),
\end{eqnarray}
with the   linearized collision operator $\mathcal L$ and the nonlinear collision operator $\Gamma(\cdot,\cdot)$ respectively given  by
\begin{equation}\label{linearbol}
 \mathcal L f=\mu^{-1/2}Q(\mu,\sqrt{\mu} f)+\mu^{-1/2}Q(\sqrt{\mu}f,\mu), 
 \end{equation}
 and  
 \begin{equation}
 \label{def.nlt}
 	\Gamma(g,h)=\mu^{-1/2}Q(\sqrt{\mu}g,\sqrt{\mu}h).
 \end{equation}

 Due to the fact that the Boltzmann collision term $Q(F,F)$ admits five collision invariants $1,v$ and $|v|^2$, a solution of \eqref{1} with suitable regularity and integrability in velocity has conservations of total mass, momentum and energy, so that for simplicity we 
 always assume that $f(t,x, v) $ satisfies
\begin{equation}
\label{con.cl}
\int_{ {\mathbb T^3}} \int_{\mathbb R^3}(1,v,|v|^2)\sqrt{\mu}f(t, x, v) \,d v d x=0
\end{equation}
% \begin{equation}
%\label{con.cl}
%\left\{\begin{aligned}
%\int_{ {\mathbb T^3}} \int_{\mathbb R^3}\sqrt{\mu}f(t, x, v) \,d v d x&=0,\\
%\int_{{\mathbb T^3}} \int_{\mathbb R^3} v_i\sqrt{\mu}f(t, x, v) \,d vd x&=0, \qquad  i=1,2,3,\\
%\int_{{\mathbb T^3}} \int_{\mathbb R^3}|v|^2\sqrt{\mu}f(t, x, v) \,d vd x&=0,
%\end{aligned}\right.
%\end{equation}
for any $t\geq 0$. In particular, \eqref{con.cl} should be satisfied for all $t>0$ if it holds true initially.

\subsection{Norms, spaces and results}
The linearized operator $\mathcal{L}$ is self-adjoint and non-positive definite on $L^2_v$, satisfying that there is a constant $c>0$ such that 
\begin{equation}
\label{coer}
-(\mathcal{L} f,f)_{L^2_v}\geq c|f|_{D}^2
\end{equation}
for any $f$ in $(\ker \mathcal{L})^{\perp}$. Here, the dissipative norm $|\cdot|_D$ can be characterized in two kinds of ways  by either the triple norm $\normm{\cdot}$ in \cite{MR2863853} or the anisotropic norm  $ |\cdot|_{N^{\gamma,s}}$ in \cite{MR2784329}, respectively defined as 
\begin{align*}
 \normm f^2:&= \int_{\mathbb R^3}\int_{\mathbb R^3}\int_{\mathbb{S}^2} B(v-v_*,\sigma) \mu_*
  \inner{f-f'}^2\,d\sigma d vd v_\ast  \\
  &\quad+\int_{\mathbb R^3}\int_{\mathbb R^3}\int_{\mathbb{S}^2}  B(v-v_*,\sigma) f_*^2 \inner{\sqrt{\mu'}-
  	\sqrt{\mu}}^2\,d\sigma d vd v_\ast, 
	%\label{tor1}
\end{align*}
and
\begin{equation*}%\label{trnor}
  | f|_{N^{\gamma,s}}^2:= \norm{\comi{v}^{s+\frac{\gamma}{2}} f}_{L^2_v}^2+\int_{\mathbb R^3} \int_{\mathbb R^3} (\comi{v}\comi{v'})^{\frac{\gamma+2s+1}{2}} \frac{(f'-f)^2}{d(v,v')^{3+2s}} {\bf 1}_{d(v,v')\leq1}\,d vd v',
\end{equation*}
where we have used the standard notations $f'=f(v')$, $f_\ast=f(v_\ast)$,  $\mu'=\mu(v')$, $\mu_\ast=\mu(v_\ast)$ for shorthand, $\comi{\cdot}=(1+\abs{\cdot}^2)^{1/2}$, and the anisotropic metric
$
d(v,v')=\{|v-v'|^2+\frac{1}{4} (|v|^2-|v'|^2)^2\}^{1/2}.
$
The third way to characterize $|\cdot|_D$, recently introduced in \cite{MR3950012}, is to use the norm $\|(a^{1/2})^w f\|_{L^2_v}$, where $(a^{1/2})^w$ stands for the Weyl quantization with symbol $a^{1/2}.$   The definition of $a^{1/2}$ as well as some basic facts  on the symbolic calculus will be given in Section  \ref{sec2} later on. Moreover, in terms of  \cite[(2.13)-(2.15)]{MR2784329}, \cite[Proposition 2.1]{MR2863853} and \cite[Theorem 1.2]{MR3950012}, one has  the equivalence of those norms as
  \begin{equation}\label{rela}
 |f|_{D}^2\sim \normm f^2  \sim  | f|_{N^{\gamma,s}}^2\sim \norm{(a^{1/2})^w f}_{L^2 _v}^2\sim -\inner{\mathcal L f, f}_{L^2 _v} + \norm{\comi v^{\ell}f}_{L^2 _v}^2   
  \end{equation}
  for any suitable function $f$ and for any $\ell \in \mathbb{R}$.
  
Throughout the paper we denote the Fourier transform of $f(t,x,v)$ with respect to space variable $x\in \mathbb{T}^3$ by
\begin{equation*}%\label{def.ft}
\hat f(t, k, v)=\mathscr{F}_x f(t,k,v) =\int_{\mathbb{T}^3}e^{-ik\cdot x} f(t,x,v)\,d x, \quad k\in \mathbb{Z}^3.
\end{equation*}
Then, to look for a solution $f=f(t,x,v)$, we define the mixed Lebesgue space $L^p_k L^{ q}_{T} L^r_v$  with the norm
\begin{equation*}
\norm{f}_{L^p_k L^{ q}_{ T } L^r_v}:=
\left\{
\begin{aligned}
 &\left(\int_{\mathbb{Z}^3}\left(\int_0^T   \norm{\hat f(t, k, \cdot)}_{L^r_v}^qd t  \right)^{\frac{p}{q}} \,  d \Sigma(k)\right)^{\frac{1}{p}}, \quad q<\infty,\\
&\left(\int_{\mathbb{Z}^3 }\left(\sup_{0< t< T}  \norm{\hat f(t, k, \cdot)}_{L^r_v}\right)^p\,  d \Sigma(k)\right)^{\frac{1}{p}},  \quad\,\, q=\infty,
\end{aligned}
\right.
\end{equation*}
for $1\leq p,r<\infty$ and $1\leq q\leq \infty$, where
for convenience we have used $  d \Sigma(k)$ through the paper to denote the discrete measure on $\mathbb{Z}^3$, meaning that 
%\begin{equation*}
$\int_{{\mathbb T^3}} g(k)\,  d \Sigma(k)=\sum_{k\in \mathbb{Z}^3}g(k)$
%\end{equation*}
for any summable function $g=g(k)$ on $\mathbb{Z}^3$. For simplicity we also denote the corresponding space $L^p_k L^{ q}_{\tau,T} L^r_v$ with $0\leq \tau \leq T$ whenever functions are restricted only to the time interval $\tau< t< T$ and hence we may write $L^p_k L^{ q}_{T} L^r_v=L^p_k L^{ q}_{0,T} L^r_v$ for $\tau=0$.  
Correspondingly, for an initial datum $f_0=f_0(x,v)$ that does not involve time variable, we define the space $L^p_k   L^r_v$ with the norm
\begin{equation*}
\norm{f_0}_{L^p_k   L^r_v}:=
\left(\int_{\mathbb{Z}^3}  \norm{\hat f_0(k, \cdot)}_{L^r_v}^p    \,  d \Sigma(k)\right)^{\frac{1}{p}},
\end{equation*}
and the norm of higher order in space variables 
\begin{equation*}
\norm{f_0}_{L^p_{k,m}   L^r_v}:=
\left(\int_{\mathbb{Z}^3} \left(\comi{k}^{m} \norm{\hat f_0(k, \cdot)}_{L^r_v}\right)^p   \,  d \Sigma(k)\right)^{\frac{1}{p}},
\end{equation*}
for $1\leq p,r<\infty$ and $m\geq 0$. 

Finally we also introduce the Gevrey space under consideration. We say that  $f=f(x,v)\in \CG^r(\mathbb T^3_x\times \mathbb R^3_v)$  of index $r\geq1$ if $f\in C^\infty (\mathbb T^3_x\times \mathbb R^3_v)$ and there is a constant $C$ such that 
\begin{eqnarray*}
\norm{ \partial_x^\alpha \partial_v^\beta f}_{ L_{x,v}^2}\leq C^{\abs\alpha+\abs\beta+1}\com{ (|\alpha|+|\beta|) !}^{r},\quad  \forall\ \alpha,\beta\in\mathbb Z_+^3.
\end{eqnarray*}
 
With the preparation of notations above, the main result of the paper   is stated as follows. 

\begin{theorem}\label{maith}
Assume \eqref{kern} and \eqref{angu} with $\gamma\geq 0$ and $0 < s < 1$. There are $\varepsilon_0>0$ and $C>0$ such that if the initial datum $F_0$ for \eqref{1} has the form of  $F_0(x,v)=\mu+\sqrt{\mu}f_0(x,v)\geq 0$  with $ f_0\in L^1_kL^2_v$  satisfying  \eqref{con.cl} and 
\begin{equation}\label{123}
 \norm{f_0}_{  L^1_kL^2_v} \leq  
 \varepsilon_0,
\end{equation}
then  the Cauchy problem on the non-cutoff Boltzmann equation \eqref{1} or  \eqref{eqforper} with initial data $F|_{t=0}=F_0$ admits a unique global smooth  solution $F(t,x,v)=\mu+\sqrt{\mu}f(t,x,v)\geq 0$ with $f\in L^1_kL^\infty_TL^2_v$ for any $T>0$ and $f(t,\cdot,\cdot)\in \CG^{\frac{1+2s}{2s}}(\mathbb T^3_x\times \mathbb R^3_v)$ for any $t>0$ such that the following estimate holds true: 
\begin{equation}\label{thm.gest}
 %\begin{split}
 \int_{\mathbb Z^3}  \Big( \sup_{t>0} \phi(t)^{\frac{1+2s}{2s} (m+|\beta|)}  \comi{k}^{m}\norm{ {\partial_v^\beta\hat f (t,k,\cdot)}}_{L^2_v} \Big) \,  d \Sigma(k)  
 \leq   C^{m+|\beta|+1}\com{ (m+|\beta|) !}^{\frac{1+2s}{2s}},
 %\end{split}
 \end{equation}  
for any $m\in\mathbb Z_+$ and $\beta\in\mathbb Z_+^3$,  where $\phi(t):=\min\big\{t,  1\big\}$. In particular, for any $t>0$,    it also holds that $F(t,\cdot,\cdot)\in \CG^{\frac{1+2s}{2s}}(\mathbb T^3_x\times \mathbb R^3_v)$.
 \end{theorem} 

%\begin{remark}
We remark that the global existence and uniqueness of mild solutions $f(t,x,v)$ in the low-regularity space $L^1_kL^\infty_TL^2_v$ have been obtained by \cite{MR4230064}.  This paper aims to establish  without imposing any  additional assumption on initial data  the global smoothness of such low-regularity solutions in the sense that $f(t,\cdot,\cdot)\in \CG^{(1+2s)/2s}$ for any $t>0$ and the quantitative estimate \eqref{thm.gest} is satisfied globally in time.
%\end{remark}

%\newpage
\subsection{Related literature on regularity}
For the Boltzmann collision operator without angular cutoff, the grazing collisions may induce the velocity diffusion similar to the case of the Landau operator.  Due to this, the solution to the non-cutoff Boltzmann equation has a smoothing effect in velocity variables. The rigorous mathematical proof was first given by Desvillettes in \cite{MR1324404} and  \cite{MR1475459} for the non-cutoff Kac equation and for the non-cutoff spatially homogeneous Boltzmann equation with Maxwell molecule potentials in two dimensions, respectively, where $C^\infty$ regularization is proved.  

For general collision kernels with an angular singularity, the fundamental important work Alexandre-Desvillettes-Villani-Wennberg \cite{MR1765272} found out that the Boltzmann operator behaves as the fractional Laplacian operator in velocity variables in the sense that
\begin{equation}
\label{diff}
F\mapsto -Q(G,F) \sim C_G (-\Delta_v)^s F +\text{\rm lower order terms},
\end{equation}  
where the fractional order diffusion property is local for any finite velocity $|v|<R$ with $R<\infty$ so that $C_G$ may also depend on $R$. For \eqref{diff}, the global nonlinear sharp version was resolved by Gressman-Strain \cite{MR2784329}, and the global linearized version  was recently obtained by Alexandre-Li-H\'erau \cite{MR3950012}. Note that the proof of \cite{MR3950012} is based on the  multiplier method and the Wick quantization together with the careful analysis of the symbolic properties of the Weyl symbol of the Boltzmann collision operator.

Motivated by \eqref{diff}, similarly for treating the heat equation with the fractional Laplacian $(-\Delta_v)^s$, it has been expected that any weak solution of the fully nonlinear spatially homogeneous Boltzmann equation belongs to the Gevrey class $\CG^{1/2s}(\mathbb R^3_v)$ at any positive time. The answer was confirmed by Barbaroux-Hundertmark-Ried-Vugalter \cite{MR3665667} for the Maxwell molecule model. Readers may refer to \cite{MR3665667,MR3485915,MR3121714,MR3680949,MR2679746,MR2476686} and references therein for an almost complete list of literature with focus on the smoothness effect for the spatially homogeneous Boltzmann equation. 

In the spatially inhomogeneous case, it is more difficult to treat the regularity problem due to the presence of the transport term so that the equation is degenerate in space variables. Series of works have been done by AMUXY \cite{MR2679369,MR2847536,MR2462585}  in perturbation framework under suitably strong assumptions on initial data in terms of the generalized uncertainty principle and the hypoelliptic regularisation basing on the complex multiplier estimates. Regarding the Gevrey regularity of solutions, inspired by \cite{MR2557895,MR2763329,MR3193940,MR2523694}, it also can be conjectured that any finite-regularity solution of the spatially inhomogeneous Boltzmann equation belongs to the Gevrey class $\CG^{1/2s}(\mathbb R^3_x\times \mathbb R^3_v)$ at any positive time. The conjecture was recently justified by Morimoto-Xu \cite{MR4147427} with $s=1$ for the case of the Landau equation with the Maxwell molecule potentials. For the Boltzmann case, Chen-Hu-Li-Zhan \cite{2018arXiv180512543C} obtained the Gevrey regularity in $\CG^{(1+2s)/2s}(\mathbb R^3_x\times \mathbb R^3_v)$  provided that the initial data belong to $H^{\ell}_m(\mathbb R^3_x\times \mathbb R^3_v)$ with the Sobolev exponent $\ell\geq 6$ and the order of velocity moments $m$ large enough. It is still a problem to prove even the same Gevrey regularity as in \cite{2018arXiv180512543C} for the lower-regularity global solution near global Maxwellians. 

In the end, we also mention extensive studies of the conditional regularity of solutions to the spatially inhomogeneous Boltzmann equation for general initial data in \cite{MR4033752,IS21,MR4229202,MR4049224,MR3551261} by Silvestre together with his collaborators. It would be interesting to develop a self-contained theory of both existence and regularity without any extra condition on solutions to the spatially inhomogeneous non-cutoff Boltzmann equation with initial data allowing to have possibly large oscillations in space variable, for instance, see \cite{MR3634029,2020arXiv201101503D} in the cutoff case where the $L^\infty$ norm can be arbitrarily large but the relative entropy is small enough.

\subsection{Strategy of the proof}
To clarify the argument roughly, we first explain how to use the time-weighted energy method as well as the derivative iteration technique to capture the Gevrey regularity basing on the following linear toy model with diffusions in both velocity and space variables
\begin{eqnarray*}
 \Big(	\partial_t+v\cdot\nabla_x+  (-\Delta_x )^{\frac{2s}{1+2s}}+(-\Delta_v )^ {s} \Big)f=0,\qquad f|_{t=0}=f_0,
\end{eqnarray*}
where the velocity diffusion operator is consistent with that of the linearized Boltzmann operator as in \eqref{coer} and \eqref{rela} while the space diffusion operator is inspired by \cite{MR3950012}.
Let $f$ be a smooth solution to the above Cauchy problem. We then may  perform the energy estimates of $\nabla_x^m f$ for any integer $m\geq 0$.  
Since we are concerned with the regularity of solutions, it is natural to introduce an auxilliary function of $t$  that vanishes at $t=0$ in order to overcome the singularity  of $\nabla_x^m f|_{t=0}$.  Choosing $t\mapsto t^{\frac{1+2s}{2s} m}$ as the desired time weight function, an informal computation gives that 
\begin{eqnarray*}
\begin{aligned}
	& \frac{1}{2}\frac{d}{dt}\Big(t^{\frac{1+2s}{s} m}\norm{\nabla_x^m f}_{L^2}^2\Big)+t^{\frac{1+2s}{ s} m}\norm{(-\Delta_x)^{\frac{s}{1+2s}}\nabla_x^{m} f}_{L^2}^2+t^{\frac{1+2s}{ s} m}\norm{(-\Delta_v)^{\frac{s}{2}}\nabla_x^{m} f}_{L^2}^2\\
	&=  \frac{1+2s}{2s} mt^{\frac{1+2s}{s} m-1}\norm{\nabla_x^m f}_{L^2}^2 \\
	& \leq  \eps  t^{\frac{1+2s}{ s} m}\norm{(-\Delta_x)^{\frac{s}{1+2s}}\nabla_x^{m} f}_{L^2}^2+C_\eps m^{\frac{1+2s}{s}} t^{\frac{1+2s}{s}(m-1)}  \norm{\nabla_x^{m-1} f}_{L^2}^2,
	\end{aligned}
\end{eqnarray*}
for an arbitrary constant $\eps>0$,  
where the last inequality follows from the interpolation inequality. This ensures to conclude  the Gevrey regularity in $x$ by induction on $m$ and then the Gevrey regularity in $v$ in a similar way. Note that one can not expect a  Gevrey index with respect to $v$ variable better than $x$ variable because the spatial derivatives have to be induced when making estimates on the commutator between the transport operator and $\partial_v^m$.  We remark that the rigorous counterpart of the above informal calculation can be achieved by introducing  some kind of  regularization operators  that commute with the diffusion parts.  
   
Back to the nonlinear Boltzmann equation with non-cutoff  potentials, we follow the similar strategy as for the above linear toy model. However, new difficulties arise from the non-trivial treatment of commutators between  the collision part and the regularization operators
   %\begin{eqnarray*}
   	\begin{center}
$(1-\delta\Delta_x)^{-1/2}$ and $(1-\delta\Delta_v)^{-1}$
\end{center} 
for $\delta>0$ suitably small.
  % \end{eqnarray*}  
To overcome these difficulties  we will make use of the symbolic calculus developed in \cite{MR3950012}. The argument here will be more subtle since the regularity iteration procedure begins with the solutions of quite low regularity from the existence theory in \cite{MR4230064} corresponding to \eqref{thm.gest} with $m=|\beta|=0$. To  deal with the commutator between the collision operator and the regularization  $(1-\delta\Delta_x)^{-1/2}$ or its Fourier counterpart $(1+\delta\abs k^2 )^{-1/2}$,  we give in Lemma \ref{est} a new elementary inequality
\begin{equation*} 
\frac{\comi{k}^{m}}{{(1+\delta|k|^2)^{1/2}}} 
\leq \sum_{1\leq j\leq m-1} \begin{pmatrix}
      m    \\
      j  
\end{pmatrix} \comi{k-\ell}^j\comi{\ell}^{m-j}  \\
+  \frac{2\comi{k-\ell}^m}{{(1+\delta|k-\ell|^2)^{1/2}}}  +  \frac{2\comi{\ell}^{m}}{{(1+\delta|\ell|^2)^{1/2}}},
\end{equation*}
which is crucially used for iteration estimates on such low regularity solution.    Moreover, to estimate the commutator 
 between the collision operator and  the regularization operator $(1-\delta\Delta_v)^{-1}$,  one main view is to write  
 \begin{equation*}
 \begin{split}
  (1-\delta\Delta_v) ^{-1} \Gamma( f,\ \partial_{v }^{\beta} f)=  (1-\delta\Delta_v) ^{-1} \Gamma(   f,\  (1-\delta\Delta_v)(1-\delta\Delta_v) ^{-1}  \partial_{v }^{\beta}  f),
 \end{split}
 \end{equation*} 
see \eqref{leib} in the proof of Lemma \ref{lemke}, for instance. 
Such view is useful for dealing with the commutator  by the Leibniz formula without involving any pseudo-differential calculus.

\subsection{Arrangement of the paper} 
The rest of this paper is arranged as follows. In Section  \ref{sec2}, we will list a few preliminary facts that will be used throughout the proof of Theorem \ref{maith}.     Section \ref{sec3} and Section \ref{sec4} are the key parts, devoted to proving the   Gevrey regularity  in space variables and  velocity variables,  respectively. In Section \ref{sec5}, we will complete the proof of Theorem \ref{maith}.   In the appendix Section \ref{secapp} , we will give some  basic facts on  the Weyl and Wick quantizations of  symbol class.  

 \section{Preliminaries}\label{sec2}

We list here a few preliminary facts that will be used throughout the paper.   
In the following discussion we denote by $\comi{D_x}^{\theta}$ for $\theta\in\mathbb R$ the Fourier multiplier in $x$ variable,  that is, 
\begin{eqnarray*}
	\mathscr F_x(\comi{D_x}^{\theta}h)(k)=\comi k^\theta \mathscr F_x h(k).
\end{eqnarray*}
Recall  that $\comi{\cdot}=(1+\abs{\cdot}^2)^{1/2} $ and  $\mathscr F_x  $    stands for the partial Fourier transform   in $x$ variable.   Similarly, letting   $\mathscr F_v  $  be   the partial Fourier transform   in $v$ variable,  
  \begin{eqnarray*}
	\mathscr F_v(\comi{D_v}^{\theta}h)(\eta)=\comi \eta^\theta \mathscr F_v h(\eta),
\end{eqnarray*}
 with $\eta\in\mathbb R^3$ being the Fourier dual variable of $v$.

We first recall  the global symbolic calculus for the linearized Boltzmann operator that was established by \cite{MR3950012}.  Denote by $b^w$ and $b^{\rm Wick}$, respectively,  the Weyl and Wick  quantizations of a symbol $b=b(v,\eta)$, with $\eta$ the Fourier dual variable of $v$.  Note the symbols considered in this work  are  independent of  $(x,k)$ variables and thus the corresponding Weyl or Wick quantizations are pseudo-differential operators  acting only on $v$ variable.     The basic properties of the quantization of symbols  are listed in Appendix \ref{secapp}, and  one may refer  to \cite{MR781536,MR2599384}  for extensive discussions.  We list two  classes of symbols under consideration. One is  
\begin{eqnarray*}
	S\big(1, \abs{dv}^2+\abs{d\eta}^2\big),
\end{eqnarray*} 
and the other is
\begin{eqnarray*}
	S\big(\tilde a, \abs{dv}^2+\abs{d\eta}^2\big).
\end{eqnarray*}
Here and below $\abs{dv}^2+\abs{d\eta}^2$ stands for the flat metric and 
 \begin{equation}\label{defofa}
 \tilde a= \tilde{a}(v,\eta) :=\comi{v}^{\gamma} (1+\abs{v}^{2} + |v\wedge \eta|^{2} +|\eta|^{2}  )^s, \quad (v,\eta) \in \mathbb R^6,
  \end{equation}
with $ \gamma, s$  the numbers given in \eqref{kern} and \eqref{angu}, and $v\wedge \eta$  the cross product, that is,
\begin{eqnarray*}
	v\wedge \eta=(v_2\eta_3-v_3\eta_2, v_3\eta_1-v_1\eta_3,  v_1\eta_2-v_2\eta_1).
\end{eqnarray*} 
Recall that we say $b\in  S\big(\tilde a, \abs{dv}^2+\abs{d\eta}^2\big)$ if 
\begin{eqnarray*}
	|\partial_v^\alpha\partial_\eta^\beta b(v,\eta)|\leq C_{\alpha,\beta} \tilde a(v,\eta),\quad \forall\  \alpha,\beta\in\mathbb Z_+^3,
\end{eqnarray*}
with $C_{\alpha,\beta}$ constants depending on $\alpha$ and $\beta$. Furthermore, by  $b\in  S\big(\tilde a, \abs{dv}^2+\abs{d\eta}^2\big)$ uniformly with respect to a parameter $\tau$, it means that the constants $C_{\alpha,\beta}$ in the above estimate are independent of $\tau$.  Similar things hold for the definition of $ S\big(1, \abs{dv}^2+\abs{d\eta}^2\big)$. An elementary property to be frequently used is the $L^2$ continuity  theorem in the class $S(1,~\abs{dv}^2+\abs{d\eta}^2)$, saying (cf. \cite[Theorem 2.5.1]{MR2599384} for instance) that  if $b\in S(1,~\abs{dv}^2+\abs{d\eta}^2)$ then there exists a constant $C$   
such that 
\begin{equation}\label{bdness}
   \norm{ b^w h}_{L_v^2}\leq C  \norm{h}_{L_v^2},\quad  \forall ~h\in L_v^2.
\end{equation}
Note that if one further assumes $b\in S(1,~\abs{dv}^2+\abs{d\eta}^2)$ uniformly with respect to a parameter $\tau$ then the constant $C$ in  \eqref{bdness} will be independent of $\tau.$
Let us also recall here the composition formula of the Weyl
quantization. Let $M_j, j=1,2,$ be two admissible weights for the flat metric $\abs{dv}^2+\abs{d\eta}^2$, see Section \ref{secapp} for the definition of admissible weights.   If $b_j\in S(M_j, \abs{dv}^2+\abs{d\eta}^2)$ then
$b_1b_2\in S(M_1M_2, \abs{dv}^2+\abs{d\eta}^2)$ and we have the following  composition formula for the Weyl
quantization: 
\begin{equation}\label{symbc}
	b_1^wb_2^w =(b_1b_2)^w+q^w,  \quad q\in S(  M_1 M_2, \ \abs{dv}^2+\abs{d\eta}^2).
\end{equation}
Moreover, if  it additionally holds that $\partial_v^\alpha\partial_\eta^\beta b_j\in S(\tilde M_j, \abs{dv}^2+\abs{d\eta}^2) \subset   S(M_j, \abs{dv}^2+\abs{d\eta}^2)
$ for any $\alpha,\beta\in\mathbb Z_+^3$ with $\abs\alpha+\abs\beta=1$, then the symbol $q$ in \eqref{symbc} satisfies 
\begin{eqnarray*}
  q\in S(\tilde M_1 \tilde M_2,\  \abs{dv}^2+\abs{d\eta}^2).
\end{eqnarray*}
 This yields that the commutator between $b_1^w$ and $b_2^w$, denoted by $[b_1^w, \ b_2^w]$,  is also a Weyl quantization of some symbol, that is, 
\begin{equation}\label{comsym}
	[b_1^w, \ b_2^w]=\tilde q^w,\quad \tilde q\in S(\tilde M_1\tilde M_2,\ \abs{dv}^2+\abs{d\eta}^2),
\end{equation}
provided that $\partial_v^\alpha\partial_\eta^\beta b_j\in S(\tilde M_j, \abs{dv}^2+\abs{d\eta}^2)
$   for any $\alpha,\beta\in\mathbb Z_+^3$ with $\abs\alpha+\abs\beta=1$. 
Recall that the commutator  $[\CT_1, \CT_2]$ between two operators $\CT_1$ and $\CT_2$ is defined by 
\begin{equation*}
	[\CT_1, \CT_2]=\CT_1\CT_2-\CT_2\CT_1.
\end{equation*}
Now we are ready to state the symbolic calculus established by Alexandre-H\'erau-Li \cite{MR3950012}. 

\begin{proposition}[Proposition 1.4 and Lemma 4.3 of  \cite{MR3950012}] \label{estaa} Suppose that the non-cutoff Boltzmann collision kernel satisfies  \eqref{kern} and \eqref{angu}  with $0 < s < 1$  and $\gamma >-3.$    Then   the linearized collision operator  $\mathcal L$  defined by \eqref{linearbol} can be written as 
$$
\mathcal L = -a^w - \mathcal R,
$$ 
  such that the following properties  are fulfilled by $a$ and  $\mathcal R$.  
\begin{enumerate}[(i)]
\item
There exists a  positive constant $C\geq 1$ such that
$$
C^{-1} \tilde{a}(v,\eta) \leq a(v, \eta) \leq  C \tilde{a}(v,\eta),\quad \forall \, (v,\eta)\in\mathbb R^6,
$$ 
with $\tilde a$ defined by \eqref{defofa}. Moreover  $ a \in S( \tilde{a} , \abs{dv}^2+\abs{d\eta}^2).$ 

\item As for the operator $\mathcal R$   we have  for any $\eps>0$,
 \begin{eqnarray*}
\norm{ \mathcal R  h}_{L_v^2}\leq \eps \norm{a^w h}_{L_v^2}+C_\eps \norm{\comi v^{2s+\gamma} h}_{L_v^2},\quad \forall\,  h\in \mathscr S(\mathbb R_v^3),
\end{eqnarray*}
 with   $C_\eps$   a constant depending on $\eps$. Here and below   $\mathscr S(\mathbb R_v^3)$ stands for the  Schwartz space  in $\mathbb R_v^3$.   
  \item  The operators $ a^w$ and $\big(a^{1/2}\big)^w$  are invertible on $L^2_v$ and  their inverses can be respectively written as
\[
    \inner{a^w}^{-1} =H_1 \inner{a^{-1}}^w=\inner{a^{-1}}^wH_2
\]
and 
\[
   \big[\big(a^{1/2}\big)^w\big]^{-1}= G_1
   \big(a^{-1/2}\big)^w= \big(a^{-1/2}\big)^wG_2,
\]
with  $H_j, G_j$  being  bounded operators on $L^2_v$.
   \end{enumerate}
   \end{proposition}

 \begin{lemma}\label{lemcomest}
	Let $b^w$ be the Weyl quantization of   symbol
	\begin{eqnarray*}
		b=b(v,\eta)\in S(1, \abs{dv}^2+\abs{d\eta}^2). 
	\end{eqnarray*}
	Then, for any $h\in L_v^2$ with $(a^{1/2})^wh\in L_v^2$,  it holds that
	\begin{eqnarray*}
		\norm{b^w(a^{1/2})^wh}_{L_v^2}+\norm{(a^{1/2})^wb^wh}_{L_v^2}+\norm{[(a^{1/2})^w,\ b^w]h}_{L_v^2}\leq C\norm{(a^{1/2})^wh}_{L_v^2}.
	\end{eqnarray*}
\end{lemma}
   
\begin{proof}
It follows from \eqref{bdness} that 
\begin{eqnarray*}
\norm{b^w(a^{1/2})^wh}_{L_v^2}\leq C\norm{(a^{1/2})^wh}_{L_v^2}.
\end{eqnarray*}
Observe 
\begin{eqnarray*}
\norm{(a^{1/2})^wb^w h}_{L_v^2}\leq\norm{b^w(a^{1/2})^wh}_{L_v^2}+	\norm{[(a^{1/2})^w,\ b^w]h}_{L_v^2}.
\end{eqnarray*} 
Then it remains to control the commutator term. In view of \eqref{comsym}, we can write $[(a^{1/2})^w,\ b^w]=\tilde q^w$ for some $\tilde q\in S(\tilde a^{1/2}, \abs{dv}^2+\abs{d\eta}^2)$.  Thus, using the composition formula \eqref{symbc} of Weyl quantization and the assertion (iii) in Proposition \ref{estaa}, we conclude that $\tilde q^w \big[(a^{1/2})^w\big]^{-1}$ is bounded on $L_v^2$. Consequently, writing that
\begin{eqnarray*}
[(a^{1/2})^w,\ b^w]=\underbrace{ [(a^{1/2})^w,\ b^w] \big[(a^{1/2})^w\big]^{-1} }_{\textrm{bounded on } L_v^2}(a^{1/2})^w,     	
\end{eqnarray*}
one has
\begin{eqnarray*}
\norm{[(a^{1/2})^w,\ b^w]h}_{L_v^2}\leq C\norm{(a^{1/2})^wh}_{L_v^2}.
\end{eqnarray*}
The proof of Lemma \ref{lemcomest} is thus completed. 
\end{proof}

The second fact is concerned with the trilinear estimate,  which says (cf. \cite[theorem 1.2]{MR3177640} or  \cite[theorem 2.1]{MR2784329}) that there is a constant $C$ such that
\begin{equation}\label{trin}
\big|\big(  \Gamma ( f, g),  \  h\big)_{L^2_v}\big|\leq C\norm{f}_{L_v^2}\norm{(a^{1/2})^wg}_{L_v^2}\norm{(a^{1/2})^wh}_{L_v^2},\quad \forall\, f,g,h\in\mathscr S (\mathbb R_v^3),
\end{equation}
where the equivalence in \eqref{rela} has been used and we also recall that $\mathscr S(\mathbb R_v^3)$ stands for the  Schwartz space  in $\mathbb R_v^3$. Furthermore, we mainly employ the counterpart of the above estimate after performing the partial Fourier transform in $x$ variable.   Precisely, 
recall that the Fourier transform in $x$ variable for the nonlinear term $\Gamma(f,g)$ in \eqref{def.nlt} is given by
\begin{equation}\label{def.GaF}
\hat{\Gamma}(\hat{f},\hat{g})(k,v)
=\int_{\mathbb{R}^3}\int_{\mathbb{S}^{2}} B(v-u,\sigma) [\mu(v)\mu(u)]^{1/2}\left([\hat{f}(u')*\hat{g}(v')](k)-[\hat{f}(u)*\hat{g}(v)](k)\right) d \sigma d u,
\end{equation}
where the convolutions are taken with respect to the Fourier variable $k\in\mathbb Z^3$:
\begin{equation*}
%\begin{aligned}
%{[\hat{f}(u')*\hat{g}(v')]}(k) :& =  \int_{\mathbb Z^3_\ell} \hat{f}(u',k-\ell)\hat{g}(v',\ell)\,d\Sigma(\ell),\\
{[\hat{f}(u)*\hat{g}(v)]}(k):  =  \int_{\mathbb Z^3_\ell}  \hat{f}(k-\ell,u)\hat{g}(\ell,v)\,d\Sigma(\ell),
%\end{aligned}
\end{equation*}
for any velocities $u,v\in {\mathbb R}^3$. Then, by \cite[Lemma 3.2]{MR4230064}),  the following estimate   
 \begin{equation*}
  	\big|\big( \hat \Gamma{(\hat f(k), \hat g(k))}, \hat h(k)\big)_{L^2_v}\big |\leq C\int_{\mathbb Z^3 } \norm{\hat f(k-\ell)}_{L^2_v} |\hat g(\ell)|_{D}|\hat h(k)|_D \,d\Sigma(\ell),
  \end{equation*}
  or equivalently
  \begin{equation}\label{orup}
\big|\big( \hat \Gamma{(\hat f(k), \hat g(k))}, \hat h(k)\big)_{L^2_v}\big|\leq  C \int_{\mathbb Z^3 }\norm{
	\hat f(k-\ell)}_{L^{2}_v}\norm{ (a^{1/2})^w \hat g(\ell)}_{L^{2}_v}\norm{(a^{1/2})^w \hat h(k)}_{L^{2}_v}\,d\Sigma(\ell),
\end{equation}
holds true for any $k\in\mathbb Z^3$ and for   any  $f,g,h\in L_k^1(\mathscr S(\mathbb R_v^3))$.

The following lemma and corollary will be used in the proof later on.
 
\begin{lemma} \label{lemtrip}
For any $f,g\in  L_v^2$ with $(a^{1/2})^wg  \in L_v^2$, it holds that $(a^{-1/2})^w   \Gamma ( f, g)\in L_v^2$ with
\begin{eqnarray*}
\norm{(a^{-1/2})^w   \Gamma ( f, g)}_{L_v^2}\leq C    \norm{
	f}_{L^{2}_v}\norm{ (a^{1/2})^w g}_{L^{2}_v}.
\end{eqnarray*}
Moreover, for any  $f,g\in L_k^1L_v^2$ such that $(a^{1/2})^wg\in L_k^1L_v^2 $, it holds that 
  \begin{eqnarray*}
  	\norm{(a^{-1/2})^w \hat \Gamma{(\hat f(k), \hat g(k))}}_{L_v^2}\leq C   \int_{\mathbb Z^3 }\norm{
	\hat f(k-\ell)}_{L^{2}_v}\norm{ (a^{1/2})^w \hat g(\ell)}_{L^{2}_v} d\Sigma(\ell),
  \end{eqnarray*} 
for any $k\in\mathbb Z^3$, where $\hat \Gamma{(\hat f(k), \hat g(k))}$ given in \eqref{def.GaF} stands for the Fourier transform of $\Gamma(f,g)$  in space variables. 
  \end{lemma}

Before giving the proof of Lemma \ref{lemtrip}, we notice that $\mathcal Lh=\Gamma(\sqrt\mu, h)+\Gamma(h, \sqrt\mu)$. Then, as an immediate consequence of the first assertion in  Lemma \ref{lemtrip}, we have the following 

\begin{corollary}\label{corupp}
For any $h\in  L_v^2$ such that $(a^{1/2})^wh \in L_v^2$, it holds that
  	    \begin{eqnarray*}
    			\norm{(a^{-1/2})^w   \mathcal Lh}_{L_v^2}\leq C  \norm{ (a^{1/2})^w h}_{L^{2}_v}.
    	\end{eqnarray*}
  \end{corollary}

  \begin{proof}[Proof of Lemma \ref{lemtrip}]
  
(a) We first claim that $(a^{-1/2})^w  \Gamma (f, g)\in L^2_v$  for any $f,g \in   \mathscr S(\mathbb R_v^3)$ with the estimate 
   \begin{equation}\label{upb}
 \norm{(a^{-1/2})^w   \Gamma ( f, g)}_{L_v^2}\leq C    \norm{
	f}_{L^{2}_v}\norm{ (a^{1/2})^w g}_{L^{2}_v},\quad  \forall\, 	f,g \in   \mathscr S(\mathbb R_v^3).
   \end{equation}
In fact, observe that $(a^{-1/2})^w$ is self-adjoint on $L_v^2$. Then, for any $h\in\mathscr S(\mathbb R_v^3)$, it follows from \eqref{trin} that
  \begin{multline*}
   \big|\big( (a^{-1/2})^w   \Gamma ( f, g),  \ h\big)_{L^2_v}\big| = \big|\big(    \Gamma ( f, g),  \ (a^{-1/2})^w h\big)_{L^2_v}\big| \\
  	 \leq   C  \norm{
	f}_{L^{2}_v}\norm{ (a^{1/2})^w g}_{L^{2}_v}   \norm{(a^{1/2})^w(a^{-1/2})^w h }_{L^{2}_v} 
  	 \leq   C    \norm{
	f}_{L^{2}_v}\norm{ (a^{1/2})^w g}_{L^{2}_v}\norm{  h }_{L^{2}_v},
	\end{multline*} 
where we have used \eqref{symbc} and \eqref{bdness} as well as Proposition \ref{estaa} in the last inequality.  This together with the fact that the Schwartz space $\mathscr S(\mathbb R_v^3)$ is dense in $L_v^2$ give the desired claim.  
  
 (b) We then consider the case of  $f\in   L_v^2$ and $ g\in\mathscr S(\mathbb R_v^3)$. In fact, using again the fact that $\mathscr S(\mathbb R_v^3)$ is dense in $L_v^2$, we can find a sequence of smooth functions $f_n\in  \mathscr S(\mathbb R_v^3)$ such that   $\norm{f_n-f}_{L^{2}_v}  \rightarrow 0$ as $n\rightarrow+\infty.$   
 This together with  \eqref{upb} imply that
 %\begin{eqnarray*}
 $(a^{-1/2})^w   \Gamma ( f_n, g)$  
 %\end{eqnarray*}
 is a Cauchy sequence in $L_v^2$ with a limit denoted as $\mathfrak m\in L_v^2$.  Next, we show that
 \begin{equation}\label{eq}
 \mathfrak m=	 (a^{-1/2})^w   \Gamma ( f, g) \textrm{ in } L_v^2.
 \end{equation}
Indeed, since it holds that $ (a^{-1/2})^w   \Gamma ( f_n, g)\rightarrow \mathfrak m$ and $f_n\rightarrow f$ in $L_v^2$-norm, we are able to extract  a subsequence $\{f_{n_j}\}_{j\geq 1}$ of $f_n$,  such that $ (a^{-1/2})^w   \Gamma ( f_{n_j}, g)\rightarrow \mathfrak m$ and 
 $f_{n_j}\rightarrow f$ pointwise a.e.  in $\mathbb R_v^3$.   As a result, by the Dominated Convergence Theorem, one has
 \begin{eqnarray*}
 	 (a^{-1/2})^w   \Gamma ( f_{n_j}, g)\rightarrow   (a^{-1/2})^w   \Gamma ( f, g) \textrm{ pointwise a.e. in }\mathbb R_v^3, 
 \end{eqnarray*}
which proves \eqref{eq}. Thus, $ (a^{-1/2})^w   \Gamma ( f, g)$ is the limit of  $ (a^{-1/2})^w   \Gamma ( f_n, g)   $ in $L_v^2$. Moreover, it follows from \eqref{upb} that 
 \begin{equation}\label{upb-low}
   		\norm{(a^{-1/2})^w   \Gamma ( f, g)}_{L_v^2}\leq C    \norm{
	f}_{L^{2}_v}\norm{ (a^{1/2})^w g}_{L^{2}_v},
   \end{equation}
for any $f\in   L_v^2$ and $ g\in\mathscr S(\mathbb R_v^3)$.
    
   (c) Finally we suppose   $f\in   L_v^2$ and $ (a^{1/2})^w g\in L_v^2$.   Then, using the density argument again, we can find a sequence of functions $G_n\in  \mathscr S(\mathbb R_v^3)$ such that   $\norm{G_n-(a^{1/2})^w g}_{L^{2}_v}  \rightarrow 0$ as $n\rightarrow+\infty.$ 
   Define 
   \begin{eqnarray*}
   	g_n:= \big[ (a^{1/2})^w \big]^{-1} G_n\in \mathscr S(\mathbb R_v^3),\quad n\geq 1.
   \end{eqnarray*}
   Then, by virtue of \eqref{symbc} and \eqref{bdness}, it follows from \eqref{upb-low} that   
\begin{eqnarray*}
   \begin{aligned}
   		\norm{(a^{-1/2})^w   \Gamma ( f, g_n-g_m)}_{L_v^2} \leq C   \norm{
	f}_{L^{2}_v}\norm{ (a^{1/2})^w(g_n-g_m)}_{L^{2}_v}\leq  C\norm{
	f}_{L^{2}_v}\norm{ G_n-G_m}_{L^{2}_v}.
	\end{aligned}
   \end{eqnarray*}
     This shows that $(a^{-1/2})^w   \Gamma ( f, g_n)$ is a Cauchy sequence in $L_v^2$ with $(a^{-1/2})^w   \Gamma ( f, g)$ as its limit in $L_v^2$ by using the same argument as above along with the fact that 
     \begin{eqnarray*}
     	\norm{g_n-g}_{L^{2}_v}=\big\|\big[ (a^{1/2})^w \big]^{-1}\big(G_n-(a^{1/2})^wg\big)\big\|_{L^{2}_v}\leq C\norm{G_n-(a^{1/2})^w g}_{L^{2}_v}  \rightarrow 0.
     \end{eqnarray*} 
     Moreover, it holds that 
   \begin{eqnarray*}
  		\norm{(a^{-1/2})^w   \Gamma ( f, g)}_{L_v^2}\leq C    \norm{
	f}_{L^{2}_v}\norm{ (a^{1/2})^w g}_{L^{2}_v}.
	   \end{eqnarray*} 
This has proved the first assertion in Lemma \ref{lemtrip}. The second one  for the counterpart after performing the Fourier transform in space variables can be treated in the same way via \eqref{orup} instead of \eqref{trin}.
 The proof of Lemma \ref{lemtrip} is thus completed.
 \end{proof}

 The following technical   lemma will be frequently used in treating estimates on $\Gamma(g,h)$.
 
 \begin{lemma}\label{lemtl} 
For an arbitrarily given integer $j_0\geq 1$, it holds that
 		 \begin{multline}
 \int_{\mathbb Z^3}\bigg[	\int_0^T \bigg(\int_{\mathbb Z^3 } \sum_{1\leq j\leq j_0} \norm{
	\hat f_j(t,k-\ell)}_{L^{2}_v}\norm{ (a^{1/2})^w \hat g_j(t,\ell)}_{L^{2}_v} d\Sigma(\ell)\bigg)^2dt\bigg]^{1/2}d\Sigma(k)\\
 \leq \sum_{1\leq j\leq j_0}\Big(\int_{\mathbb Z^3} 	 \sup_{0<t<T}\norm{
	\hat f_j(t,k)}_{L^{2}_v}  d\Sigma(k) \Big)\int_{\mathbb Z^3 } \Big(\int_0^T   \norm{ (a^{1/2})^w \hat g_j(t,k)}_{L^{2}_v}^2 dt\Big)^{1/2}d\Sigma(k),
\label{lemtlde}
	\end{multline}
for any 
 	$f_j\in L_k^1L_T^\infty L_v^2$ and any $g_j $ such that  $(a^{1/2})^wg_j\in L_k^1L_T^2L_v^2$ with $1\leq j\leq j_0$. 
 \end{lemma}
 
 \begin{proof}
Using the triangle inequality that $\norm{\sum_{j=1}^{j_0} A_j}_{L^2(0,T)}\leq \sum_{j=1}^{j_0} \norm{A_j}_{L^2(0,T)}$ for a sequence of functions $A_j$ in $L^2(0,T)$, it suffices to prove that the desired estimate \eqref{lemtlde} is satisfied for $j_0=1$. We then write $f$ for $f_1$ and likewise $g$ for $g_1$. Direct computations give that   
 %\begin{equation}
 \begin{align}
  &\int_{\mathbb Z^3}\bigg(	\int_0^T    \Big(\int_{\mathbb Z^3 }\norm{
	\hat f (t,k-\ell)}_{L^{2}_v}\norm{ (a^{1/2})^w \hat g(t,\ell)}_{L^{2}_v} d\Sigma(\ell)\Big)^2dt\bigg)^{1/2}d\Sigma(k)\notag\\
  &\leq \int_{\mathbb Z^3} 	\bigg[\int_{\mathbb Z^3 }\Big(\int_0^T   \norm{
	\hat f(t,k-\ell)}_{L^{2}_v}^2\norm{ (a^{1/2})^w \hat g(t,\ell)}_{L^{2}_v}^2 dt\Big)^{1/2}d\Sigma(\ell) \bigg]d\Sigma(k)\notag\\
  &\leq \int_{\mathbb Z^3} 	\bigg[\int_{\mathbb Z^3 }\sup_{0<t<T}\norm{
	\hat f(t,k-\ell)}_{L^{2}_v} \Big(\int_0^T   \norm{ (a^{1/2})^w \hat g(t,\ell)}_{L^{2}_v}^2 dt\Big)^{1/2}d\Sigma(\ell) \bigg]d\Sigma(k)\notag\\
  &= \Big(\int_{\mathbb Z^3} 	 \sup_{0<t<T}\norm{
	\hat f(t,k)}_{L^{2}_v}  d\Sigma(k) \Big)\int_{\mathbb Z^3 } \Big(\int_0^T   \norm{ (a^{1/2})^w \hat g(t,k)}_{L^{2}_v}^2 dt\Big)^{1/2}d\Sigma(k),
\label{dica}
\end{align}
%\end{equation}
     where we have used 
  	Minkowski's inequality and Fubini's theorem in the first and  last inequalities, respectively.  The proof of Lemma \ref{lemtl} is completed.   
  	   \end{proof}
  
The following Lemma gives an elementary inequality that will be essentially adopted to treat the iterative estimates for obtaining the Gevrey regularity in space variables.

\begin{lemma}\label{est}
There is a generic constant $C>0$ such that for any $m\geq 1$  the following estimate  
\begin{equation}\label{lem.ei} 
  \frac{\comi{k}^{m}}{{(1+\delta|k|^2)^{1/2}}} 
  \leq \sum_{1\leq j\leq m-1} \begin{pmatrix}
      m    \\
      j 
\end{pmatrix} \comi{k-\ell}^j\comi{\ell}^{m-j}   
  +  \frac{2\comi{k-\ell}^m}{{(1+\delta|k-\ell|^2)^{1/2}}}  +  \frac{2\comi{\ell}^{m}}{{(1+\delta|\ell|^2)^{1/2}}} 
\end{equation}
holds for any $k, \ell  \in \mathbb{Z}^3$ and any $0<\delta <1$, with the convention that the summation term over $1\leq j\leq m-1$ on the right hand side disappears when  $m=1$.  
\end{lemma} 

\begin{proof}
First note that the function $|k|\mapsto \langle k\rangle^m/(1+\delta |k|^2)^{1/2}$ is nondecreasing in $|k|$ when $0<\delta<1$ and $m\geq 1$. Then, in case of  $|k|\leq |\ell|$, it is direct to see
\begin{equation*}
\frac{\comi{k}^{m}}{{(1+\delta|k|^2)^{1/2}}}\leq 
\frac{\comi{\ell}^{m}}{{(1+\delta|\ell|^2)^{1/2}}},
 \end{equation*}
so \eqref{lem.ei} holds true.  In case of $|k|> |\ell|$, using $\comi{k} \leq \comi{k-\ell}+\comi{\ell}$, it follows that
\begin{equation*}
%\label{ }
\comi{k}^m\leq (\comi{k-\ell}+\comi{\ell})^m=\sum_{0\leq j\leq m} \begin{pmatrix}
      m    \\
      j  
\end{pmatrix} \comi{k-\ell}^j\comi{\ell}^{m-j},
\end{equation*}
so one has
\begin{equation*} 
  	 \begin{split}
  	 \frac{\comi{k}^{m}}{{(1+\delta|k|^2)^{1/2}}} 
  	 \leq& \sum_{1\leq j\leq m-1} \begin{pmatrix}
      m    \\
      j 
\end{pmatrix} \frac{\comi{k-\ell}^j\comi{\ell}^{m-j}}{(1+\delta|k|^2)^{1/2}}   +   \frac{\comi{k-\ell}^m}{{(1+\delta|k |^2)^{1/2}}}  +   \frac{\comi{\ell}^{m}}{{(1+\delta|k|^2)^{1/2}}}.
  	 \end{split}
  	 \end{equation*}
Then, to show \eqref{lem.ei} it suffices to verify that  
\begin{equation}
\label{lem.eip1}
\frac{\comi{k-\ell}^m}{{(1+\delta|k |^2)^{1/2}}}  +   \frac{\comi{\ell}^{m}}{{(1+\delta|k|^2)^{1/2}}}\leq  \frac{2\comi{k-\ell}^m}{{(1+\delta|k-\ell|^2)^{1/2}}}  +  \frac{2\comi{\ell}^{m}}{{(1+\delta|\ell|^2)^{1/2}}}
\end{equation}
for any $k$ and $\ell$ with $|k|> |\ell|$. Indeed, we consider two cases $|k|\geq 2|\ell|$ and $ |\ell| \leq |k| \leq 2|\ell|$ as follows. For  $|k|\geq 2|\ell|$,
it holds that \begin{equation*}
  	  \frac{\comi{\ell}^{m}}{{(1+\delta|k|^2)^{1/2}}}\leq  \frac{\comi{\ell}^{m}}{{(1+4\delta|\ell|^2)^{1/2}}} \leq  \frac{\comi{\ell}^{m}}{{(1+\delta|\ell|^2)^{1/2}}},
\end{equation*}
and since $|k|\geq 2|\ell|$ implies $|k-\ell| \leq |k|+|\ell| \leq |k|+\frac{1}{2}|k|= \frac32 |k|$ as well, similarly one has  
\begin{equation*}
  	 \frac{\comi{k-\ell}^m}{{(1+\delta|k |^2)^{1/2}}}\leq \frac{\comi{k-\ell}^m}{{(1+\frac{4}{9}\delta|k-\ell |^2)^{1/2}}}  \leq  \frac{\frac{3}{2}\comi{k-\ell}^m}{{(1+\delta|k-\ell |^2)^{1/2}}},
\end{equation*}
so \eqref{lem.eip1} is satisfied in case of $|k|\geq 2|\ell|$.  Similarly, for $ |\ell| \leq |k| \leq 2|\ell|$ that gives $|k-\ell|\leq |k|+|\ell|\leq 2|k|$, one has
  	  \begin{equation*}
  	  \frac{\comi{\ell}^{m}}{{(1+\delta|k|^2)^{1/2}}} \leq \frac{\comi{\ell}^{m}}{{(1+\delta|\ell|^2)^{1/2}}},\quad
  	  \frac{\comi{k-\ell}^m}{{(1+\delta|k |^2)^{1/2}}}  \leq  \frac{2 \comi{k-\ell}^m}{{(1+\delta|k-\ell |^2)^{1/2}}},
  	  \end{equation*}
that yield \eqref{lem.eip1}  as well. Therefore this shows \eqref{lem.eip1} and completes the proof of Lemma \ref{est}.   
\end{proof}

\section{Gevrey  smoothing effect in spatial variable}\label{sec3}

In this section we start to study the nonlinear Cauchy problem on the reformulated equation \eqref{eqforper} supplemented with $f|_{t=0}=f_0$. To the end, for convenience we always assume $\gamma\geq 0$ and $0 < s < 1$ for the collision kernel \eqref{kern} and \eqref{angu}. First of all, we state the existence result established in \cite{MR4230064}.

\begin{proposition}\label{prop.exis}
There are $\varepsilon_0>0$ and $C_0>0$ such that if the initial datum $f_0$ is chosen such that $F_0(x,v)=\mu+\sqrt{\mu}f_0(x,v)\geq 0$  with $ f_0\in L^1_kL^2_v$  satisfying  \eqref{con.cl} and 
%\begin{equation}
 $\norm{f_0}_{  L^1_kL^2_v} \leq  \varepsilon_0$,
%\end{equation}
then  the Cauchy problem on the non-cutoff Boltzmann equation \eqref{eqforper} with $f|_{t=0}=f_0$ admits a unique global mild solution $f(t,x,v)$ such that $F(t,x,v)=\mu+\sqrt{\mu}f(t,x,v)\geq 0$ with $f\in L^1_kL^\infty_TL^2_v$ for any $T>0$ satisfying the estimate  
\begin{equation}\label{+1234}	
 	  \norm{f }_{L_k^1L^\infty_TL_v^2}+   \norm{(a^{1/2})^w  f}_{L_k^1L^2_TL_v^2} \leq 
 	C_0  \norm{f_0}_{  L^1_kL^2_v}.
\end{equation} 
\end{proposition} 

The main goal of this section is to further prove the Gevrey smoothness in space variable $x$ for the obtained solution $f(t,x,v)$.
 
\begin{theorem}\label{thm3.1}
Let $\varepsilon_0>0$ be further small, then there is a  constant $\tilde C_0>0$,  depending only on $s,\gamma,\eps_0$ and the constant $C_0$ above, such that  for any $0<T<\infty$  and any integer $m\in\mathbb Z_+$, the solution $f(t,x,v)$ obtained in Proposition \ref{prop.exis} satisfies 
\begin{equation}
\label{thm3.1Reg}
\nabla_x^m f\in L^1_kL^\infty_{\tau,T}L^2_v,\quad \nabla_x^{m+\frac{s}{1+2s}}f,\ (a^{1/2})^w\nabla_x^mf\in L^1_kL^2_{\tau,T}L^2_v
\end{equation}
for any small $\tau>0$, with the quantitative estimate
%\begin{eqnarray*}
\begin{align*}
    &\int_{{\mathbb Z}^3}  \comi{k}^{m}\Big( \sup_{0<t< T} \phi(t)^{ \varsigma m }  \norm{ {\hat f (t,k )}}_{L^2_v}\Big) d \Sigma(k)+\int_{{\mathbb Z}^3}  \comi{k}^{m+\frac{s}{1+2s}}\Big( \int_0^T \phi(t)^{ 2\varsigma m }  \norm{ \hat f (t,k )}_{L^2_v}^2dt\Big)^{\frac{1}{2}} d \Sigma(k)  \\
    & \quad +\int_{\mathbb Z^3}  \comi{k}^{m}\Big( \int_0^T \phi(t)^{ 2\varsigma m }  \norm{(a^{1/2})^w {\hat f (t,k )}}_{L^2_v}^2dt\Big)^{1/2} d \Sigma(k)  
  \leq    {\tilde C}_0^{m+1} (m !) ^{\frac{1+2s}{2s}}.
  %\norm{f_0}_{ L^1_kL^2_v}.
\end{align*}
%\end{eqnarray*}
Here and below we have denoted $\phi(t)=\min\{t,  1\}$ and
 %\begin{equation}
	$\varsigma=\frac{1+2s}{2s}$  
%\end{equation}
 with $s$ the parameter given in \eqref{angu}.
\end{theorem}

Theorem  \ref{thm3.1} is just an immediate consequence of the  following two  propositions, by using induction on $m$.

\begin{proposition}[Initial step for $m=0,1$]\label{prplow}
Under the same assumption as in Theorem \ref{thm3.1}, there is  a constant $C_1$, depending only on $s,\gamma$ and the number  $  C_0$ in \eqref{+1234},  such that \eqref{thm3.1Reg} holds true for $m=0,1$ and the following estimate
%\begin{eqnarray*}
\begin{align*}
	  	 &\int_{\mathbb Z^3}  \comi{k}^{m}\Big( \sup_{0<t<T} \phi(t)^{ \varsigma m }  \norm{ {\hat f (t,k )}}_{L^2_v}\Big) d \Sigma(k) +  \int_{\mathbb Z^3}  \comi{k}^{m+\frac{s}{1+2s}}\Big( \int_0^T \phi(t)^{ 2\varsigma m }  \norm{ \hat f (t,k )}_{L^2_v}^2dt\Big)^{1/2} d \Sigma(k)\\
 &\qquad\qquad+ \int_{\mathbb Z^3}  \comi{k}^{m}\Big( \int_0^T \phi(t)^{ 2\varsigma m }  \norm{(a^{1/2})^w {\hat f (t,k )}}_{L^2_v}^2dt\Big)^{1/2} d \Sigma(k) \leq C_1 \norm{f_0}_{  L^1_kL^2_v},
\end{align*}
%\end{eqnarray*}
is satisfied for $m=0,1$.  
    %$m\in\mathbb Z_+$ with $0\leq m\leq 1.$
  \end{proposition}

\begin{proposition}[Inductive regularity]\label{prpind}
Let $f(t,x,v)$ satisfy the same conditions as in Theorem \ref{thm3.1} and let $C_1$ be the constant constructed in the previous Proposition \ref{prplow}. Then there is a constant $\tilde C_0\geq C_1$, depending only on $s,\gamma$ and the number  $ C_0$ in \eqref{+1234}, such that for an  integer $m\geq 2$,  we have \eqref{thm3.1Reg} as well as 
%\begin{eqnarray*}
\begin{align*}
  &\int_{\mathbb Z^3}  \comi{k}^{m}\Big( \sup_{0<t<T} \phi(t)^{ \varsigma m }  \norm{ {\hat f (t,k )}}_{L^2_v}\Big) d \Sigma(k)  + \int_{\mathbb Z^3}  \comi{k}^{m+\frac{s}{1+2s}}\Big( \int_0^T \phi(t)^{ 2\varsigma m }  \norm{ \hat f (t,k )}_{L^2_v}^2dt\Big)^{\frac{1}{2}} d \Sigma(k)   \\
 &\qquad+ \int_{\mathbb Z^3}  \comi{k}^{m}\Big( \int_0^T \phi(t)^{ 2\varsigma m }  \norm{(a^{1/2})^w {\hat f (t,k )}}_{L^2_v}^2dt\Big)^{1/2} d \Sigma(k)   
  \leq  \tilde C_0^{m-1}\big[(m-1)!\big]^{\frac{1+2s}{2s}},
  %\norm{f_0}_{  L^1_kL^2_v}, 
\end{align*}
%\end{eqnarray*} 
provided that \eqref{thm3.1Reg} together with the following estimate 
 %\begin{equation}\label{inassump}
 \begin{align}
	  	 &\int_{\mathbb Z^3}  \comi{k}^{n}\Big( \sup_{0<t<T} \phi(t)^{ \varsigma n }  \norm{ {\hat f (t,k )}}_{L^2_v}\Big) d \Sigma(k)+  \int_{\mathbb Z^3}  \comi{k}^{n+\frac{s}{1+2s}}\Big( \int_0^T \phi(t)^{ 2\varsigma n }  \norm{ \hat f (t,k )}_{L^2_v}^2dt\Big)^{\frac{1}{2}} d \Sigma(k) \notag\\
 &\qquad\qquad\qquad+ \int_{\mathbb Z^3}  \comi{k}^{n}\Big( \int_0^T \phi(t)^{ 2\varsigma n }  \norm{(a^{1/2})^w {\hat f (t,k )}}_{L^2_v}^2dt\Big)^{1/2} d \Sigma(k)\notag\\
 &\leq 
 \left\{
 \begin{aligned}
&  \tilde C_0, \quad  {\rm if }\  n\leq 1,\\
& \tilde C_0^{n-1}\big[(n-1)!\big]^{\frac{1+2s}{2s}},  \quad   {\rm if } \ n\geq 2,
 \end{aligned}
 	\right.\label{inassump}
	  	\end{align}
%\end{equation}
hold true for any integer $n$ with $0\leq n \leq m-1$.
	 	
\end{proposition}

The rest part of this section is devoted to proving  Propositions \ref{prplow} and \ref{prpind}.   We remark that it suffices to consider only the case of $0<t\leq 1$,  where  the main difficulty is to control the terms involving  the large factor $t^{-1}$ that arises from the auxilliary function $\phi(t)$ introduced to overcome the singularity of  $\comi k^m \hat f(t,k)$ at $t=0$.  Once the regularity   is achieved for $0<t\leq 1$, the counterpart over $1\leq t< T$  can be treated in a similar way as in the case of   $0<t\leq 1$ but with the simpler argument,  since it is essentially the propagation of  regularity from  $t=1$ to $1\leq t< T$.  As to be seen below, we will combine the subelliptic estimates with energy estimates to deal with the large factor $t^{-1}$ for $0<t\leq 1$.
  
As discussed above    we will focus on in the following discussion the case of $0<t\leq 1$ and hence we choose $\phi(t)=t.$  For simplicity we will use  the capital letter $C$ to denote some generic constants, that may vary from line to line and depend only on $\gamma,s$ and the number $C_0$ in \eqref{+1234}, and moreover use  $C_\eps$ to denote  some generic constants depending on a given number $0<\eps\ll 1$ additionally.   Note these generic  constants $C$ and $C_\eps$ as below  are  independent of  the derivative order related to $m$.

\subsection{Regularization operators and  uniform estimates} 
 Note that we can  not directly perform  estimates for $\comi k^m t^{\varsigma m}\hat f(t,k,v)$ due to its low regularity.  So to begin with we  introduce its  regularization defined by
	\begin{eqnarray}\label{fmdel}
		   \hat f_{m,\delta}(t,k,v)=t^{\varsigma m}  (1+\delta\abs k^2)^{-1/2} \comi{k}^m \hat f(t,k,v), 	\quad 0<\delta\ll 1,\quad m\geq 0.
	\end{eqnarray}
 Moreover, with each $k\in\mathbb Z^3$  
  we associate an operator  
   \begin{eqnarray}\label{regoper}
\Lambda_{\delta_1}=\Lambda_{\delta_1,k}= \big  (1+\delta_1  \abs v^2 \big)^{-1-\gamma} \big(1+\delta_1 \abs k^2-\delta_1\Delta_v\big)^{-1}, \quad \delta_1\ll 1,
  \end{eqnarray} 
with $\gamma$ given in  \eqref{kern}. Here and below we have written $\Lambda_{\delta_1,k}$ as $\Lambda_{\delta_1}$ to omit the dependence on $k$ for brevity.  
  
 If $m\geq 2$, then
 by the induction assumption \eqref{inassump} we see 
 $$
 \sup_{0<t\leq 1}t^{\varsigma (m-1)}  \comi{k}^{m-1}\norm{\hat f(t,k)}_{L_v^2}\in L_k^1.
$$
Note that $L^1_k$ is the set of absolutely convergent series. This yields 
  \begin{eqnarray*}
  \sup_{k\in\mathbb Z^3}\Big(	\sup_{0<t\leq 1}t^{\varsigma (m-1)}  \comi{k}^{m-1}\norm{\hat f(t,k)}_{L_v^2} \Big) \leq C<+\infty.
    \end{eqnarray*}
 As a result, it holds that
   \begin{equation}\label{vans}
 \norm{\hat f_{m,\delta}(t,k)}_{L_v^2} \leq C_\delta t^{\varsigma}  t^{\varsigma (m-1)}  \comi{k}^{m-1}\norm{\hat f(t,k)}_{L_v^2}\leq   C_\delta t^{\varsigma},\quad \forall\ k \in\mathbb Z^3,\ \forall\ 0< t\leq 1,
  \end{equation}
for some constant  $C_\delta$    depending  on $\delta$.   Similarly, one has 
   \begin{equation}\label{L2conv}
  \Big(\int_0^1 \norm{(a^{1/2})^w\hat f_{m,\delta}(t,k)}_{L_v^2}^2dt \Big)^{\frac{1}{2}}\leq     C_\delta,\quad \forall\  k\in\mathbb Z^3.
  \end{equation} 
Note that the assertions \eqref{vans} and \eqref{L2conv} are also true for $0\leq m\leq 1$ due to the condition \eqref{+1234}.

 Direct computation shows  
  \begin{equation*}
  \begin{aligned}
\left\{
  \begin{aligned}
  	&\big|\partial_{v}^\alpha   \big  (1+\delta_1  \abs v^2 \big)^{-1-\gamma}  \big|  \leq C_{\alpha}  \big  (1+\delta_1  \abs v^2 \big)^{-1-\gamma},\\  	 
  &	\big| \partial_{\eta}^\beta   \big(1+\delta_1 \abs k^2+\delta_1\abs\eta^2\big)^{-1} \big|  \leq C_{\beta}   \big(1+ \delta_1\abs\eta^2\big)^{-1},
  	\end{aligned}
  	\right.\qquad\forall\ \alpha,\beta\in\mathbb Z_+^3,
  	\end{aligned}
  \end{equation*}
   and moreover
   \begin{equation*} 
 \big|\partial_{v}^\alpha   \big  (1+\delta_1  \abs v^2 \big)^{-1-\gamma}  \big|  + \big| \partial_{\eta}^\alpha   \big(1+\delta_1 \abs k^2+\delta_1\abs\eta^2\big)^{-1} \big| \leq C_{\alpha}  \delta_1^{1/2}, 
  \quad	\forall\ \alpha\in\mathbb Z_+^3 \textrm{ with } \abs\alpha\geq 1,
  \end{equation*} 
 where  $C_{\alpha},C_{\beta}$ are   constants  depending only on $\alpha$ and $\beta$ respectively  but not on $k,\delta_1$.   
  Thus, by \eqref{symbc} we can write $\Lambda_{\delta_1}$ defined by \eqref{regoper} as 
 \begin{equation}\label{ew+}
 	\Lambda_{\delta_1}=q^w   
 \end{equation}
with  $q\in S\big(   (1+\delta_1  \abs v^2  )^{-1-\gamma} (1+ \delta_1\abs\eta^2 )^{-1}, \   \abs{dv}^2+\abs{d\eta}^2\big)$, 
and moreover for any $\alpha,\beta\in\mathbb Z_+^3$ with $\abs\alpha+\abs\beta=1$ we have
\begin{equation}\label{ew}
   q \in  S(1, \ \abs{dv}^2+\abs{d\eta}^2) \textrm{ and }\delta_1^{-1/2}\partial_v^\alpha\partial_\eta^\beta q \in  S(1, \ \abs{dv}^2+\abs{d\eta}^2)   \textrm{ uniformly w.r.t. }  \delta_1 \textrm{ and  } k.
 \end{equation}
 This with assertions (i) and (iii) in Proposition \ref{estaa} as well as  \eqref{symbc} give that
 \begin{equation}
 	\label{bd}
	\norm{(a^{1/2})^w\Lambda_{\delta_1} h}_{L_v}+\norm{\Lambda_{\delta_1} (a^{1/2})^w h}_{L_v}
 \leq C_{\delta_1}\norm{h}_{L_v^2},\quad \forall\, h\in L_v^2.
 \end{equation}
Now  we list some uniform estimates for the regularization operator $\Lambda_{\delta_1}$ to  be used frequently later.  By uniform it means that the estimates presented below hold with constants independent of $\delta_1$ and $k$.  It is clear to see that
 \begin{equation}\label{ubdo}
 	\norm{\Lambda_{\delta_1}   h}_{L_v^2}\leq \norm{h}_{L_v^2},\quad \forall\, h\in L_v^2.
 \end{equation}
Moreover, combining \eqref{ew+} with \eqref{ew}, we apply   Lemma \ref{lemcomest}   to conclude that for any $h\in  L_v^2$ 
with  $(a^{1/2})^w h \in L_v^2$, 
  \begin{equation*}%\label{ce}
		\norm{\Lambda_{\delta_1} (a^{1/2})^wh}_{L_v^2}+\norm{(a^{1/2})^w\Lambda_{\delta_1} h}_{L_v^2} +\norm{[\Lambda_{\delta_1},\  (a^{1/2})^w]h}_{L_v^2}\leq C\norm{(a^{1/2})^wh}_{L_v^2},
\end{equation*}
with $C$ independent of $\delta_1$ and $k$.  Meanwhile, by \eqref{comsym}  and the second assertion in \eqref{ew}, one has
\begin{equation}\label{smacomest}
	\norm{[\Lambda_{\delta_1},\  (a^{1/2})^w]h}_{L_v^2}\leq C\delta_1^{1/2}\norm{(a^{1/2})^wh}_{L_v^2}.
\end{equation}
 Note that the above estimate \eqref{smacomest} still holds with $\Lambda_{\delta_1}$ replaced by its adjoint $\Lambda_{\delta_1}^*$ on $L_v^2$. 

Next, we will perform estimates for the regularization  $\Lambda_{\delta_1} \hat f_{m,\delta}$.     To begin with we derive the equations solved by $\Lambda_{\delta_1} \hat f_{m,\delta}$. Observe 
	\begin{equation}\label{eqfour}
		\big(\partial_t+iv\cdot k-\mathcal L\big) \hat f =  \hat\Gamma \big({\hat f},~{\hat f}\big),
\end{equation}
and  			thus  
		\begin{multline}\label{equreg}
		\big(\partial_t+iv\cdot k\big) \Lambda_{\delta_1} \hat f_{m,\delta} -\Lambda_{\delta_1} \mathcal  L \hat f_{m,\delta}=\Lambda_{\delta_1}t^{\varsigma m}  (1+\delta\abs k^2)^{-1/2} \comi{k}^m\hat\Gamma \big({\hat f},~{\hat f}\big) \\ 
		+\varsigma mt^{-1} \Lambda_{\delta_1}\hat f_{m,\delta}+i[v\cdot k,\ \Lambda_{\delta_1}] \hat f_{m,\delta}.
			\end{multline}

\begin{lemma}[Regularization]\label{verf}
Let $f$ satisfy the condition \eqref{+1234}. Then for any $0<t\leq T\leq 1$, it holds that 
	\begin{eqnarray*}
	  t^{-1} \Lambda_{\delta_1}\hat f_{m,\delta},\  [v\cdot k,\  \Lambda_{\delta_1}] \hat f_{m,\delta},\ 	(v\cdot k) \Lambda_{\delta_1} \hat f_{m,\delta}	\in  L_k^1L_T^\infty L_v^2
		\end{eqnarray*}
		and
		\begin{eqnarray*}
		\partial_t \Lambda_{\delta_1} \hat f_{m,\delta},\  \Lambda_{\delta_1} \mathcal  L \hat f_{m,\delta}, \	\Lambda_{\delta_1}t^{\varsigma m}  (1+\delta\abs k^2)^{-1/2} \comi{k}^m\hat\Gamma \big({\hat f},~{\hat f}\big)\in  L_k^1L_T^2 L_v^2
		\end{eqnarray*} 
for $0\leq m\leq 1$,  and the above assertion still holds  for $m\geq 2$  
	provided that the induction assumption \eqref{inassump} is fulfilled.
\end{lemma}
 
 \begin{proof} We first consider the case $m\geq 2$.   Direct verification shows 
 \begin{eqnarray*}
 	\norm{(v\cdot k) \Lambda_{\delta_1} \hat f_{m,\delta}(t,k)}_{L_v^2}\leq C_{\delta_1}\norm{\hat f_{m,\delta}(t,k)}_{L_v^2}\leq C_{\delta_1}C_\delta  t^{\varsigma (m-1)}\comi k^{m-1} \norm{\hat f(t,k)}_{L_v^2}.
 \end{eqnarray*} 
 This with the induction assumption \eqref{inassump} give
 \begin{eqnarray*}
 	(v\cdot k) \Lambda_{\delta_1} \hat f_{m,\delta}	\in  L_k^1L_T^\infty L_v^2,
 \end{eqnarray*} 
 and likewise for $ t^{-1} \Lambda_{\delta_1}\hat f_{m,\delta}$ and $[v\cdot k,\ \Lambda_{\delta_1}] \hat f_{m,\delta}$ by observing the facts that   
 \begin{eqnarray*}
 	 t^{-1} \norm{\Lambda_{\delta_1}\hat f_{m,\delta}(t,k)}_{L_v^2}\leq C_\delta  t^{\varsigma m-1}\comi k^{m-1} \norm{\hat f(t,k)}_{L_v^2}\leq C_\delta  t^{\varsigma (m-1)}\comi k^{m-1} \norm{\hat f(t,k)}_{L_v^2}
 \end{eqnarray*}
 and that the commutator
 \begin{equation}\label{comtes}
 [v\cdot k,\ \Lambda_{\delta_1}]=-2\big  (1+\delta_1  \abs v^2 \big)^{-1-\gamma} \big(1+\delta_1 \abs k^2-\delta_1\Delta_v\big)^{-2} \delta_1 k\cdot \nabla_v
 \end{equation}
 is uniformly bounded on $L_v^2$ with respect to $k$ and $\delta_1$.  We have proved the first assertion for $m\geq 2$.  For the second one    
  we apply Lemma \ref{lemtrip}   
 to obtain,  by virtue of  \eqref{bd} as well as the assertion (iii) in Proposition \ref{estaa}, 
 \begin{eqnarray*}
 \begin{aligned}
 &	\norm{\Lambda_{\delta_1}t^{\varsigma m}  (1+\delta\abs k^2)^{-1/2} \comi{k}^m\hat\Gamma \big({\hat f},~{\hat f}\big) }_{L^2}\leq C_{\delta_1}C_\delta t^{\varsigma m}\comi{k}^{m-1}	\norm{ (a^{-1/2})^w\hat\Gamma \big({\hat f},~{\hat f}\big) }_{L^2}\\
 	&\leq C_mC_{\delta_1}C_\delta t^{\varsigma  m}  \int_{\mathbb Z^3 } \comi {k-\ell}^{m-1}  \norm{
	\hat f(k-\ell)}_{L^{2}_v} \Big( \norm{ (a^{1/2})^w \hat f(\ell)}_{L^{2}_v} \Big)d\Sigma(\ell)\\
	&\quad +C_mC_{\delta_1}C_\delta t^{\varsigma  m}  \int_{\mathbb Z^3 }   \norm{
	\hat f(k-\ell)}_{L^{2}_v} \Big( \comi {\ell}^{m-1} \norm{ (a^{1/2})^w \hat f(\ell)}_{L^{2}_v} \Big)d\Sigma(\ell).
 	\end{aligned}
 \end{eqnarray*} 
 This, along with the induction assumption \eqref{inassump},  enable us to repeat the calculation in \eqref{dica}  to conclude  
 \begin{eqnarray*}
 	\Lambda_{\delta_1}t^{\varsigma m}  (1+\delta\abs k^2)^{-1/2} \comi{k}^m\hat\Gamma \big({\hat f},~{\hat f}\big)\in  L_k^1L_T^2 L_v^2,
 \end{eqnarray*}
  and likewise for $\Lambda_{\delta_1} \mathcal  L \hat f_{m,\delta}$.  Combining  the above assertions with \eqref{equreg} gives $\partial_t\Lambda_{\delta_1}\hat f_{m,\delta}\in  L_k^1L_T^2 L_v^2$.  We have proved 
  the conclusion as desired in Lemma  \ref{verf} for $m\geq 2$,  and  the treatment for $0\leq m\leq 1$ is straightforward by following the above argument.   The proof is thus completed. 
 \end{proof}

\subsection{Subelliptic estimate for regularized solutions} \label{secsub}

In this subsection we will derive  a subelliptic estimate  for $\Lambda_{\delta_1} \hat{f}_{m,\delta}$ that is defined by \eqref{fmdel} and \eqref{regoper}. 
Since the conditions for  $m\leq 1$ and  $m\geq 2$ may be different in the following argument,  we introduce a uniform  Assumption ${\mathcal H}_m$ for $m\geq 0$ that is defined as below.

\begin{definition}[Assumption $\mathcal H_m$] \label{defas}
 Let $m\geq \mathbb Z_+$ and let  $f(t,x,v)$ be the global mild solution to \eqref{eqforper}.   We say that $f$ satisfies the Assumption  $\mathcal H_m$ if
 $f$ satisfies the  estimate \eqref{+1234} when $m\leq 1$,  and satisfies additionally the induction assumption \eqref{inassump} when $m\geq 2$. 
\end{definition}

\begin{proposition}\label{prpsub} Let $f$ satisfy Assumption $\mathcal H_m$ above. Then the following estimates hold.

\noindent (i) It holds that  
\begin{eqnarray*}
	 \int_{\mathbb Z^3} \comi k^{\frac{s}{1+2s}} 	\Big(\int_{0}^{1}\norm{ \hat f(t,k)}_{L_v^2} ^2dt\Big)^{1/2} d\Sigma(k) \leq  C  \norm{f_0}_{  L^1_kL^2_v}.
\end{eqnarray*}

\noindent (ii) For $m=1$, it holds that 
 \begin{multline*}
	  \int_{\mathbb Z^3} \comi k^{\frac{s}{1+2s}} 	\Big(\int_{0}^{1}\norm{ \hat f_{m,\delta}(t,k)}_{L_v^2} ^2dt\Big)^{1/2} d\Sigma(k) 
	 \leq  C \norm{f_0}_{  L^1_kL^2_v}\\+ C\int_{\mathbb Z^3} \sup_{0<t\leq 1} \norm{      \hat f_{m,\delta} (t,k)}_{L_v^2}   d\Sigma(k)
		+C \int_{\mathbb Z^3} \left(\int_0^1    \norm{(a^{1/2})^w\hat f_{m,\delta}(t,k) }^2_{L_v^2}  d t\right)^{1/2}d \Sigma(k).
		\end{multline*}

\noindent (iii)  For $m\geq2$, it holds that  	
\begin{multline*}
				 \int_{\mathbb Z^3} \comi k^{\frac{s}{1+2s}} 	\Big(\int_{0}^{1}\norm{ \hat f_{m,\delta}(t,k)}_{L_v^2} ^2dt\Big)^{1/2} d\Sigma(k) \leq  C  \tilde C_0^{m-2}  \com{(m-1)!} ^{\frac{1+2s}{2s}}\\
				 +C \int_{\mathbb Z^3} \sup_{0<t\leq 1}     \norm{ \hat f_{m,\delta}(t,k )}_{L^2_v}    d \Sigma(k) +C \int_{\mathbb Z^3} \left(\int_0^1    \norm{(a^{1/2})^w\hat f_{m,\delta}(t,k) }^2_{L_v^2}  d t\right)^{1/2}d \Sigma(k).
		\end{multline*}	 
		Here $\tilde C_0$ is the constant in the assumption \eqref{inassump}. 
\end{proposition}

 We will follow the argument presented in \cite{MR3950012}  where the standard subelliptic estimate was established in $L^2(\mathbb T_x^3)$. Here we will derive the estimates in the setting of $L^1_k(\mathbb T_x^3)$ instead.    
Let  $\lambda^{\rm Wick}_k$ be the Wick quantization of symbol $\lambda_k$ (see   Appendix \ref{secapp} for the definition of Wick quantization),  which is defined by, recalling $\tilde a $
 is given in \eqref{defofa},      
\begin{equation*}%\label{cutcut}
\lambda_{k}(v,\eta)=\frac{  d_k(v,\eta)}{ \tilde a(v,k)^{\frac{2s}{1+2s}}}\chi\bigg( \frac{  \tilde a(v,\eta)}{  \tilde a(v,k)^{\frac{1}{1+2s}}}\bigg),
\end{equation*}
with
 \begin{eqnarray*}
	d_k(v,\eta) = \comi v^{\gamma}\inner{1+\abs v^2+\abs k^2+\abs{v\wedge k}^2}
	^{s-1}\Big(k\cdot\eta+(v\wedge k)\cdot(v\wedge \eta)\Big) 
\end{eqnarray*}
and    $\chi\in C_0^\infty(\mathbb R;~[0,1])$ a given  cut-off function such that
$\chi=1$ on
$[-1,1]$ and supp~$\chi \subset[-2,2]$.  
To obtain the subelliptic estimate  it will rely on the following property linking $d_k$ and $\tilde a$ that   
 \begin{equation*}
 	\big\{d_k(v,\eta), \ v\cdot k\big\}=\tilde a(v,k)-\comi v^{2+\gamma}\inner{1+\abs v^2+\abs k^2+\abs{v\wedge k}^2}
	^{s-1},
 \end{equation*}
 where $\{\cdot, \cdot\}$ is the Poisson bracket defined in \eqref{11051505}.
Observe  by direct calculations that 
\begin{eqnarray*}\abs{\partial_v^\alpha\partial_\eta^\beta d_k(v,\eta)}\leq C_{\alpha,\beta} \tilde a(v,k)^{\frac{2s-1}{2s}} \tilde a(v,\eta)^{\frac{1}{2s}},\quad \forall\ \alpha, \beta	\in\mathbb Z_+^3,
\end{eqnarray*}
with $C_{\alpha,\beta}$ constants depending only on $\alpha,\beta$ but independent of $k$, 
and it is clear to see that $\tilde a(v,\eta)\leq 2 \tilde a(v,k)^{\frac{1}{ 1+2s}}$ on the support of  $$(v,\eta)\mapsto \chi\bigg( \frac{  \tilde a(v,\eta)}{  \tilde a(v,k)^{\frac{1}{1+2s}}}\bigg).$$ 
Thus we can verify directly  that 
\begin{eqnarray}\label{45327} \abs{\partial_v^\alpha\partial_\eta^\beta\lambda_k (v,\eta)} \leq C_{\alpha,\beta},\quad \forall \ \alpha,\beta\in\mathbb Z_+^3, \ \forall\  (v,\eta)\in\mathbb R^3\times\mathbb R^3,
\end{eqnarray}
with $C_{\alpha,\beta}$   being constants  depending only on  $\alpha, \beta$ but independent of $k$. This gives
\begin{equation*}%\label{unlamb}
	 \lambda_k\in S(1, \abs{dv}^2+\abs{d\eta}^2)    \textrm{ uniformly for } k.
\end{equation*} 
Here by uniformly for $k$ we mean that the constants in \eqref{45327} are independent of $k.$ Using the relationship \eqref{ww} between the Wick and Weyl quantizations  we can write 
\begin{equation}\label{wwyl}
	\lambda_k^{\rm Wick}=\tilde\lambda_k ^w
\end{equation}
for some real-valued symbol $\tilde\lambda_k\in S(1, \abs{dv}^2+\abs{d\eta}^2) $ uniformly for $k$.  As a result, since any quantization $b^w$ of real-valued symbol $b$ is   self-adjoint on  $L_v^2$  (cf.~Appendix \ref{secapp} ),  so is  $\lambda_k^{\rm Wick}$ in particular.  Moreover, by \eqref{bdness} and Lemma \ref{lemcomest},  there exists a   constant $C_{\gamma,s}$ that depends only on $\gamma$ and $s$ but is independent of $k,$  such that  for any $   h\in L^2(\mathbb R_v^3),$
\begin{equation}\label{ube}
  	\norm{\lambda^{\rm Wick}_k h}_{L_v^2}\leq  C_{\gamma,s} \norm{ h}_{L_v^2},
\end{equation}  
and moreover, with  $(a^{1/2})^wh\in L_v^2$ additionally, 
	\begin{equation*}
 	\big\|\big[ \lambda^{\rm Wick}_k, \  (a^{1/2})^w\big]h \big \|_{L_v^2}\leq C_{s,\gamma} \norm{(a^{1/2})^w  h}_{L_v^2}.
	\end{equation*}
	   The main reason that we use the Wick quantization rather than the classical Weyl quantization  is  due to its positivity of the former; see \eqref{posit} in Appendix \ref{secapp}. 

In view of \eqref{ubdo}, \eqref{ube} and Assumption $\mathcal H_m$ in Definition \ref{defas}, we see  $\lambda_{k}^{\rm Wick} \Lambda_{\delta_1} \hat f_{m,\delta}\in L_k^1L_T^\infty L_v^2$, and thus combining this with Lemma \ref{verf}, we are able to take the scalar $L_v^2$-product on both sides of \eqref{equreg} with $\lambda_{k}^{\rm Wick} \Lambda_{\delta_1} \hat f_{m,\delta} $,  and then  integrate the real parts of the resulting equation over $[t_1,t_2]$;   this gives 
\begin{equation}\label{+t1t2}
	\int_{t_1}^{t_2} {\rm Re}\big( i(v\cdot k) \Lambda_{\delta_1} \hat f_{m,\delta},\  \lambda^{\rm Wick} \Lambda_{\delta_1}   \hat f_{m,\delta}\big)_{L_v^2} dt=\sum_{1\leq p\leq 4} J_{m,p} 
\end{equation}
with
%\begin{eqnarray*}
\begin{align*}
	J_{m,1}&=\frac{1}{2} \big(  \Lambda_{\delta_1} \hat f_{m,\delta}(t_1,k),\   \lambda^{\rm Wick}_k\Lambda_{\delta_1} \hat f_{m,\delta}(t_1,k)\big)_{{L_v^2  }}-\frac{1}{2}\big(  \Lambda_{\delta_1} \hat f_{m,\delta}(t_2,k),\   \lambda^{\rm Wick}_k\Lambda_{\delta_1} \hat f_{m,\delta}(t_2,k)\big)_{{L_v^2  }},\\
	J_{m,2}&=\int_{t_1}^{t_2} {\rm Re}\big(\Lambda_{\delta_1} \mathcal L \hat f_{m,\delta},\ \lambda^{\rm Wick}_k\Lambda_{\delta_1}   \hat f_{m,\delta}\big)_{L_v^2} dt,\\
	J_{m,3}&=\int_{t_1}^{t_2} {\rm Re}\Big(\Lambda_{\delta_1}t^{\varsigma m}  (1+\delta\abs k^2)^{-1/2} \comi{k}^m\hat\Gamma \big({\hat f},~{\hat f}\big),\ \lambda^{\rm Wick}_k\Lambda_{\delta_1}   \hat f_{m,\delta}\Big)_{L_v^2} dt,\\
	J_{m,4}&=\int_{t_1}^{t_2} {\rm Re}\Big(
 \varsigma mt^{-1} \Lambda_{\delta_1}\hat f_{m,\delta}+i[v\cdot k,\ \Lambda_{\delta_1}] \hat f_{m,\delta},\ \lambda^{\rm Wick}_k\Lambda_{\delta_1}   \hat f_{m,\delta}\Big)_{L_v^2} dt.
\end{align*}
%\end{eqnarray*}
Here  we have used  the relation 
\begin{eqnarray*}
	{\rm Re}  \big( \partial_t \Lambda_{\delta_1} \hat f_{m,\delta},\   \lambda^{\rm Wick}_k \Lambda_{\delta_1} \hat f_{m,\delta} \big)_{{L_v^2  }}=\frac{1}{2}\frac{d}{dt}\big(  \Lambda_{\delta_1} \hat f_{m,\delta},\   \lambda^{\rm Wick}_k\Lambda_{\delta_1} \hat f_{m,\delta}\big)_{{L_v^2  }}
\end{eqnarray*}
due to the fact that  $\lambda^{\rm Wick}_k$ is self-adjoint on $L_v^2$.  As a result, 
\begin{equation}\label{sumj}
\int_{\mathbb Z^3}	\Big|\int_{t_1}^{t_2} {\rm Re}\big( i(v\cdot k) \Lambda_{\delta_1} \hat f_{m,\delta},\  \lambda^{\rm Wick} \Lambda_{\delta_1}   \hat f_{m,\delta}\big)_{L_v^2} dt\Big|^{\frac{1}{2}} d\Sigma(k) \leq C \sum_{1\leq p\leq 4} \int_{\mathbb Z^3}\abs{J_{m,p}}^{\frac{1}{2}} d\Sigma(k).
\end{equation}

We will proceed through the following lemmas to derive the lower and upper bounds respectively  for the terms on the left and  right hand sides of \eqref{sumj}.

\begin{lemma}\label{sd}
	Let $m\in\mathbb Z_+$. Under Assumption $\mathcal H_m$ given in Definition \ref{defas},  it holds that, for any $0<t_1<t_2\leq 1$,
	\begin{multline*}
	\int_{\mathbb Z^3} \comi k^{\frac{s}{1+2s}} 	\Big(\int_{t_1}^{t_2}\norm{\Lambda_{\delta_1}   \hat f_{m,\delta}(t,k)}_{L_v^2} ^2dt\Big)^{1/2} d\Sigma(k)\\
	  \leq  C	\int_{\mathbb Z^3} \Big|\int_{t_1}^{t_2} {\rm Re}\big( i(v\cdot k) \Lambda_{\delta_1} \hat f_{m,\delta},\  \lambda^{\rm Wick} \Lambda_{\delta_1}   \hat f_{m,\delta}\big)_{L_v^2} dt\Big|^{1/2} d\Sigma(k)\\ +C\int_{\mathbb Z^3} \Big(\int_{0}^{1} \norm{ (a^{1/2})^w \hat f_{m,\delta}(t,k)}_{L_v^2}^2 dt\Big)^{1/2}d\Sigma(k).
	\end{multline*}
\end{lemma}

 \begin{proof}
 	We  follow the argument presented in \cite{MR3950012}.   Observe $ v\cdot k =\inner{v\cdot k}^{\rm Wick}$ by \eqref{ww}. 
 	Then, by \eqref{11082406} in Appendix \ref{secapp},
  \begin{equation}\label{11040808}
    {\rm Re}~\big(i \inner{v\cdot k} \Lambda_{\delta_1} \hat f_{m,\delta},~\lambda_{k}^{\rm Wick} \Lambda_{\delta_1} \hat f_{m,\delta}\big)_{L_v^2}=  \big(\big\{\lambda_k,~v\cdot k
    \big\}^{\rm Wick}\Lambda_{\delta_1} \hat f_{m,\delta},~\Lambda_{\delta_1} \hat f_{m,\delta}\big)_{L_v^2},
  \end{equation}
  with  the Poisson
bracket $\{\cdot, \cdot\}$  defined by
 \eqref{11051505}.
 Moreover, recalling 
\begin{equation*}%\label{cutcut}
\lambda_{k}(v,\eta)=\frac{  d_k(v,\eta)}{ \tilde a(v,k)^{\frac{2s}{1+2s}}}\psi(v,\eta), 
\end{equation*}
with 
\begin{equation}\label{psf}
\psi(v,\eta)=\chi\bigg( \frac{  \tilde a(v,\eta)}{  \tilde a(v,k)^{\frac{1}{1+2s}}}\bigg)
\end{equation}
and 
 \begin{eqnarray*}
	d_k(v,\eta) = \comi v^{\gamma}\big(1+\abs v^2+\abs k^2+\abs{v\wedge k}^2\big)
	^{s-1}\big(k\cdot\eta+(v\wedge k)\cdot(v\wedge \eta)\big),
\end{eqnarray*}
we compute directly 
\begin{eqnarray*}
\begin{aligned}
     & \big\{\lambda_k,~v\cdot k\big\}= \tilde a(v, k)^{\frac{1}{1+2s}}\psi-\frac{\comi
v^{\gamma+2}\big(1+\abs v^2+\abs k^2+\abs{v\wedge
     k}^2\big)^{s-1}}{\tilde a(v, k)^{\frac{2s}{1+2s}}}\psi+\frac{  d_k(v,\eta)}
{\tilde a(v, k)^{\frac{2s}{1+2s}}} k\cdot\partial_\eta\psi\\  
&=  \tilde a(v,
 k)^{\frac{1}{1+2s}}-\tilde a(v, k)^{\frac{1}{1+2s}}\inner{1-\psi}- \comi 
v^{\gamma+2}\big(1+\abs v^2+\abs k^2+\abs{v\wedge
     k}^2\big)^{s-1} \tilde a(v, k)^{-\frac{2s}{1+2s}}\psi\\
 &\quad+  d_k(v,\eta)\tilde a(v, k)^{-\frac{2s}{1+2s}} k\cdot\partial_\eta
\psi.
\end{aligned}
    \end{eqnarray*}
   This     with \eqref{11040808}  yield that 
%\begin{equation}
 \begin{align}
 &  \Big( \big(\tilde a(v, k)^{\frac{1}{1+2s}}\big)^{\rm Wick} \Lambda_{\delta_1} \hat f_{m,\delta},~\Lambda_{\delta_1} \hat f_{m,\delta}\Big)_{L_v^2}
=  {\rm Re}~\big(i \inner{v\cdot k} \Lambda_{\delta_1} \hat f_{m,\delta},~\lambda_{k}^{\rm Wick} \Lambda_{\delta_1} \hat f_{m,\delta}\big)_{L_v^2}\notag \\
&\quad +\Big(\big(\tilde
      a(v, k)^{\frac{1}{1+2s}}\inner{1-\psi}\big)^{\rm Wick}
   \Lambda_{\delta_1} \hat f_{m,\delta},~~\Lambda_{\delta_1} \hat f_{m,\delta}\Big)_{L_v^2}\notag \\
   &\quad+\Big(\big(\comi
v^{\gamma+2}\big(1+\abs v^2+\abs k^2+\abs{v\wedge
     k}^2\big)^{s-1}\tilde a(v, k)^{-\frac{2s}{1+2s}}\psi\big)^{\rm Wick} \Lambda_{\delta_1} \hat f_{m,\delta},~~\Lambda_{\delta_1} \hat f_{m,\delta}\Big)_{L_v^2}\notag \\
     &\quad+\Big(\big(-  d_k(v,\eta)\tilde a(v, k)^{-\frac{2s}{1+2s}}
 k\cdot\partial_\eta\psi\big)^{\rm Wick} \Lambda_{\delta_1} \hat f_{m,\delta},~~\Lambda_{\delta_1} \hat f_{m,\delta}\Big)_{L_v^2}.\label{11052920}
\end{align}
%\end{equation}
 Note that
$\tilde a(v, k)^{\frac{1}{1+2s}} \leq \tilde a(v,\eta)$
on the support of $ 1-\psi$ with $\psi$ defined by \eqref{psf}, and thus
\[
   \tilde a(v, k)^{\frac{1}{1+2s}}(1-\psi) \leq  \tilde a(v, \eta).
\]
Then the  positivity \eqref{posit} of the Wick quantization gives
\begin{equation}\label{+11052921}
  \Big(\big(\tilde
      a(v, k)^{\frac{1}{1+2s}}\inner{1-\psi}\big)^{\rm Wick}
   \Lambda_{\delta_1} \hat f_{m,\delta},~~\Lambda_{\delta_1} \hat f_{m,\delta}\Big)_{L_v^2}\leq \Big(\big( \tilde a(v, \eta)\big)^{\rm Wick}
   \Lambda_{\delta_1} \hat f_{m,\delta},~~\Lambda_{\delta_1} \hat f_{m,\delta}\Big)_{L_v^2}.
\end{equation}
Similarly, observing \[
    \comi
v^{\gamma+2}\inner{1+\abs v^2+\abs k^2+\abs{v\wedge
     k}^2}^{s-1}\tilde a(v, k)^{-\frac{2s}{1+2s}}\psi\leq \comi{ v}^{2s+
\gamma}\leq C\tilde a(v,\eta),
\]
and
\[
   - d_k(v,\eta)\tilde
     a(v, k)^{-\frac{2s}{1+2s}} k\cdot\partial_\eta\psi\leq C
   \tilde a(v,\eta),
   \]
we have
\begin{multline*}
  \Big(\big(\comi
v^{\gamma+2}\big(1+\abs v^2+\abs k^2+\abs{v\wedge
     k}^2\big)^{s-1}\tilde a(v, k)^{-\frac{2s}{1+2s}}\psi\big)^{\rm Wick} \Lambda_{\delta_1} \hat f_{m,\delta},~~\Lambda_{\delta_1} \hat f_{m,\delta}\Big)_{L_v^2}\\
     +\Big(\big(-  d_k(v,\eta)\tilde a(v, k)^{-\frac{2s}{1+2s}}
 k\cdot\partial_\eta\psi\big)^{\rm Wick} \Lambda_{\delta_1} \hat f_{m,\delta},~~\Lambda_{\delta_1} \hat f_{m,\delta}\Big)_{L_v^2}\\
      \leq C \Big(\big( \tilde a(v,\eta)\big)^{\rm Wick}
   \Lambda_{\delta_1} \hat f_{m,\delta},~~\Lambda_{\delta_1} \hat f_{m,\delta}\Big)_{L_v^2}.
\end{multline*}
We combine the above estimate with \eqref{+11052921} and \eqref{11052920} to conclude
\begin{multline}\label{laet}
    \Big( \big(\tilde a(v, k)^{\frac{1}{1+2s}}\big)^{\rm Wick} \Lambda_{\delta_1} \hat f_{m,\delta},~\Lambda_{\delta_1} \hat f_{m,\delta}\Big)_{L_v^2}
\leq   {\rm Re}~\big(i \inner{v\cdot k} \Lambda_{\delta_1} \hat f_{m,\delta},~\lambda_{k}^{\rm Wick} \Lambda_{\delta_1} \hat f_{m,\delta}\big)_{L_v^2}\\
  +   C \Big(\big( \tilde a(v,\eta)\big)^{\rm Wick}
   \Lambda_{\delta_1} \hat f_{m,\delta},~~\Lambda_{\delta_1} \hat f_{m,\delta}\Big)_{L_v^2} .
  \end{multline}
For the  term on the left hand side, we have, by  \eqref{ww},
\begin{equation*}
%\label{ }
\big(\tilde a(v, k)^{\frac{1}{1+2s}}\big)^{\rm Wick}=8\int \tilde
     a(v-\tilde v, k)^{\frac{1}{1+2s}} e^{- 2\pi(|\tilde v|^2+|\tilde\eta|^2)}  
     d\tilde v   d\tilde \eta \geq C^{-1} \tilde a(v, k)^{1/\inner{1+2s}} \geq C^{-1}\comi k^{\frac{2s}{1+2s}},
\end{equation*}
where the first inequality follows from direct calculations (cf. \cite[p.61]{MR3950012}), and the last inequality holds because of the definition \eqref{defofa} of $\tilde a$ by observing   $\gamma\geq 0$.   On the other hand, for the last term on the right hand side of \eqref{laet} we write
\begin{eqnarray*}
	\big( \tilde a(v,\eta)\big)^{\rm Wick}=(a^{1/2})^w \underbrace{ \big[(a^{1/2})^w\big]^{-1}\big( \tilde a(v,\eta)\big)^{\rm Wick} \big[(a^{1/2})^w\big]^{-1}}_{\textrm{bounded on } L_v^2 } (a^{1/2})^w,
\end{eqnarray*}
where the boundedness on $L_v^2$ follows from the assertion (iii) in Proposition \ref{estaa} as well as the composition formula \eqref{symbc}  and the relationship \eqref{ww} between the Wick and Weyl quantizations.    
This gives that  
\begin{eqnarray*}
	\Big(\big( \tilde a(v,\eta)\big)^{\rm Wick}
   \Lambda_{\delta_1} \hat f_{m,\delta},~~\Lambda_{\delta_1} \hat f_{m,\delta}\Big)_{L_v^2}\leq C\norm{(a^{1/2})^w\Lambda_{\delta_1} \hat f_{m,\delta}}_{L_v^2}^2\leq C\norm{(a^{1/2})^w  \hat f_{m,\delta}}_{L_v^2}^2,
\end{eqnarray*}
where in the last inequality we have used Lemma \ref{lemcomest} in view of \eqref{ew+} and \eqref{ew}. 
  As  a result, combining the above inequalities with \eqref{laet}  we have  
 \begin{multline*}
	 \comi k^{\frac{2s}{1+2s}} 	\int_{t_1}^{t_2}\norm{\Lambda_{\delta_1}   \hat f_{m,\delta}(t,k)}_{L_v^2} ^2dt\\
	  \leq  C	\int_{t_1}^{t_2} {\rm Re}\big( i(v\cdot k) \Lambda_{\delta_1} \hat f_{m,\delta},\  \lambda^{\rm Wick} \Lambda_{\delta_1}   \hat f_{m,\delta}\big)_{L_v^2} dt+C\int_{t_1}^{t_2} \norm{ (a^{1/2})^w \hat f_{m,\delta}(t,k)}_{L_v^2}^2 dt,
	\end{multline*}
which yields the desired estimate in Lemma \ref{sd}. The proof  is completed.  
   \end{proof}
 
 \begin{lemma}
Let $J_{m,1}$ be defined in terms of  \eqref{+t1t2}, that is 
\begin{equation*}
		J_{m,1} =\frac{1}{2} \big(  \Lambda_{\delta_1} \hat f_{m,\delta}(t_1,k),\   \lambda^{\rm Wick}_k\Lambda_{\delta_1} \hat f_{m,\delta}(t_1,k)\big)_{{L_v^2  }}-\frac{1}{2}\big(  \Lambda_{\delta_1} \hat f_{m,\delta}(t_2,k),\   \lambda^{\rm Wick}_k\Lambda_{\delta_1} \hat f_{m,\delta}(t_2,k)\big)_{{L_v^2  }}. 
\end{equation*}	
Then, for any $0<t_1<t_2\leq 1$ and  any $m\geq 0$, it holds that  
\begin{eqnarray*}
	\int_{\mathbb Z^3}  \abs{J_{m,1}}^{1/2} d\Sigma(k) 	  \leq  C  \int_{\mathbb Z^3}  \norm{ \hat f_{m,\delta} (t_1,k)}_{L_v^2}    d\Sigma(k)+ C\int_{\mathbb Z^3}  \norm{      \hat f_{m,\delta} (t_2,k)}_{L_v^2}    d\Sigma(k). 	
\end{eqnarray*}
 
\end{lemma}

\begin{proof}
	This just follows from \eqref{ubdo} and \eqref{ube}.
\end{proof}

\begin{lemma}\label{lemj2}
With  $J_{m,2}$  defined in terms of  \eqref{+t1t2}, that is,
\begin{eqnarray*}
	J_{m,2} =\int_{t_1}^{t_2} {\rm Re}\big(\Lambda_{\delta_1} \mathcal L \hat f_{m,\delta},\ \lambda^{\rm Wick}_k\Lambda_{\delta_1}   \hat f_{m,\delta}\big)_{L_v^2} dt,
\end{eqnarray*}
 it holds that for any $0<t_1<t_2\leq 1$ and any $m\geq 0,$  	\begin{eqnarray*}
	\begin{aligned}
   \int_{\mathbb Z^3} \abs{J_{m,2}}^{1/ 2}d\Sigma(k) 
 \leq  C\left(\int_0^1    \norm{(a^{1/2})^w\hat f_{m,\delta}(t,k) }^2_{L_v^2}  d t\right)^{1/2}d \Sigma(k). \end{aligned}	
\end{eqnarray*} 
\end{lemma}

\begin{proof}
	Using the assertion (iii) in Proposition \ref{estaa} gives that   
%\begin{eqnarray*}
\begin{align*}
			\abs{J_{m,2}} 
			&\leq C \int_{t_1}^{t_2}  \norm{(a^{-1/2})^w\mathcal L \hat f_{m,\delta}}_{L_v^2}\norm{(a^{1/2})^w\Lambda_{\delta_1}^*\lambda^{\rm Wick}_k\Lambda_{\delta_1}   \hat f_{m,\delta}}_{L_v^2} dt\leq C \int_{0}^{1}  \norm{(a^{1/2})^w \hat f_{m,\delta}}_{L_v^2}^2 dt,
\end{align*}
%\end{eqnarray*}
where we have denoted by $\Lambda_{\delta_1}^*$ the adjoint operator of $\Lambda_{\delta_1}$ on $L_v^2$ and the last inequality follows from   Corollary \ref{corupp} and Lemma \ref{lemcomest}  in view of    \eqref{ew} and \eqref{wwyl}.  Thus the proof of Lemma \ref{lemj2} is completed. 
\end{proof}

\begin{lemma}\label{lemj3}
 Let $m\in\mathbb Z_+$ and let  $J_{m,3}$ be defined in terms of \eqref{+t1t2}, that is,
 \begin{eqnarray*}
	J_{m,3} =\int_{t_1}^{t_2} {\rm Re}\Big(\Lambda_{\delta_1}t^{\varsigma m}  (1+\delta\abs k^2)^{-1/2} \comi{k}^m\hat\Gamma \big({\hat f},~{\hat f}\big),\ \lambda^{\rm Wick}_k\Lambda_{\delta_1}   \hat f_{m,\delta}\Big)_{L_v^2} dt.
 \end{eqnarray*}
  Suppose that $f$ satisfies Assumption $\mathcal H_m$ given in Definition \ref{defas}.     Then, for any $\eps>0 $, the following things hold. 
 
 \noindent (i) For $m=0$, it holds that  
 \begin{eqnarray*}
 \begin{aligned}
 	  \int_{\mathbb Z^3} \abs{J_{m,3}}^{1/ 2}d\Sigma(k) 
 \leq C\norm{f_0}_{  L^1_kL^2_v}.
 \end{aligned}
 \end{eqnarray*}

 \noindent (ii) For $m=1$, it holds that 
 \begin{eqnarray*}
 \begin{aligned}
 	  \int_{\mathbb Z^3} \abs{J_{m,3}}^{1/ 2}d\Sigma(k) 
 &\leq  \left( \eps +  C \eps^{-1}\varepsilon_0   \right)\int_{\mathbb Z^3} \left(\int_0^1    \norm{(a^{1/2})^w\hat f_{m,\delta}(t,k) }^2_{L_v^2}  d t\right)^{1/2}d \Sigma(k)\\
  &\quad+C \eps^{-1}\varepsilon_0 \int_{\mathbb Z^3} \sup_{0<t\leq 1}     \norm{ \hat f_{m,\delta}(t,k )}_{L^2_v}    d \Sigma(k).
 \end{aligned}
 \end{eqnarray*}

 \noindent (iii)  For $m\geq 2$, it holds that   
 %\begin{eqnarray*}
\begin{align*}
   \int_{\mathbb Z^3} \abs{J_{m,3}}^{1/ 2}d\Sigma(k) 
&\leq  \left( \eps +  C \eps^{-1}\varepsilon_0   \right)\int_{\mathbb Z^3} \left(\int_0^1    \norm{(a^{1/2})^w\hat f_{m,\delta}(t,k) }^2_{L_v^2}  d t\right)^{1/2}d \Sigma(k)\\
&\quad+C \eps^{-1}\varepsilon_0 \int_{\mathbb Z^3} \sup_{0<t\leq 1}     \norm{ \hat f_{m,\delta}(t,k )}_{L^2_v}    d \Sigma(k)+C \eps^{-1} \tilde C_0^{m-2}  \com{(m-1)!} ^{\frac{1+2s}{2s}},
\end{align*}
%\end{eqnarray*}	
with $\tilde C_0$ the constant in the induction assumption \eqref{inassump}.
\end{lemma}
 
\begin{proof}
	 We first apply the assertion (iii) of Proposition  \ref{estaa} and then Lemmas \ref{lemcomest} and \ref{lemtrip},  to compute 
 \begin{equation*} 
 \begin{aligned}
 &\big|\big(\Lambda_{\delta_1}t^{\varsigma m}  (1+\delta\abs k^2)^{-1/2} \comi{k}^m\hat\Gamma \big({\hat f},~{\hat f}\big) ,\   \lambda_{k}^{\rm Wick}\Lambda_{\delta_1} f_{m,\delta} \big)_{L^2_v}\big|\\
& \leq    C \norm{(a^{1/2})^w  \Lambda_{\delta_1}^*  \lambda_{k}^{\rm Wick} \Lambda_{\delta_1}\hat  f_{m,\delta} (k)}_{L^2_v}  \times\Big(t^{\varsigma m}  (1+\delta\abs k^2)^{-1/2} \comi{k}^m\norm{(a^{-1/2})^w\hat\Gamma \big({\hat f},~{\hat f}\big)}_{L^2_v} \Big)\\
& \leq C  \norm{(a^{1/2})^w   \hat  f_{m,\delta} (k)}_{L^2_v} \int_{\mathbb Z^3}  t^{\varsigma m}\frac{\comi{k}^m}{(1+\delta\abs k^2)^{1/2}}\norm{ \hat f (k-\ell)}_{L^2_v}\norm{(a^{1/2})^w   \hat f  (\ell)}_{L^2_v} d\Sigma(\ell).
 \end{aligned}
 \end{equation*}
 This gives
 %\begin{equation}
 	\begin{align*}
 		&\int_{\mathbb Z^3}\abs{J_{m,3}}^{1/2}d\Sigma(k)\leq C \int_{\mathbb Z^3}\bigg[\int_0^1 \norm{(a^{1/2})^w   \hat  f_{m,\delta} (t,k)}_{L^2_v} \notag \\
 		&\qquad\qquad\qquad\quad \times \Big(\int_{\mathbb Z^3} \frac{ t^{\varsigma m}\comi{k}^m}{(1+\delta\abs k^2)^{1/2}}\norm{ \hat f (t,k-\ell)}_{L^2_v}\norm{(a^{1/2})^w   \hat f  (t,\ell)}_{L^2_v} d\Sigma(\ell)\Big)dt\bigg]^{1/2}d\Sigma(k).
 	\end{align*}
Hence it holds that
 %\begin{equation}
 	\begin{multline}
 		\int_{\mathbb Z^3}\abs{J_{m,3}}^{1/2}d\Sigma(k)\leq \eps      \int_{\mathbb Z^3}\ \Big(\int_0^1  \norm{(a^{1/2})^w  \hat  f_{m,\delta}(t,k)}_{L^2_v}^2  d t\Big)^{1/2} d \Sigma(k)  \\
 +C\eps^{-1} \int_{\mathbb Z^3} \bigg[\int_0^1\Big(\int_{\mathbb Z^3} \frac{ t^{\varsigma m}\comi{k}^m}{(1+\delta\abs k^2)^{1/2}}\norm{ \hat f (t,k-\ell)}_{L^2_v}\norm{(a^{1/2})^w   \hat f  (t,\ell)}_{L^2_v} d\Sigma(\ell)\Big)dt\bigg]^{1/2} d \Sigma(k).
		\label{jm3}
 	\end{multline}
 %\end{equation}
 
 \noindent  (a) We first consider the case when $m=0$. 
 In view of \eqref{jm3}  the first assertion for $m=0$ obviously holds as a result of \eqref{+1234} by observing that for $m=0$ we have $\norm{(a^{1/2})^w  \hat  f_{m,\delta}}_{L^2_v}\leq \norm{(a^{1/2})^w  \hat  f}_{L^2_v}$ and
 \begin{eqnarray*}
 \begin{aligned}
 	&\int_{\mathbb Z^3} \bigg[\int_0^1\Big(\int_{\mathbb Z^3} \frac{ t^{\varsigma m}\comi{k}^m}{(1+\delta\abs k^2)^{1/2}}\norm{ \hat f (t,k-\ell)}_{L^2_v}\norm{(a^{1/2})^w   \hat f  (t,\ell)}_{L^2_v} d\Sigma(\ell)\Big)dt\bigg]^{1/2} d \Sigma(k)\\
 &\leq C \Big(\int_{\mathbb Z^3} 	 \sup_{0<t\leq 1}\norm{
	\hat f(t,k)}_{L^{2}_v}  d\Sigma(k) \Big)\int_{\mathbb Z^3 } \Big(\int_0^1   \norm{ (a^{1/2})^w \hat f(t,k)}_{L^{2}_v}^2 dt\Big)^{1/2}d\Sigma(k),
 	\end{aligned}
 \end{eqnarray*}
where in the last inequality we have used Lemma \ref{lemtl}.  
 
 \medskip
  \noindent  (b) 
 Next consider the case when  $m=1$.    Applying  Lemma \ref{est} for $m=1$ gives
 \begin{eqnarray*}
 \frac{ t^{\varsigma  }\comi{k} }{(1+\delta\abs k^2)^{1/2}} \leq         \frac{2t^{\varsigma  }\comi{k-\ell}}{{(1+\delta|k-\ell|^2)^{1/2}}}  +  \frac{2t^{\varsigma  }\comi{\ell}}{{(1+\delta|\ell|^2)^{1/2}}}.
 \end{eqnarray*}
 Then,    using again Lemma \ref{lemtl} and \eqref{+1234},   
 \begin{eqnarray*}
 \begin{aligned}
 	&\int_{\mathbb Z^3} \bigg[\int_0^1\Big(\int_{\mathbb Z^3} \frac{ t^{\varsigma }\comi{k}}{(1+\delta\abs k^2)^{1/2}}\norm{ \hat f (t,k-\ell)}_{L^2_v}\norm{(a^{1/2})^w   \hat f  (t,\ell)}_{L^2_v} d\Sigma(\ell)\Big)dt\bigg]^{1/2} d \Sigma(k)\\
 &\leq C \eps_0 \int_{\mathbb Z^3} 	 \sup_{0<t\leq 1}\norm{
	\hat f_{1,\delta}(t,k)}_{L^{2}_v}  d\Sigma(k) +C\eps_0\int_{\mathbb Z^3 } \Big(\int_0^1   \norm{ (a^{1/2})^w \hat f_{1,\delta}(t,k)}_{L^{2}_v}^2 dt\Big)^{1/2}d\Sigma(k).
 	\end{aligned}
 \end{eqnarray*}
Substituting  the above estimate into \eqref{jm3} with $m=1$, we obtain the second assertion for $m=1$.

 \medskip
 \noindent (c) 
 It remains to consider the case when $m=2$.  In view of \eqref{jm3}
 the  desired assertion follows from showing that 
%\begin{equation}
 	\begin{align}
 		&\int_{\mathbb Z^3} \bigg[\int_0^1\Big(\int_{\mathbb Z^3} \frac{ t^{\varsigma m}\comi{k}^m}{(1+\delta\abs k^2)^{1/2}}\norm{ \hat f (t,k-\ell)}_{L^2_v}\norm{(a^{1/2})^w   \hat f  (t,\ell)}_{L^2_v} d\Sigma(\ell)\Big)dt\bigg]^{1/2} d \Sigma(k)\notag \\
 		&\leq C \tilde C_0^{m-2}  \com{(m-1)!} ^{\frac{1+2s}{2s}} +C  \varepsilon_0 \int_{\mathbb Z^3} \sup_{0<t\leq 1}     \norm{ \hat f_{m,\delta}(t,k )}_{L^2_v}    d \Sigma(k)\notag \\
		&\quad+C\varepsilon_0 \int_{\mathbb Z^3} \left(\int_0^1    \norm{(a^{1/2})^w\hat f_{m,\delta}(t,k) }^2_{L_v^2}  d t\right)^{1/2}d \Sigma(k).\label{estlae}
 	\end{align}
%\end{equation}
Indeed, we use Lemma \ref{est} to get  \begin{equation*} 
 \begin{split}
  &\int_{\mathbb Z_\ell^3} t^{\varsigma m}\frac{\comi{k}^m}{(1+\delta\abs k^2)^{1/2}} \norm{ \hat f (k-\ell)}_{L^2_v}\norm{(a^{1/2})^w   \hat f  (\ell)}_{L^2_v} d\Sigma(\ell) \\
 \leq& C  \sum_{j=1}^{m-1} \begin{pmatrix}
      m    \\
      j 
\end{pmatrix} t^{\varsigma m}\int_{\mathbb Z^3_\ell}\comi{k-\ell}^j\comi{\ell}^{m-j} \norm{ \hat f(k-\ell)}_{L^2_v}\norm{(a^{1/2})^w   f(\ell)}_{L^2_v}d\Sigma(\ell)\\
 &+C t^{\varsigma m}\int_{\mathbb Z^3_\ell} \frac{\comi{k-\ell}^m}{{(1+\delta|k-\ell|^2)^{1/2}}}   \norm{ \hat f(k-\ell)}_{L^2_v}\norm{(a^{1/2})^w  \hat  f(\ell)}d\Sigma(\ell)\\
  &+C t^{\varsigma m}\int_{\mathbb Z^3_\ell} \frac{\comi{\ell}^{m}}{{(1+\delta|\ell|^2)^{1/2}}} \norm{ \hat f(k-\ell)}_{L^2_v}\norm{(a^{1/2})^w  \hat  f(\ell)}d\Sigma(\ell).
 \end{split}
 \end{equation*}
Moreover, we may apply Lemma \ref{lemtl} for $j_0=m+1$, with 
\begin{eqnarray*}
	(f_j, g_j)=\big( t^{\varsigma j} \comi{D_x}^jf/j!,\ t^{\varsigma (m-j)} \comi{D_x}^{m-j}f/(m-j)!\big), \quad 1\leq j\leq m-1,
\end{eqnarray*}
and 
\begin{eqnarray*}
	(f_m, g_m)=\big( t^{\varsigma m} \comi{D_x}^m\comi{\delta D_x}^{-1}f,\ f\big), \quad 	(f_{m+1}, g_{m+1})=\big( f,\ t^{\varsigma m} \comi{D_x}^m\comi{\delta D_x}^{-1}f\big).\end{eqnarray*}
These, together with  \eqref{+1234},  give that 
%\begin{eqnarray*}
\begin{align*}
		&\int_{\mathbb Z^3} \bigg[\int_0^1\Big(\int_{\mathbb Z^3} \frac{ t^{\varsigma m}\comi{k}^m}{(1+\delta\abs k^2)^{1/2}}\norm{ \hat f (t,k-\ell)}_{L^2_v}\norm{(a^{1/2})^w   \hat f  (t,\ell)}_{L^2_v} d\Sigma(\ell)\Big)dt\bigg]^{1/2} d \Sigma(k)\\
		&  \leq   C\sum_{ j=1}^{m-1} \begin{pmatrix}
      m    \\
      j 
\end{pmatrix}    \int_{\mathbb Z^3}\comi{k}^j\sup_{0<t\leq 1}   t^{{\varsigma j}}   \norm{\hat f (t) }_{L_v^2}    d \Sigma(k)   \int_{\mathbb Z^3}\comi{k}^{m-j}\Big(\int_0^1 t^{2{\varsigma (m-j)}}   \norm{\hat f (t) }^2_{L_v^2}  d t\Big)^{\frac{1}{2}}d \Sigma(k) \\
		&\quad +C \eps_0\int_{\mathbb Z^3}  \sup_{0<t\leq 1}       \norm{ \hat f_{m,\delta}(t,k )}_{L^2_v}    d \Sigma(k)  +C\eps_0 \int_{\mathbb Z^3}  \left(\int_0^1     \norm{(a^{1/2})^w   \hat f_{m,\delta}(t, k )}_{L^2_v}  ^{2} d t \right)^{1/2} d \Sigma(k).	\end{align*}
%\end{eqnarray*}
Moreover, for the summation on the left hand side,    we  use  the induction assumption \eqref{inassump} to compute 
%\begin{equation} 
 \begin{align}
  &\sum_{1\leq j\leq m-1} \begin{pmatrix}
      m    \\
      j 
\end{pmatrix}\int_{\mathbb Z^3}\comi{k }^{j} \sup_{0<t\leq 1} t^{\varsigma j}    \norm{ \hat f(t,k )}_{L^2_v}    d \Sigma(k) \notag \\
  & \qquad \qquad  \qquad \qquad \qquad \quad \times  \int_{\mathbb Z^3}\comi{k }^{m-j}\left(\int_0^1t^{2\varsigma (m-j)}    \norm{(a^{1/2})^w   \hat f(t,k )}_{L^2_v}  ^{2} d t \right)^{1/2} d \Sigma(k) \notag  \\
  \leq&\sum_{1\leq j \leq m-1} \frac{m !} {j!(m-j)!}  \tilde C_0^{j-1}\com{\inner{ j-1}! }^{\frac{1+2s}{2s}}  \Big(  \tilde C_0^{m-j-1}\com{(m-j-1)!} ^{\frac{1+2s}{2s}}  \Big) \notag \\
  \leq &C \tilde C_0^{m-2}\sum_{1\leq j \leq m-1} \frac{m !} {j (m-j) }   \com{\inner{ j-1}! }^{\frac{1}{2s}}  \com{(m-j-1)!} ^{\frac{1}{2s}} \notag \\
  \leq & C  \tilde C_0^{m-2}\sum_{1\leq j \leq m-1} \frac{m !} {j (m-j) }      \com{(m-2)!} ^{\frac{1}{2s}} \notag \\
   \leq & C  \tilde C_0^{m-2}  \com{(m-1)!} ^{\frac{1+2s}{2s}} \sum_{1\leq j \leq m-1} \frac{m} { j (m-j) m^{\frac{1}{2s}}}   \leq   C  \tilde C_0^{m-2}  \com{(m-1)!} ^{\frac{1+2s}{2s}}, \label{tecalp1}
 \end{align}
 %\end{equation}   
where the last inequality holds because of the fact that
 \begin{equation*}%\label{tecal}
 \begin{aligned}
 	\sum_{1\leq j \leq m-1} \frac{m} { j (m-j) m^{\frac{1}{2s}}}  &= \bigg(
 		\sum_{1\leq j < (m-1)/2}+	\sum_{(m-1)/2\leq j \leq m-1}  \bigg)\frac{m} { j (m-j) m^{\frac{1}{2s}}} \\
 		&\leq   \frac{8}{m^{\frac{1}{2s}}}\sum_{1\leq j \leq  m-1} \frac{1} { j  }\leq C_s
 		\end{aligned} 
 \end{equation*}
with $C_s$ a constant depending only on  $s$.  As a result, combining these inequalities we obtain 
  the desired estimate \eqref{estlae}. The proof of Lemma \ref{lemj3} is thus completed. 
      \end{proof}

 	\begin{lemma}\label{lemindc}
Let $m\in\mathbb Z_+$ and let $J_{m,4}$ be defined in terms of  \eqref{+t1t2}, that is,
\begin{eqnarray*}
	J_{m,4} =\int_{t_1}^{t_2} {\rm Re}\Big(
 \varsigma mt^{-1} \Lambda_{\delta_1}\hat f_{m,\delta}+i[v\cdot k,\ \Lambda_{\delta_1}] \hat f_{m,\delta},\ \lambda^{\rm Wick}_k\Lambda_{\delta_1}   \hat f_{m,\delta}\Big)_{L_v^2} dt.
\end{eqnarray*} 
Suppose that $f$ satisfies Assumption $\mathcal H_m$ given in Definition \ref{defas}.      Then  for any $\tilde\eps>0$ the following estimates hold. 

\noindent (i) For $m=0$, it holds that 
\begin{eqnarray*}
	  \int_{\mathbb Z^3}	\abs{J_{m,4}}^{1/2}d\Sigma(k)\leq C\norm{f_0}_{  L^1_kL^2_v}.
\end{eqnarray*}

\noindent (ii) For $m=1$, it holds that
\begin{eqnarray*}
  \int_{\mathbb Z^3}	\abs{J_{m,4}}^{1/2}d\Sigma(k)\leq	 \tilde\varepsilon  \int_{\mathbb Z^3}\comi{k}^{\frac{s}{1+2s} }\left(   \int_0^1 \norm{  \Lambda_{\delta_1}\hat f_{m,\delta}(t,k)}_{L^2_v}^2d t\right)^{1/2}d \Sigma(k)\\
  +C_{\tilde\eps}\int_{\mathbb Z^3}\comi{k}^{\frac{s}{1+2s} }\left(   \int_0^1 \norm{  \Lambda_{\delta_1}\hat f_{m-1,\delta}(t,k)}_{L^2_v}^2d t\right)^{1/2}d \Sigma(k).
\end{eqnarray*}

\noindent (iii) For $m\geq 2$, it holds that
\begin{multline*}
  \int_{\mathbb Z^3}	\abs{J_{m,4}}^{1/2}d\Sigma(k)\leq  \tilde\varepsilon  \int_{\mathbb Z^3}\comi{k}^{\frac{s}{1+2s} }\left(   \int_0^1 \norm{  \Lambda_{\delta_1}\hat f_{m,\delta}(t,k)}_{L^2_v}^2d t\right)^{1/2}d \Sigma(k)\\
    +C_{\tilde \varepsilon}     \tilde C_0^{m-2}  \com{(m-1)!} ^{\frac{1+2s}{2s}}. 	
\end{multline*}
 \end{lemma}

\begin{proof} 
We write
	\begin{eqnarray*}
	\begin{aligned}
	\int_{\mathbb Z^3}	\abs{J_{m,4}}^{1/2}d\Sigma(k)&\leq 
	C \int_{\mathbb Z^3} \Big(\int_{0}^{1}\big|\big(
  mt^{-1} \Lambda_{\delta_1}\hat f_{m,\delta} ,\ \lambda^{\rm Wick}_k\Lambda_{\delta_1}   \hat f_{m,\delta}\Big)_{L_v^2} dt\Big)^{1/2}d\Sigma(k)\\
 &\quad+
	C \int_{\mathbb Z^3} \Big(\int_{0}^{1}\big|\big(
   [v\cdot k,\ \Lambda_{\delta_1}] \hat f_{m,\delta},\ \lambda^{\rm Wick}_k\Lambda_{\delta_1}   \hat f_{m,\delta}\Big)_{L_v^2} dt\Big)^{1/2}d\Sigma(k).
  	\end{aligned}
	\end{eqnarray*}
	Using \eqref{ube} gives
	\begin{multline*}
	 	\int_{\mathbb Z^3} \Big(\int_{0}^{1}\big|\big(
  mt^{-1} \Lambda_{\delta_1}\hat f_{m,\delta} ,\ \lambda^{\rm Wick}_k\Lambda_{\delta_1}   \hat f_{m,\delta}\Big)_{L_v^2} dt\Big)^{1/2}d\Sigma(k)\\
   \leq C \int_{\mathbb Z^3} \Big(\int_{0}^{1} mt^{-1}\norm{  \Lambda_{\delta_1}\hat f_{m,\delta}(t,k)}_{L_v^2} ^2dt\Big)^{1/2}d\Sigma(k).
	\end{multline*}
	It follows from \eqref{comtes} that  $[v\cdot k,\ \Lambda_{\delta_1}]\Lambda_{\delta_1}^{-1}$ is bounded on $L_v^2$ uniformly with respect to $k$ and  $\delta_1$. Thus, writing 
	$[v\cdot k,\ \Lambda_{\delta_1}]=[v\cdot k,\ \Lambda_{\delta_1}]\Lambda_{\delta_1}^{-1}\Lambda_{\delta_1}$, we have
	\begin{multline*}
		\int_{\mathbb Z^3} \Big(\int_{0}^{1}\big|\big(
  [v\cdot k,\ \Lambda_{\delta_1}] \hat f_{m,\delta},\ \lambda^{\rm Wick}_k\Lambda_{\delta_1}   \hat f_{m,\delta}\Big)_{L_v^2} dt\Big)^{1/2}d\Sigma(k)\\
  \leq C\int_{\mathbb Z^3} \Big(\int_{0}^{1}\norm{  \Lambda_{\delta_1}\hat f_{m,\delta}(t,k)}_{L_v^2} ^2dt\Big)^{1/2}d\Sigma(k).
	\end{multline*}
	Combining these estimates yields 
 that 
	\begin{equation}\label{jmest}
		\int_{\mathbb Z^3}	\abs{J_{m,4}}^{1/2}d\Sigma(k)\leq  C\int_{\mathbb Z^3} \Big(\int_{0}^{1}mt^{-1}\norm{  \Lambda_{\delta_1}\hat f_{m,\delta}(t,k)}_{L_v^2} ^2dt\Big)^{1/2}d\Sigma(k) 	\end{equation}
	for $m\geq 1 $,  
	and that
\begin{equation*}
		\int_{\mathbb Z^3}	\abs{J_{m,4}}^{1/2}d\Sigma(k)\leq  C\int_{\mathbb Z^3} \Big(\int_{0}^{1} \norm{  \hat f_{m,\delta}}_{L_v^2} ^2dt\Big)^{\frac{1}{2}}d\Sigma(k)\leq C\int_{\mathbb Z^3} \Big(\int_{0}^{1} \norm{(a^{1/2})^w\hat f}_{L_v^2} ^2dt\Big)^{\frac{1}{2}}d\Sigma(k)
\end{equation*}
	for $m=0$.   So, from \eqref{+1234}, the assertion for $m=0$ follows. It remains to consider the case of $m\geq 1$.
	We use  Young inequality  with an arbitrary parameter $\vartheta>0$ that 
    \begin{eqnarray*}
    	 1 \leq \vartheta  \comi{k}^{\frac{2s}{1+2s}}+\vartheta^{-\frac{1+s}{s}} \comi{k}^{-2+\frac{2s}{1+2s}},
    \end{eqnarray*} 
      to obtain, choosing  in particular  $\vartheta=\tilde\eps^2t/m$ with arbitrarily small number $\tilde \eps>0$ and recalling $\varsigma=\frac{1+2s}{2s}$, 
     \begin{equation*}%\label{inter}
   \begin{split}
   m t^{-1} \leq \tilde\varepsilon^2 \comi{k}^{\frac{2s}{1+2s}} + \tilde\eps^{-\frac{2(1+s)}{s}} m^{\frac{1+2s}{s}} t^{-2\varsigma } \comi{k}^{-2+\frac{2s}{1+2s} }.
   \end{split}
   \end{equation*}
   This  with \eqref{jmest} yield
\begin{multline} 
%\begin{split}
\int_{\mathbb Z^3}	\abs{J_{m,4}}^{1/2}d\Sigma(k)\leq \tilde\varepsilon  \int_{\mathbb Z^3}\comi{k}^{\frac{s}{1+2s} }\left(   \int_0^1 \norm{  \Lambda_{\delta_1}\hat f_{m,\delta}(t,k)}_{L^2_v}^2d t\right)^{1/2}d \Sigma(k)\\
   +C_{\tilde \varepsilon} m^{\frac{1+2s}{2s}}  \int_{\mathbb Z^3}\comi{k}^{\frac{s}{1+2s} }\left(   \int_0^1\comi{k}^{-2}t^{-2 \varsigma  }\norm{  \Lambda_{\delta_1}\hat f_{m,\delta}(t,k)}_{L^2_v}^2d t\right)^{1/2}d \Sigma(k). \label{interp1}   
%\end{split}
\end{multline}
Observe for $m\geq 1,$
 \begin{eqnarray*}
 	\comi{k}^{-2}t^{-2 \varsigma  }\norm{  \Lambda_{\delta_1}\hat f_{m,\delta}(t,k)}_{L^2_v}^2= \norm{  \Lambda_{\delta_1}\hat f_{m-1,\delta}(t,k)}_{L^2_v}^2\leq t^{2\varsigma(m-1)}\comi k^{2(m-1)}\norm{\hat f(t,k)}_{L_v^2}^2
 \end{eqnarray*} 
due to the definition \eqref{fmdel} of $\hat f_{m,\delta}$. Thus we combine the above two inequalities to conclude
\begin{multline*}
		 \int_{\mathbb Z^3}	\abs{J_{m,4}}^{1/2}d\Sigma(k)\leq  \tilde\varepsilon  \int_{\mathbb Z^3}\comi{k}^{\frac{s}{1+2s} }\left(   \int_0^1 \norm{  \Lambda_{\delta_1}\hat f_{m,\delta}(t,k)}_{L^2_v}^2d t\right)^{1/2}d \Sigma(k)\\ +C_{\tilde \varepsilon} \int_{\mathbb Z^3}\comi{k}^{\frac{s}{1+2s} }\left(   \int_0^1 \norm{  \Lambda_{\delta_1}\hat f_{m-1,\delta}(t,k)}_{L^2_v}^2d t\right)^{1/2}d \Sigma(k)
	\end{multline*}
for $m=1$, and meanwhile  
\begin{eqnarray*}
	\begin{aligned}
		& \int_{\mathbb Z^3}	\abs{J_{m,4}}^{1/2}d\Sigma(k)\leq  \tilde\varepsilon  \int_{\mathbb Z^3}\comi{k}^{\frac{s}{1+2s} }\left(   \int_0^1 \norm{  \Lambda_{\delta_1}\hat f_{m,\delta}(t,k)}_{L^2_v}^2d t\right)^{1/2}d \Sigma(k)\\
   &\qquad\qquad\qquad\qquad+C_{\tilde \varepsilon} m^{\frac{1+2s}{2s}}   \tilde C_0^{m-2}  \com{(m-2)!} ^{\frac{1+2s}{2s}}\\
   &\leq \tilde\varepsilon  \int_{\mathbb Z^3}\comi{k}^{\frac{s}{1+2s} }\left(   \int_0^1 \norm{  \Lambda_{\delta_1}\hat f_{m,\delta}(t,k)}_{L^2_v}^2d t\right)^{1/2}d \Sigma(k) +C_{\tilde \varepsilon}    \tilde C_0^{m-2}  \com{(m-1)!} ^{\frac{1+2s}{2s}} 	
   \end{aligned}
\end{eqnarray*} 
for $m\geq 2$, where we have used  the  induction assumption \eqref{inassump} when $m\geq2$.
Thus, we have proved all the assertions in Lemma \ref{lemindc}, completing the proof.  
\end{proof}

\begin{proof}
	[Ending the proof of Proposition \ref{prpsub}.] 

(a) We   prove the first assertion in Proposition \ref{prpsub} under the estimate     \eqref{+1234}.   
Combining the estimates for $m=0$ in  Lemmas \ref{sd}-\ref{lemindc}  with \eqref{sumj} gives 
\begin{multline*}
		  \int_{\mathbb Z^3} \comi k^{\frac{s}{1+2s}} 	\Big(\int_{t_1}^{t_2}\norm{\Lambda_{\delta_1}   \hat f_{0,\delta}(t,k)}_{L_v^2} ^2dt\Big)^{1/2} d\Sigma(k)\leq C\norm{f_0}_{  L^1_kL^2_v}+C  \int_{\mathbb Z^3}  \norm{ \hat f_{0,\delta} (t_1,k)}_{L_v^2}   d\Sigma(k) \\
		 + C\int_{\mathbb Z^3}  \norm{      \hat f_{0,\delta} (t_2,k)}_{L_v^2}   d\Sigma(k)+C \int_{\mathbb Z^3} \left(\int_0^1    \norm{(a^{1/2})^w\hat f_{0,\delta}(t,k) }^2_{L_v^2}  d t\right)^{1/2}d \Sigma(k)\\
		 \leq C\norm{f_0}_{  L^1_kL^2_v}+C  \int_{\mathbb Z^3}  \sup_{0<t\leq 1}\norm{ \hat f (t,k)}_{L_v^2}   d\Sigma(k)  
		 +C \int_{\mathbb Z^3} \left(\int_0^1    \norm{(a^{1/2})^w\hat f (t,k) }^2_{L_v^2}  d t\right)^{1/2}d \Sigma(k),
%\leq C\eps_0,
		\end{multline*}
that further can be bounded by $C\norm{f_0}_{  L^1_kL^2_v}$, where we have used the estimate \eqref{+1234} by observing
\begin{eqnarray*}
	\norm{ \hat f_{0,\delta}}_{L_v^2}\leq \norm{ \hat f}_{L_v^2},\quad \norm{(a^{1/2})^w\hat f_{0,\delta}}_{L_v^2}\leq \norm{(a^{1/2})^w\hat f}_{L_v^2}.
\end{eqnarray*}
Letting  $t_1\rightarrow 0$ and    $t_2\rightarrow 1 $, and further
letting  $\delta_1,\delta\rightarrow 0$, we conclude by the Fatou Lemma that 
\begin{equation}\label{frg}
	 \int_{\mathbb Z^3} \comi k^{\frac{s}{1+2s}} 	\Big(\int_{0}^{1}\norm{ \hat f(t,k)}_{L_v^2} ^2dt\Big)^{1/2} d\Sigma(k) \leq  C \norm{f_0}_{  L^1_kL^2_v}.
\end{equation}	
We have proved the assertion (i) in Proposition \ref{prpsub}.

\medskip
  (b)  Next, we deal with the case of $m=1$.   Combining the estimates for $m=1$ in   Lemmas \ref{sd}-\ref{lemindc} with \eqref{sumj}, we get, letting the constant  $\tilde\eps$ in Lemma \ref{lemindc} be small enough, 
\begin{align*}
&\int_{\mathbb Z^3} \comi k^{\frac{s}{1+2s}} 	\Big(\int_{t_1}^{t_2}\norm{\Lambda_{\delta_1}   \hat f_{1,\delta}(t,k)}_{L_v^2} ^2dt\Big)^{1/2} d\Sigma(k) 
		  \leq C  \int_{\mathbb Z^3}\comi{k}^{\frac{s}{1+2s} }\left(   \int_0^1 \norm{  \Lambda_{\delta_1}\hat f_{0,\delta}}_{L^2_v}^2d t\right)^{\frac{1}{2}}d \Sigma(k)\\
&\quad+ C\int_{\mathbb Z^3} \sup_{0<t\leq 1} \norm{      \hat f_{1,\delta} (t,k)}_{L_v^2}   d\Sigma(k)
		+C \int_{\mathbb Z^3} \left(\int_0^1    \norm{(a^{1/2})^w\hat f_{1,\delta}(t,k) }^2_{L_v^2}  d t\right)^{1/2}d \Sigma(k)\\
&\leq C\norm{f_0}_{  L^1_kL^2_v}+ C\int_{\mathbb Z^3} \sup_{0<t\leq 1} \norm{      \hat f_{1,\delta} (t,k)}_{L_v^2}   d\Sigma(k)\\
&\quad+C \int_{\mathbb Z^3} \left(\int_0^1    \norm{(a^{1/2})^w\hat f_{1,\delta}(t,k) }^2_{L_v^2}  d t\right)^{1/2}d \Sigma(k),
\end{align*}
where in the last inequality we have used   \eqref{frg} since $\norm{  \Lambda_{\delta_1}\hat f_{0,\delta}}_{L^2_v}\leq \norm{  \hat f}_{L^2_v}$.   
Thus, letting $t_1\rightarrow 0$ and $t_2\rightarrow 1 $, and further
letting  $\delta_1 \rightarrow 0$, we obtain the assertion (ii) in Proposition \ref{prpsub}.

\medskip
 
 (c) It remains to treat the case of $m\geq 2$.  We combine the estimates for $m\geq 2$ in  
 Lemmas \ref{sd}-\ref{lemindc} with the estimate \eqref{sumj},  and then let the constant  $\tilde \eps$ in Lemma \ref{lemindc} be small enough; this gives that 
%\begin{eqnarray*}
\begin{align*}
				& \int_{\mathbb Z^3} \comi k^{\frac{s}{1+2s}} 	\Big(\int_{t_1}^{t_2}\norm{\Lambda_{\delta_1}   \hat f_{m,\delta}(t,k)}_{L_v^2} ^2dt\Big)^{1/2} d\Sigma(k)  \leq  C \tilde C_0^{m-2}  \com{(m-1)!} ^{\frac{1+2s}{2s}} 
\\ &\quad\qquad+C \int_{\mathbb Z^3} \sup_{0<t\leq 1}     \norm{ \hat f_{m,\delta}(t,k )}_{L^2_v}    d \Sigma(k) 
				 +C \int_{\mathbb Z^3} \left(\int_0^1    \norm{(a^{1/2})^w\hat f_{m,\delta}(t,k) }^2_{L_v^2}  d t\right)^{1/2}d \Sigma(k).
\end{align*}
%\end{eqnarray*}
 Similarly as above, letting $t_1\rightarrow 0$ and   $t_2\rightarrow 1$  and then letting $\delta_1\rightarrow 0$, we get  the assertion (iii) as desired in Proposition \ref{prpsub}.  The proof of Proposition \ref{prpsub} is thus completed. 
  \end{proof}

\subsection{Energy estimates for regularized solutions}\label{subeng}   This part is devoted to   proving the following energy estimate for $f_{m,\delta}$. 

\begin{proposition}\label{prpeng}
Let $m\in\mathbb Z_+$  with  $m\geq 1$ and let   $f(t,x,v)$ satisfy the Assumption $\mathcal H_m$ given in Definition \ref{defas}.   The the following estimates hold.

\noindent (i) For $m=1$, it holds that
 \begin{eqnarray*}
	\begin{aligned}
\int_{\mathbb Z^3} \sup_{0<t\leq 1}\norm{ \hat f_{m,\delta}(t,k)}_{L_v^2} d\Sigma(k) +\int_{\mathbb Z^3}\Big(\int_{0}^{1} \norm{(a^{1/2})^w  \hat f_{m,\delta}}_{L_v^2}^2 dt\Big)^{\frac{1}{2}}
 	d\Sigma(k) \leq C\norm{f_0}_{  L^1_kL^2_v}.
 \end{aligned}
\end{eqnarray*}

	\noindent (ii)
If    $m\geq 2 $, it holds that  
%\begin{eqnarray*}
\begin{equation*}
  		\int_{\mathbb Z^3} \sup_{0<t\leq 1}\norm{ \hat f_{m,\delta}(t,k)}_{L_v^2} d\Sigma(k) +\int_{\mathbb Z^3}\Big(\int_{0}^{1} \norm{(a^{1/2})^w  \hat f_{m,\delta}}_{L_v^2}^2 dt\Big)^{\frac{1}{2}}
 	d\Sigma(k)     
    \leq   C  \tilde C_0^{m-2}  \com{(m-1)!} ^{\frac{1+2s}{2s}},
\end{equation*}
%\end{eqnarray*}
where $\tilde C_0$ is the constant in the induction assumption \eqref{inassump}.
\end{proposition}

\begin{proof}
The argument of the proof is quite similar as in the previous Subsection \ref{secsub}, so we only sketch it for brevity.  
	
Taking the $L_v^2$-product on both sides of \eqref{equreg} with $\Lambda_{\delta_1} \hat f_{m,\delta}$ and then integrating the real parts of  the resulting equation over $[t_1,t_2]$ for any $0<t_1<t_2\leq 1$, we have 
 \begin{equation*} 
 \begin{aligned}
 	 &\frac{1}{2}\norm{\Lambda_{\delta_1} \hat f_{m,\delta}(t_2,k)}_{L_v^2}^2+\int_{t_1}^{t_2} {\rm Re}\big(\Lambda_{\delta_1} \mathcal L \hat f_{m,\delta},\ \Lambda_{\delta_1}   \hat f_{m,\delta}\big)_{L_v^2} dt\\
 & 	 = \frac{1}{2}\norm{\Lambda_{\delta_1} \hat f_{m,\delta}(t_1,k)}_{L_v^2}^2 +
 	\int_{t_1}^{t_2} {\rm Re}\Big(\Lambda_{\delta_1}t^{\varsigma m}  (1+\delta\abs k^2)^{-1/2} \comi{k}^m\hat\Gamma \big({\hat f},~{\hat f}\big),\ \Lambda_{\delta_1}   \hat f_{m,\delta}\Big)_{L_v^2} dt\\
  &\quad 	+
 	\int_{t_1}^{t_2} {\rm Re}\Big(
 \varsigma mt^{-1} \Lambda_{\delta_1}\hat f_{m,\delta}+i[v\cdot k,\ \Lambda_{\delta_1}] \hat f_{m,\delta},\ \Lambda_{\delta_1}   \hat f_{m,\delta}\Big)_{L_v^2} dt.
 \end{aligned}
 \end{equation*}
 This, along with
\begin{multline*}
\norm{(a^{1/2})^w\Lambda_{\delta_1} \hat f_{m,\delta}}_{L_v^2}^2\leq  C   \big( \mathcal L  \Lambda_{\delta_1}\hat f_{m,\delta},\ \Lambda_{\delta_1}   \hat f_{m,\delta}\big)_{L_v^2}   +C\norm{\Lambda_{\delta_1} \hat f_{m,\delta}}_{L_v^2}^2\\  \leq  C  \Big(	{\rm Re}\big(\Lambda_{\delta_1} \mathcal L \hat f_{m,\delta},\ \Lambda_{\delta_1}   \hat f_{m,\delta}\big)_{L_v^2} +    	{\rm Re}\big([\mathcal L ,\ \Lambda_{\delta_1}] \hat f_{m,\delta},\ \Lambda_{\delta_1}   \hat f_{m,\delta}\big)_{L_v^2} \Big) +C\norm{\Lambda_{\delta_1} \hat f_{m,\delta}}_{L_v^2}^2
\end{multline*}
due to \eqref{rela}, yield that
 \begin{equation*} 
 \begin{aligned}
 & \norm{\Lambda_{\delta_1} \hat f_{m,\delta}(t_2,k)}_{L_v^2} d\Sigma(k) + \Big(\int_{t_1}^{t_2} \norm{(a^{1/2})^w\Lambda_{\delta_1} \hat f_{m,\delta}}_{L_v^2}^2 dt\Big)^{1/2}
 	   \\
  &  \leq  C  \norm{\Lambda_{\delta_1}\hat f_{m,\delta}(t_1,k)}_{L_v^2}      +C \Big(\int_{0}^{1} \norm{ \Lambda_{\delta_1} \hat f_{m,\delta}}_{L_v^2}^2 dt\Big)^{\frac{1}{2}}  + C \Big(\int_{0}^{1} \Big|\big([\mathcal L ,\ \Lambda_{\delta_1}] \hat f_{m,\delta},\ \Lambda_{\delta_1}   \hat f_{m,\delta}\big)_{L_v^2} \Big| dt\Big)^{\frac{1}{2}}
  \\
  &\quad +C \bigg[ 
 	\int_{0}^1\Big|\Big(\Lambda_{\delta_1}t^{\varsigma m}  (1+\delta\abs k^2)^{-1/2} \comi{k}^m\hat\Gamma \big({\hat f},~{\hat f}\big),\ \Lambda_{\delta_1}   \hat f_{m,\delta}\Big)_{L_v^2} \Big|dt\bigg]^{1/2} \\
 & \qquad 	+C \bigg[
  	\int_{0}^1 \Big|\Big( \varsigma mt^{-1} \Lambda_{\delta_1}\hat f_{m,\delta}+i[v\cdot k,\ \Lambda_{\delta_1}] \hat f_{m,\delta},\ \Lambda_{\delta_1}   \hat f_{m,\delta}\Big)_{L_v^2} \Big| dt\bigg]^{1/ 2}.
 \end{aligned}
 \end{equation*}
 Letting $t_1\rightarrow 0$, observing that  the first term on the left side vanishes because of  \eqref{vans},  and then taking the supremum for $0<t_2\leq 1$,  we conclude that after integrating for $k\in\mathbb Z^3$,
%\begin{equation}
\begin{align}
 &	\int_{\mathbb Z^3} \sup_{0<t\leq 1}\norm{\Lambda_{\delta_1} \hat f_{m,\delta}(t,k)}_{L_v^2} d\Sigma(k) +\int_{\mathbb Z^3}\Big(\int_{0}^{1} \norm{(a^{1/2})^w\Lambda_{\delta_1} \hat f_{m,\delta}}_{L_v^2}^2 dt\Big)^{1/2}
 	d\Sigma(k)  \notag \\
  &  \leq   C\int_{\mathbb Z^3}\Big(\int_{0}^{1} \norm{ \Lambda_{\delta_1} \hat f_{m,\delta}}_{L_v^2}^2 dt\Big)^{1/2}
 	d\Sigma(k)  + C\int_{\mathbb Z^3}\Big(\int_{0}^{1} \Big|\big([\mathcal L ,\ \Lambda_{\delta_1}] \hat f_{m,\delta},\ \Lambda_{\delta_1}   \hat f_{m,\delta}\big)_{L_v^2} \Big| dt\Big)^{1/2}
 	d\Sigma(k)\notag \\
  &\quad +C \int_{\mathbb Z^3} \bigg[ 
 	\int_{0}^1\Big|\Big(\Lambda_{\delta_1}t^{\varsigma m}  (1+\delta\abs k^2)^{-1/2} \comi{k}^m\hat\Gamma \big({\hat f},~{\hat f}\big),\ \Lambda_{\delta_1}   \hat f_{m,\delta}\Big)_{L_v^2} \Big|dt\bigg]^{1/2}d\Sigma(k)\notag \\
 & \qquad 	+C\int_{\mathbb Z^3}\bigg[
  	\int_{0}^1 \Big|\Big( \varsigma mt^{-1} \Lambda_{\delta_1}\hat f_{m,\delta}+i[v\cdot k,\ \Lambda_{\delta_1}] \hat f_{m,\delta},\ \Lambda_{\delta_1}   \hat f_{m,\delta}\Big)_{L_v^2} \Big| dt\bigg]^{1/ 2}d\Sigma(k).\label{t1t2}
 \end{align}
 %\end{equation} 
 
 \medskip
 (a) We consider the case when $m\geq 2$.   Just repeating the proof of Lemma \ref{lemj3}, we have that for any $\eps>0$,
 \begin{eqnarray*}
	\begin{aligned}
 & \int_{\mathbb Z^3} \bigg[ 
 	\int_{0}^1\Big|\Big(\Lambda_{\delta_1}t^{\varsigma m}  (1+\delta\abs k^2)^{-1/2} \comi{k}^m\hat\Gamma \big({\hat f},~{\hat f}\big),\ \Lambda_{\delta_1}   \hat f_{m,\delta}\Big)_{L_v^2} \Big|dt\bigg]^{1/ 2}d\Sigma(k)\\
&\leq  \left( \eps +  C \eps^{-1}\varepsilon_0   \right)\int_{\mathbb Z^3} \left(\int_0^1    \norm{(a^{1/2})^w\hat f_{m,\delta}(t,k) }^2_{L_v^2}  d t\right)^{1/2}d \Sigma(k)\\
&\quad+C \eps^{-1}\varepsilon_0 \int_{\mathbb Z^3} \sup_{0<t\leq 1}     \norm{ \hat f_{m,\delta}(t,k )}_{L^2_v}    d \Sigma(k)+C \eps^{-1} \tilde C_0^{m-2}  \com{(m-1)!} ^{\frac{1+2s}{2s}}.
\end{aligned}	
\end{eqnarray*}
Following the same argument as for proving Lemma \ref{lemindc} gives
\begin{multline*}
 \int_{\mathbb Z^3}\Big(\int_{0}^{1} \norm{ \Lambda_{\delta_1} \hat f_{m,\delta}}_{L_v^2}^2 dt\Big)^{1/2}
 	d\Sigma(k) \\
 	+
\int_{\mathbb Z^3}\bigg[
  	\int_{0}^1 \Big|\Big( \varsigma mt^{-1} \Lambda_{\delta_1}\hat f_{m,\delta}+i[v\cdot k,\ \Lambda_{\delta_1}] \hat f_{m,\delta},\ \Lambda_{\delta_1}   \hat f_{m,\delta}\Big)_{L_v^2} \Big| dt\bigg]^{1/ 2}d\Sigma(k)\\
  	\leq   \varepsilon  \int_{\mathbb Z^3}\comi{k}^{\frac{s}{1+2s} }\Big(   \int_0^1 \norm{  \Lambda_{\delta_1}\hat f_{m,\delta}(t,k)}_{L^2_v}^2d t\Big)^{1/2}d \Sigma(k)+C_{ \eps}\tilde C_0^{m-2}  \com{(m-1)!} ^{\frac{1+2s}{2s}}, 	
\end{multline*}
for   any $ \eps>0.$ Now, we substitute these inequalities into \eqref{t1t2} to obtain that for any $\eps  >0,$
\begin{eqnarray*}
	\begin{aligned}
 &		\int_{\mathbb Z^3} \sup_{0<t\leq 1}\norm{\Lambda_{\delta_1} \hat f_{m,\delta}(t,k)}_{L_v^2} d\Sigma(k) +\int_{\mathbb Z^3}\Big(\int_{0}^{1} \norm{(a^{1/2})^w\Lambda_{\delta_1} \hat f_{m,\delta}}_{L_v^2}^2 dt\Big)^{1/2}
 	d\Sigma(k)    \\
  &  \leq   \left( \eps +  C \eps^{-1}\varepsilon_0   \right)\int_{\mathbb Z^3} \left(\int_0^1    \norm{(a^{1/2})^w\hat f_{m,\delta}(t,k) }^2_{L_v^2}  d t\right)^{1/2}d \Sigma(k)\\
&\quad+C \eps^{-1}\varepsilon_0 \int_{\mathbb Z^3} \sup_{0<t\leq 1}     \norm{ \hat f_{m,\delta}(t,k )}_{L^2_v}    d \Sigma(k)+C_\eps  \tilde C_0^{m-2}  \com{(m-1)!} ^{\frac{1+2s}{2s}}\\
 & \quad 	+ \varepsilon  \int_{\mathbb Z^3}\comi{k}^{\frac{s}{1+2s} }\Big(   \int_0^1 \norm{  \Lambda_{\delta_1}\hat f_{m,\delta}(t,k)}_{L^2_v}^2d t\Big)^{1/2}d \Sigma(k)\\
    &\quad +C	\int_{\mathbb Z^3}\Big(\int_{0}^{1} \Big|\big([\mathcal L ,\ \Lambda_{\delta_1}] \hat f_{m,\delta},\ \Lambda_{\delta_1}   \hat f_{m,\delta}\big)_{L_v^2} \Big| dt\Big)^{1/2}
 	d\Sigma(k).
 \end{aligned}
\end{eqnarray*}
This with the assertion (iii) in Proposition \ref{prpsub} give that  for any $\eps>0,$ 
\begin{eqnarray*}
	\begin{aligned}
 &		\int_{\mathbb Z^3} \sup_{0<t\leq 1}\norm{\Lambda_{\delta_1} \hat f_{m,\delta}(t,k)}_{L_v^2} d\Sigma(k) +\int_{\mathbb Z^3}\Big(\int_{0}^{1} \norm{(a^{1/2})^w\Lambda_{\delta_1} \hat f_{m,\delta}}_{L_v^2}^2 dt\Big)^{1/2}
 	d\Sigma(k)    \\
  &  \leq   C \inner{\eps+\eps^{-1}\varepsilon_0    }\int_{\mathbb Z^3} \sup_{0<t\leq 1}     \norm{ \hat f_{m,\delta}(t,k )}_{L^2_v}    d \Sigma(k) \\
&\quad+C\left( \eps +    \eps^{-1}\varepsilon_0  \right)\int_{\mathbb Z^3} \left(\int_0^1    \norm{(a^{1/2})^w\hat f_{m,\delta}(t,k) }^2_{L_v^2}  d t\right)^{1/2}d \Sigma(k)+C_ \eps  \tilde C_0^{m-2}  \com{(m-1)!} ^{\frac{1+2s}{2s}} \\
    &\quad + C	\int_{\mathbb Z^3}\Big(\int_{0}^{1} \Big|\big([\mathcal L ,\ \Lambda_{\delta_1}] \hat f_{m,\delta},\ \Lambda_{\delta_1}   \hat f_{m,\delta}\big)_{L_v^2} \Big| dt\Big)^{1/2}d\Sigma(k).
 \end{aligned}
\end{eqnarray*}
Thus, letting  $\delta_1\rightarrow 0$  first  and then choosing  $\eps>0$ suitably small, we obtain the assertion (ii) in 
 Proposition \ref{prpeng},   provided  that   $\eps_0$ is  small enough and that 
\begin{equation}\label{lim}
	\lim_{\delta_1\rightarrow 0}\int_{\mathbb Z^3}\Big(\int_{0}^{1} \Big|\big([\mathcal L ,\ \Lambda_{\delta_1}] \hat f_{m,\delta},\ \Lambda_{\delta_1}   \hat f_{m,\delta}\big)_{L_v^2} \Big| dt\Big)^{1/2}d\Sigma(k)=0.
\end{equation}
It remains to prove \eqref{lim}. 
To do so we write
\begin{eqnarray*}
	 [\mathcal L ,\ \Lambda_{\delta_1}]= [\mathcal L ,\ \Lambda_{\delta_1}-{\rm Id}]= \mathcal L  \inner{\Lambda_{\delta_1}-{\rm Id}} - \inner{\Lambda_{\delta_1}-{\rm Id} }\mathcal L 
\end{eqnarray*} 
with $\textrm{Id}$    the identity operator on $L_v^2$.  Moreover,  
applying Corollary \ref{corupp} as well as the assertion (iii) in Proposition \ref{estaa} gives
\begin{eqnarray*}
	\begin{aligned}
	 & \Big|\big(\mathcal L  \inner{\Lambda_{\delta_1}-{\rm Id}} \hat f_{m,\delta},\ \Lambda_{\delta_1}   \hat f_{m,\delta}\big)_{L_v^2} \Big|  
	  \leq C   \norm{(a^{1/2})^w \inner{\Lambda_{\delta_1}-{\rm Id}}  \hat f_{m,\delta}}_{L_v^2}  \norm{(a^{1/2})^w\Lambda_{\delta_1} \hat f_{m,\delta}}_{L_v^2} \\
	& \leq C   \norm{\inner{\Lambda_{\delta_1}-{\rm Id}} (a^{1/2})^w   \hat f_{m,\delta}}_{L_v^2}  \norm{(a^{1/2})^w \hat f_{m,\delta}}_{L_v^2}  + C\delta_1^{1/2}    \norm{(a^{1/2})^w \hat f_{m,\delta}}_{L_v^2}^2
	 	\end{aligned}
\end{eqnarray*}
and   
\begin{eqnarray*}
	\begin{aligned}
	  & \Big|\big(\inner{\Lambda_{\delta_1}-{\rm Id}}  \mathcal L  \hat f_{m,\delta},\ \Lambda_{\delta_1}   \hat f_{m,\delta}\big)_{L_v^2} \Big|  
	\leq C   \norm{(a^{1/2})^w \inner{\Lambda_{\delta_1}^*-{\rm Id}}  \Lambda_{\delta_1} \hat f_{m,\delta}}_{L_v^2}  \norm{(a^{1/2})^w\hat f_{m,\delta}}_{L_v^2} \\
	& \leq C   \norm{ \inner{\Lambda_{\delta_1}^*-{\rm Id}}  (a^{1/2})^w  \hat f_{m,\delta}}_{L_v^2}  \norm{(a^{1/2})^w\hat f_{m,\delta}}_{L_v^2}+ C\delta_1^{1/2}    \norm{(a^{1/2})^w \hat f_{m,\delta}}_{L_v^2}^2,    
	 	\end{aligned}
\end{eqnarray*}
 where  we have used \eqref{smacomest}.  Thus, we conclude
 %\begin{equation}
	\begin{align}
	& \Big(\int_{0}^{1} \Big|\big([\mathcal L ,\ \Lambda_{\delta_1}] \hat f_{m,\delta},\ \Lambda_{\delta_1}   \hat f_{m,\delta}\big)_{L_v^2} \Big| dt\Big)^{1/2} \notag \\
	&\leq C  	 \Big(\int_{0}^{1}  \norm{\inner{\Lambda_{\delta_1}-{\rm Id}}(a^{1/2})^w\hat f_{m,\delta}}_{L_v^2} ^2 dt\Big)^{1/4} \Big(\int_{0}^{1}  \norm{(a^{1/2})^w\hat f_{m,\delta}}_{L_v^2} ^2 dt\Big)^{1/4}\notag\\
	& \quad  + C  	 \Big(\int_{0}^{1}  \norm{\inner{\Lambda_{\delta_1}^*-{\rm Id}}(a^{1/2})^w\hat f_{m,\delta}}_{L_v^2} ^2 dt\Big)^{1/4} \Big(\int_{0}^{1}  \norm{(a^{1/2})^w\hat f_{m,\delta}}_{L_v^2} ^2 dt\Big)^{1/4}\notag\\
	& \quad  + C\delta_1^{1/4}  	  \Big(\int_{0}^{1}  \norm{(a^{1/2})^w\hat f_{m,\delta}}_{L_v^2} ^2 dt\Big)^{1/2}. 	\label{cl}
\end{align}
%\end{equation}
We claim  that
\begin{equation}\label{dct}
\begin{aligned}
	 \lim_{\delta_1\rightarrow 0} \int_{0}^{1}  \norm{\inner{\Lambda_{\delta_1}-{\rm Id}}(a^{1/2})^w\hat f_{m,\delta}}_{L_v^2} ^2 dt= \lim_{\delta_1\rightarrow 0} \int_{0}^{1}  \norm{\inner{\Lambda_{\delta_1}^*-{\rm Id}}(a^{1/2})^w\hat f_{m,\delta}}_{L_v^2} ^2 dt  =0.
	\end{aligned}
\end{equation}
In fact, by Fubini theorem, 
\begin{multline*}
\int_ {[0,1]\times \mathbb R^3} \big|\inner{\Lambda_{\delta_1}-{\rm Id}}(a^{1/2})^w\hat f_{m,\delta} \big|^2  dtdv=	\int_{0}^{1}  \norm{\inner{\Lambda_{\delta_1}-{\rm Id}}(a^{1/2})^w\hat f_{m,\delta}}_{L_v^2} ^2    dt \\
	  \leq   \int_{0}^{1}  \norm{(a^{1/2})^w\hat f_{m,\delta}}_{L_v^2} ^2 dt = \int_{[0,1]\times\mathbb R^3}  \big|(a^{1/2})^w\hat f_{m,\delta}\big|^2dt dv   \leq C_\delta
\end{multline*}
with $C_\delta$ depending on $\delta$ but independent of $\delta_1$, where in the last inequality we have used  \eqref{L2conv}. It then follows from the Dominated Convergence Theorem    
that 
\begin{eqnarray*}
	 \lim_{\delta_1\rightarrow 0} \int_{0}^{1}  \norm{\inner{\Lambda_{\delta_1}-{\rm Id}}(a^{1/2})^w\hat f_{m,\delta}}_{L_v^2} ^2 dt= \int_ {[0,1]\times \mathbb R^3} \lim_{\delta_1\rightarrow 0}\big|\inner{\Lambda_{\delta_1}-{\rm Id}}(a^{1/2})^w\hat f_{m,\delta} \big|^2  dtdv =0.
\end{eqnarray*}
It is likewise for $\inner{\Lambda_{\delta_1}^*-{\rm Id}}$. Thus we have proved \eqref{dct}.    Combining \eqref{cl} and \eqref{dct}   yields
\begin{eqnarray*}
	\lim_{\delta_1\rightarrow 0}\Big(\int_{0}^{1} \Big|\big([\mathcal L ,\ \Lambda_{\delta_1}] \hat f_{m,\delta},\ \Lambda_{\delta_1}   \hat f_{m,\delta}\big)_{L_v^2} \Big| dt\Big)^{1/2}=0,
\end{eqnarray*}
which along with the Dominated Convergence Theorem give the desired estimate \eqref{lim}  by observing that 
\begin{equation*}
	\int_{\mathbb Z^3} \Big(\int_{0}^{1} \Big|\big([\mathcal L , \Lambda_{\delta_1}] \hat f_{m,\delta}, \Lambda_{\delta_1}   \hat f_{m,\delta}\big)_{L_v^2} \Big| dt\Big)^{\frac{1}{2}} d\Sigma(k) \leq  C	\int_{\mathbb Z^3}\Big(\int_{0}^{1}  \norm{(a^{1/2})^w\hat f_{m,\delta}}_{L_v^2} ^2 dt\Big)^{\frac{1}{2}} d\Sigma(k)\leq C_\delta.
\end{equation*}
Therefore, the assertion (ii)  in Proposition \ref{prpeng} is proved. 

\medskip
(b) It remains to consider the case when $m=1$. In fact, repeating the proof of  the assertion (ii)  for $m=1$ in Lemma \ref{lemj3} gives that for any $\eps>0,$
 \begin{eqnarray*}
	\begin{aligned}
 & \int_{\mathbb Z^3} \bigg[ 
 	\int_{0}^1\Big|\Big(\Lambda_{\delta_1}t^{\varsigma m}  (1+\delta\abs k^2)^{-1/2} \comi{k}^m\hat\Gamma \big({\hat f},~{\hat f}\big),\ \Lambda_{\delta_1}   \hat f_{m,\delta}\Big)_{L_v^2} \Big|dt\bigg]^{1/ 2}d\Sigma(k)\\
&\leq  \left( \eps +  C \eps^{-1}\varepsilon_0   \right)\int_{\mathbb Z^3} \left(\int_0^1    \norm{(a^{1/2})^w\hat f_{m,\delta}(t,k) }^2_{L_v^2}  d t\right)^{1/2}d \Sigma(k)\\
&\quad+C \eps^{-1}\varepsilon_0 \int_{\mathbb Z^3} \sup_{0<t\leq 1}     \norm{ \hat f_{m,\delta}(t,k )}_{L^2_v}    d \Sigma(k).
\end{aligned}	
\end{eqnarray*}
Similarly for showing the assertion (ii) for $m=1$ in 
  Lemma \ref{lemindc} and by using the fact that $\norm{  \Lambda_{\delta_1}\hat f_{0,\delta}}_{L^2_v}\leq \norm{\hat f}_{L^2_v}^2$, we have
%\begin{eqnarray*}
\begin{align*}
& \int_{\mathbb Z^3}\Big(\int_{0}^{1} \norm{ \Lambda_{\delta_1} \hat f_{m,\delta}}_{L_v^2}^2 dt\Big)^{1/2}
 	d\Sigma(k) \\
 &\qquad\qquad	+
\int_{\mathbb Z^3}\bigg[
  	\int_{0}^1 \Big|\Big( \varsigma mt^{-1} \Lambda_{\delta_1}\hat f_{m,\delta}+i[v\cdot k,\ \Lambda_{\delta_1}] \hat f_{m,\delta},\ \Lambda_{\delta_1}   \hat f_{m,\delta}\Big)_{L_v^2} \Big| dt\bigg]^{1/ 2}d\Sigma(k)\\
  	&\leq   \varepsilon  \int_{\mathbb Z^3}\comi{k}^{\frac{s}{1+2s} }\Big(   \int_0^1 \norm{  \Lambda_{\delta_1}\hat f_{m,\delta}(t,k)}_{L^2_v}^2d t\Big)^{\frac{1}{2}}d \Sigma(k)+C_{ \eps}\int_{\mathbb Z^3}\comi{k}^{\frac{s}{1+2s} }\Big(   \int_0^1 \norm{\hat f(t,k)}_{L^2_v}^2d t\Big)^{\frac{1}{2}}d \Sigma(k)\\
  	&	\leq   \varepsilon  \int_{\mathbb Z^3}\comi{k}^{\frac{s}{1+2s} }\Big(   \int_0^1 \norm{  \Lambda_{\delta_1}\hat f_{m,\delta}(t,k)}_{L^2_v}^2d t\Big)^{1/2}d \Sigma(k)+C_{ \eps}\norm{f_0}_{  L^1_kL^2_v},
\end{align*}	
%\end{eqnarray*}
for   any $ \eps>0,$ where in the last line we have used the  assertion (i) in Proposition   \ref{prpsub}.   Now, we substitute these inequalities into \eqref{t1t2} to obtain that for any $\eps  >0,$
\begin{eqnarray*}
	\begin{aligned}
 &		\int_{\mathbb Z^3} \sup_{0<t\leq 1}\norm{\Lambda_{\delta_1} \hat f_{m,\delta}(t,k)}_{L_v^2} d\Sigma(k) +\int_{\mathbb Z^3}\Big(\int_{0}^{1} \norm{(a^{1/2})^w\Lambda_{\delta_1} \hat f_{m,\delta}}_{L_v^2}^2 dt\Big)^{1/2}
 	d\Sigma(k)    \\
  &  \leq   \left( \eps +  C \eps^{-1}\varepsilon_0   \right)\int_{\mathbb Z^3} \left(\int_0^1    \norm{(a^{1/2})^w\hat f_{m,\delta}(t,k) }^2_{L_v^2}  d t\right)^{1/2}d \Sigma(k)\\
&\quad+C \eps^{-1}\varepsilon_0 \int_{\mathbb Z^3} \sup_{0<t\leq 1}     \norm{ \hat f_{m,\delta}(t,k )}_{L^2_v}    d \Sigma(k)+C_\eps\norm{f_0}_{  L^1_kL^2_v}\\
 & \quad 	+ \varepsilon  \int_{\mathbb Z^3}\comi{k}^{\frac{s}{1+2s} }\Big(   \int_0^1 \norm{  \Lambda_{\delta_1}\hat f_{m,\delta}(t,k)}_{L^2_v}^2d t\Big)^{1/2}d \Sigma(k)\\
    &\quad +C	\int_{\mathbb Z^3}\Big(\int_{0}^{1} \Big|\big([\mathcal L ,\ \Lambda_{\delta_1}] \hat f_{m,\delta},\ \Lambda_{\delta_1}   \hat f_{m,\delta}\big)_{L_v^2} \Big| dt\Big)^{1/2}
 	d\Sigma(k).
 \end{aligned}
\end{eqnarray*}
The rest argument  to obtain the assertion (i) in Proposition \ref{prpeng} is the same as that  in the case of $m\geq 2$,  so we omit it for brevity.   The proof of Proposition \ref{prpeng} is thus completed. 
\end{proof}

\subsection{Completing  the proofs} 

 Now we are ready to prove the main results in this section. 
 
 \begin{proof} 	[Proof of Proposition \ref{prplow}] Let $f(t,x,v)$ be the global mild solution to \eqref{eqforper} obtained in Proposition \ref{prop.exis} satisfying  the estimate  \eqref{+1234}.  Then, by the assertions (i) in Propositions \ref{prpsub} and \ref{prpeng}, we have
\begin{equation}\label{m0}
 		 \int_{\mathbb Z^3} \comi k^{\frac{s}{1+2s}} 	\Big(\int_{0}^{1} \norm{ \hat f(t,k)}_{L_v^2} ^2dt\Big)^{1/2} d\Sigma(k)\leq  C \norm{f_0}_{  L^1_kL^2_v},
 	\end{equation}
 	and
 	\begin{eqnarray*}
\int_{\mathbb Z^3} \sup_{0<t\leq 1}\norm{ \hat f_{m,\delta}(t,k)}_{L_v^2} d\Sigma(k) +\int_{\mathbb Z^3}\Big(\int_{0}^{1} \norm{(a^{1/2})^w  \hat f_{m,\delta}}_{L_v^2}^2 dt\Big)^{\frac{1}{2}}
 	d\Sigma(k) \leq  C \norm{f_0}_{  L^1_kL^2_v},
 	\end{eqnarray*}
with  $m=1$. These together with the assertion (ii) in Proposition \ref{prpsub} yield that with $m=1$, 
 	\begin{multline*}
\int_{\mathbb Z^3} \sup_{0<t\leq 1} \norm{      \hat f_{m,\delta} (t,k)}_{L_v^2}   d\Sigma(k)+ \int_{\mathbb Z^3} \comi k^{\frac{s}{1+2s}} 	\Big(\int_{0}^{1}\norm{ \hat f_{m,\delta}(t,k)}_{L_v^2} ^2dt\Big)^{1/2} d\Sigma(k)\\
		+  \int_{\mathbb Z^3} \left(\int_0^1    \norm{(a^{1/2})^w\hat f_{m,\delta}(t,k) }^2_{L_v^2}  d t\right)^{1/2}d \Sigma(k) 
	 \leq  C\norm{f_0}_{  L^1_kL^2_v}.
  \end{multline*}
  Combining the above estimate  with \eqref{m0} and \eqref{+1234} and letting the parameter $\delta$ above tend to $0$,  we conclude  that  for any $m\in\mathbb Z_+$ with $0\leq m\leq 1$,  
\begin{multline}\label{esmt}
\int_{\mathbb Z^3}  \comi{k}^{m}\Big( \sup_{0<t\leq 1} t^{ \varsigma m }  \norm{ {\hat f (t,k )}}_{L^2_v}\Big) d \Sigma(k) +  \int_{\mathbb Z^3}  \comi{k}^{m+\frac{s}{1+2s}}\Big( \int_0^1 t^{ 2\varsigma m }  \norm{ \hat f (t,k )}_{L^2_v}^2\Big)^{1/2} d \Sigma(k)\\
 + \int_{\mathbb Z^3}  \comi{k}^{m}\Big( \int_0^1 t^{ 2\varsigma m }  \norm{(a^{1/2})^w {\hat f (t,k )}}_{L^2_v}^2\Big)^{1/2} d \Sigma(k) \leq 4^{-1}C_1\norm{f_0}_{L^1_kL^2_v},	  	
\end{multline}
provided that the constant  $C_1$ is chosen large enough.  
We then have proved Proposition \ref{prplow} for $0<t\leq 1$.  As mentioned at the beginning, the treatment for  $1\leq t< T$ is similar. By combining the $L_k^1L_v^2$-norm of $ f|_{t=1}$  established in \eqref{esmt}, we see that all the estimates in Subsections \ref{secsub}-\ref{subeng}  are also true with $0<t\leq 1$ therein replaced by $1\leq t<T$. It follows that for  $T>1$,
\begin{multline*}
	\int_{\mathbb Z^3}  \comi{k}^{m}\Big( \sup_{1\leq t< T}  \norm{ {\hat f (t,k )}}_{L^2_v}\Big) d \Sigma(k) +  \int_{\mathbb Z^3}  \comi{k}^{m+\frac{s}{1+2s}}\Big( \int_1^T  \norm{ \hat f (t,k )}_{L^2_v}^2\Big)^{1/2} d \Sigma(k)\\
 + \int_{\mathbb Z^3}  \comi{k}^{m}\Big( \int_1^T    \norm{(a^{1/2})^w {\hat f (t,k )}}_{L^2_v}^2\Big)^{1/2} d \Sigma(k) \leq 2^{-1}C_1\norm{f_0}_{  L^1_kL^2_v}.
\end{multline*}
This together with \eqref{esmt} give the desired result in Proposition \ref{prplow}.  The proof is therefore completed. 
 \end{proof}
 
\begin{proof}[Proof of Proposition \ref{prpind}]
	The estimate for $0<t\leq 1$ just follows from the assertions for $m\geq 2$  in Propositions \ref{prpsub} and \ref{prpeng}.     Following the similar argument and using the estimate for $f|_{t=1}$, we then obtain  the desired estimate for $1\leq t\leq T$. The proof of Proposition \ref{prpind} is thus completed.  
\end{proof}   

\section{Gevrey regularity in velocity variables} \label{sec4}

This section  is devoted to proving the Gevrey smoothness  in velocity variables. Compared with the result for space variables in the previous section,  the main difference arises from the treatment of  two commutators with one  between the regularization operator and the collision operator and the other one between the velocity derivative and the transport operator.

\begin{theorem}\label{thm3.2}
Let $f(t,x,v)$ be the global mild solution to \eqref{eqforper} obtained in Proposition \ref{prop.exis} and  all the results as in Theorem \ref{thm3.1} be satisfied. Recall that $\tilde C_0$ is the constant constructed in Proposition \ref{prpind}. Let $\eps_0>0$ be further small, then there is  a positive constant $\tilde C_* \geq (\tilde C_0+1)^2$,  depending only on $s,\gamma$  in \eqref{kern} and \eqref{angu},   such that  for any $0<T<\infty$  and multi-index $\beta$ with   $|\beta|\geq 0$, the solution $f(t,x,v)$ satisfies
\begin{equation}
\label{thm3.2r1}
\partial_v^\beta f\in L^1_kL^\infty_{\tau,T}L^2_v,\quad (a^{1/2})^w\partial_v^\beta f\in L^1_kL^2_{\tau, T}L^2_v
\end{equation}
for any small $\tau>0$, with the quantitative estimate 
	\begin{multline*}%\label{thm3.2r2}
	\int_{\mathbb Z^3}	 \sup_{0<t< T} \phi(t) ^{\varsigma \abs\beta }\norm{  \partial_v^\beta \hat  f(t,k)}_{ L^2_v} d \Sigma(k) \\
	+ \int_{\mathbb Z^3}	\left(\int_0^T  \phi(t) ^{2\varsigma \abs\beta }\norm{(a^{1/2})^w  \partial_v^\beta \hat  f(t,k)}_{ L^2_v}^2d t\right)^{1/2}d \Sigma(k) 
	\leq   \tilde C_*^{|\beta| +1} ( |\beta|!)   ^{\frac{1+2s}{2s}}.
	\end{multline*}
%	provided  the number $\eps_0$ in \eqref{+1234} is small enough. 
Here we recall  that $\phi(t)=\min\{t,  1\}$ and $\varsigma=\frac{1+2s}{2s}$. 
\end{theorem}

We will prove the above result by using induction on  the  order $\abs\beta$ of velocity derivatives.  Precisely, Theorem \ref{thm3.2} follows from the  following two propositions.

\begin{proposition}[$H_v^2$-regularity]\label{lem}
Under all the assumptions of Theorem \ref{thm3.2}, there is a positive constant $ C_2\leq 1$,  depending only on $s$ and $\gamma$,   such that   for any $0<T<\infty$ and  for any $\beta\in\mathbb Z_+^3$ with $0\leq |\beta|\leq 2$, we have \eqref{thm3.2r1} with any small $\tau>0$ as well as the estimate
	 \begin{multline*} 
	 %\begin{split}
	 \int_{\mathbb Z^3} \sup_{0<t< T} \phi(t) ^{\varsigma \abs\beta }\norm{  \partial_v^\beta \hat  f(t,k)}_{ L^2_v} d \Sigma(k)\\
	 + \int_{\mathbb Z^3}	\left(\int_0^T  \phi(t) ^{2\varsigma \abs\beta }\norm{(a^{1/2})^w  \partial_v^\beta \hat  f(t,k)}_{ L^2_v}^2d t\right)^{\frac12}d \Sigma(k) 
	 \leq   C_2.
	 %\end{split}
	 \end{multline*} 
\end{proposition}

\begin{proposition}[Inductive Gevrey regularity]\label{progevofv} 
Suppose that all the assumptions of Theorem \ref{thm3.2} hold. Let $C_2\leq 1$ and $\tilde C_0$ be constants constructed in Proposition \ref{lem}  and  Proposition \ref{prpind}, respectively. Fix an arbitrary integer $m\geq 3$. Then, there is a positive constant  $ \tilde C_* \geq   (1+\tilde C_0)^2$,  depending only on $s$ and $\gamma$  but independent of $m,$ such that if for   any $\beta\in\mathbb Z_+^3$ with $0\leq \abs\beta\leq m-1 $ \eqref{thm3.2r1} holds for any small $\tau>0$ with the estimate
  \begin{eqnarray}\label{asonv}
 \begin{aligned}
	  &\int_{\mathbb Z^3}	 \sup_{0<t<T} \phi(t) ^{\varsigma \abs\beta }\norm{  \partial_v^\beta \hat  f(t)}_{ L^2_v} d \Sigma(k) +\int_{\mathbb Z^3}	 \inner{\int_0^T \phi(t) ^{2\varsigma \abs\beta } \norm{ (a^{1/2})^w\partial_v^\beta \hat f(t)}_{L^2_v}^2 d t}^{1/ 2} d \Sigma(k)\\
	&  \leq  
	\left\{
	\begin{aligned}  
	&C_2, %\tilde C_*,
	\quad \textrm{ if } \abs\beta\leq 2\\
	&\tilde C_*^{  \abs\beta-1}\com{ ( \abs\beta-1) !}^{\frac{1+2s}{2s}},\quad \textrm{ if } \abs\beta\geq 3,
	  \end{aligned}
	  \right.
	 \end{aligned}
	\end{eqnarray}
then we have the same thing for any  $\beta\in\mathbb Z_+^3$ with  $\abs\beta=m$.    
\end{proposition}

For convenience of presentation, we first focus on the proof of Proposition \ref{progevofv} in  the following Subsection \ref{secvelh}. The proof of Proposition \ref{lem}  for $H^2_v$ regularity will be postponed to Subsection \ref{prph2} since we will treat it in a similar way with the more straightforward  argument.

\subsection{Inductive Gevrey regularity} \label{secvelh}
We are going to first give the proof of Proposition \ref{progevofv}. As in the previous section for treating smoothness in space variables, we only consider the case of $0<t\leq 1$ and perform estimates for  the regularization     
	 \begin{equation}\label{Fmd}
		   \hat F_{m,\delta}(t,k,v)=t^{m\varsigma}  (1-\delta\partial_{v_1}^2)^{-1}  \partial_{v_1} ^m\hat f(t,k,v), 	\quad 0<\delta\ll 1, \ m\geq 1.
\end{equation}
Here we are first concerned with the velocity derivative in the first component $v_1$ and the same thing can be done for $v_2$ and $v_3$ as well.   
Then,  in view of \eqref{eqfour} and recalling that $\Lambda_{\delta_1}$ is defined by \eqref{regoper}, 
we have 
\begin{multline*}
		\big(\partial_t+iv\cdot k-\mathcal L \big) \Lambda_{\delta_1} \hat F_{m,\delta}= \Lambda_{\delta_1}t^{m\varsigma}  (1-\delta\partial_{v_1}^2)^{-1}  \partial_{v_1} ^m\hat\Gamma \big({\hat f},~{\hat f}\big) \\
		+
		\big[\Lambda_{\delta_1}t^{m\varsigma}  (1-\delta\partial_{v_1}^2)^{-1}  \partial_{v_1} ^m,\ \mathcal L \big ]  \hat f 
		+\varsigma mt^{-1} \Lambda_{\delta_1} \hat F_{m,\delta}+ \big[iv\cdot k,\ \Lambda_{\delta_1}t^{m\varsigma}  (1-\delta\partial_{v_1}^2)^{-1}  \partial_{v_1} ^m\big] \hat f.
			\end{multline*}
Note that similar results as in  Lemma \ref{verf} also hold for $ \hat F_{m,\delta}$.  This enables us to take the $L_v^2$ inner product of the above equation with $\hat F_{m,\delta}$  and then integrate the resulting result over $[0,t]$ for any $0<t\leq 1$, so as to obtain by following the argument in the previous Subsection \ref{subeng} that   
%\begin{equation}
\begin{align}
   & \int_{Z^3 } \sup_{0<t\leq 1}  \norm{\Lambda_{\delta_1}  \hat F_{m,\delta}(t,k)}_{ L_v^2}d \Sigma(k)   + \int_{\mathbb Z^3}\left(\int_0^1 \norm{(a^{1/2})^w \Lambda_{\delta_1} \hat F_{m,\delta}(t)}_{L^2_v}^2d t\right)^{1/2}d \Sigma(k) \notag \\
   & \leq C\int_{\mathbb Z^3}\left(\int_0^1   \big|\big(\Lambda_{\delta_1}t^{m\varsigma}  (1-\delta\partial_{v_1}^2)^{-1}  \partial_{v_1} ^m\hat\Gamma \big({\hat f},~{\hat f}\big)  ,\   \Lambda_{\delta_1} \hat F_{m,\delta}\big)_{L^2_v}\big| d t\right)^{1/2}d \Sigma(k)\notag \\
   &\quad+ C\int_{\mathbb Z^3}\left(\int_0^1  \big|\big(\big[\Lambda_{\delta_1}t^{m\varsigma}  (1-\delta\partial_{v_1}^2)^{-1}  \partial_{v_1} ^m, \ \mathcal L\big] \hat f,\   \Lambda_{\delta_1} \hat F_{m,\delta}\big)_{L^2_v}\big| d t\right)^{1/2}d \Sigma(k)\notag \\
   &\quad + C \int_{\mathbb Z^3}\left( \int_0^1 mt^{-1} \norm{ \Lambda_{\delta_1} \hat F_{m,\delta}(t)}_{L^2_v}^2d t\right)^{1/2}d \Sigma(k)\notag \\
   &\quad+ C\int_{\mathbb Z^3}\left(\int_0^1   \big|\big(\big[iv\cdot k,\ \Lambda_{\delta_1}t^{m\varsigma}  (1-\delta\partial_{v_1}^2)^{-1}  \partial_{v_1} ^m\big] \hat f,\   \Lambda_{\delta_1} \hat F_{m,\delta}\big)_{L^2_v}\big| d t\right)^{1/2}d \Sigma(k). 
   \label{tildereterm}
\end{align}
%\end{equation}
We proceed  through series of  lemmas as below  to estimate those  terms on the right.

\begin{lemma}\label{lemke}
Let all the assumptions of Proposition \ref{progevofv} hold. Then, assuming the induction assumption specified in  Proposition \ref{progevofv}, it holds that for any $\eps>0,$
\begin{multline*}
\int_{\mathbb Z^3}\left(\int_0^1   \big|\big(\Lambda_{\delta_1}t^{m\varsigma}  (1-\delta\partial_{v_1}^2)^{-1}  \partial_{v_1} ^m\hat\Gamma \big({\hat f},~{\hat f}\big)  ,\   \Lambda_{\delta_1} \hat F_{m,\delta}\big)_{L^2_v}\big| d t\right)^{1/2}d \Sigma(k)\\
    		\leq \inner{\eps+C \eps^{-1} \eps_0}\int_{\mathbb Z^3} \Big(\int_0^1    \norm{(a^{1/2})^w\hat F_{m,\delta}(t,k) }^2_{L_v^2}  d t\Big)^{1/2}d \Sigma(k)\\  + C \eps^{-1} \eps_0 \int_{\mathbb Z^3} \sup_{0<t\leq 1}      \norm{ \hat F_{m,\delta}(t,k)}_{L^2_v}    d \Sigma(k) 
    		+C_\eps \tilde C_*^{m-2}  [(m-1)!]^{\frac{1+2s }{2s}}.
\end{multline*}
\end{lemma}
    
\begin{proof}
Recall that $\Gamma(g,h)=\mathcal T(g,h, \mu^{1/2})$ with 
\begin{equation}\label{matht}
\mathcal T (g,h, w): =   \int_{\mathbb{R}^3}\int_{\mathbb{S}^2} B(v-v_*,\sigma) w(v_*) \inner{g(v_*')h(v')-g(v_*)h}\,d\sigma dv_*.
\end{equation}
Since $\Gamma $ is invariant under translation with respect to $v$, one can see that   the Leibniz formula in $v$ can be applied to obtain 
\begin{equation*}%\label{sumtrilinear}
\partial_{v_1}^m \hat \Gamma (\hat {f}, \hat {f} )=\sum_{ j=0}^{m}\sum_{  p =0}^{j} \begin{pmatrix}
      m    \\
      j 
\end{pmatrix}\begin{pmatrix}
      j    \\
      p  
\end{pmatrix} \hat  {\mathcal T}(\partial_{v_1}^{m-j} \hat {f}, \ \partial_{v_1}^{j- p } \hat {f},  \ \partial_{v_1}^{ p } \mu^{1/2}).
\end{equation*} 
Then,
\begin{equation}\label{ijk}
	\int_{\mathbb Z^3}\left(\int_0^1   \big|\big(\Lambda_{\delta_1}t^{m\varsigma}  (1-\delta\partial_{v_1}^2)^{-1}  \partial_{v_1} ^m\hat\Gamma \big({\hat f},~{\hat f}\big)  ,\   \Lambda_{\delta_1} \hat F_{m,\delta}\big)_{L^2_v}\big| d t\right)^{1/2}d \Sigma(k)\leq I+J+K,
\end{equation}
with
\begin{align}
		I&=  \int_{\mathbb Z^3}\left(\int_0^1   \big|\big(\Lambda_{\delta_1}t^{m\varsigma}  (1-\delta\partial_{v_1}^2)^{-1} \hat  {\mathcal T}( \hat {f}, \ \partial_{v_1}^{m} \hat {f},  \  \mu^{1/2})
 ,\   \Lambda_{\delta_1} \hat F_{m,\delta}\big)_{L^2_v}\big| d t\right)^{1/2}d \Sigma(k)\notag \\
 &\ + \sum_{ p =1}^m \begin{pmatrix}
      m    \\
      p 
\end{pmatrix}\int_{\mathbb Z^3}\left(\int_0^1   \big|\big(\Lambda_{\delta_1}t^{m\varsigma}  (1-\delta\partial_{v_1}^2)^{-1} \hat  {\mathcal T}( \hat {f}, \ \partial_{v_1}^{m- p } \hat {f},  \  \partial_{v_1}^{ p }\mu^{\frac{1}{2}})
 ,\   \Lambda_{\delta_1} \hat F_{m,\delta}\big)_{L^2_v}\big| d t\right)^{\frac{1}{2}}d \Sigma(k),\label{vei}
 \end{align}
 %\begin{eqnarray}
	\begin{align}
	J	&=  \int_{\mathbb Z^3} \sum_{ j=1}^{m-1}\sum_{  p =0}^{j} \begin{pmatrix}
      m    \\
      j 
\end{pmatrix}\begin{pmatrix}
      j    \\
      p  
\end{pmatrix} \notag \\
	&\qquad\quad\times\Big(\int_0^1   \big|\big(\Lambda_{\delta_1}t^{m\varsigma}  (1-\delta\partial_{v_1}^2)^{-1}\hat  {\mathcal T}(\partial_{v_1}^{m-j} \hat {f}, \ \partial_{v_1}^{j- p } \hat {f},  \ \partial_{v_1}^{ p } \mu^{1/2})
 ,\   \Lambda_{\delta_1} \hat F_{m,\delta}\big)_{L^2_v}\big| d t\Big)^{1/2}d \Sigma(k),\label{vej}
 \end{align}
%\end{eqnarray}
and 
\begin{eqnarray}
	\begin{aligned}
  K &= \int_{\mathbb Z^3}\left(\int_0^1   \big|\big(\Lambda_{\delta_1}t^{m\varsigma}  (1-\delta\partial_{v_1}^2)^{-1} \hat  {\mathcal T}( \partial_{v_1}^{m} \hat {f}, \ \hat {f},  \  \mu^{1/2})
 ,\   \Lambda_{\delta_1} \hat F_{m,\delta}\big)_{L^2_v}\big| d t\right)^{1/2}d \Sigma(k).	\label{vek}
 \end{aligned}
\end{eqnarray}
We estimate $I,J$ and $K$ as follows.
\medskip

\noindent\underline{\bf Estimate on  $I$}.  
For $I$ in \eqref{vei}, we claim that
for any $\eps>0$,
\begin{equation}\label{esi}
	I\leq (\eps+C\eps^{-1}\eps_0)    \int_{\mathbb Z^3}\ \Big(\int_0^1  \norm{(a^{1/2})^w    \hat  F_{m,\delta}(k)}_{L^2_v}^2  d t\Big)^{1/2} d \Sigma(k) + C_\eps   \tilde  C_*^{m-2}[(m-1)!]^{\frac{1+2s}{2s}}.
\end{equation}
In fact, in order to estimate the first term on the right hand side of $I$, we use  the fact that $(1-\delta\partial_{v_1}^2)(1-\delta\partial_{v_1}^2)^{-1}=1$ and then apply the Leibniz formula to write    
% \begin{eqnarray}
 	\begin{align}
 		&\hat {\mathcal T}( \hat f,\ \partial_{v_1}^{m} \hat f, \mu^{1/2})= \hat {\mathcal T}( \hat f,\ (1-\delta\partial_{v_1}^2)(1-\delta\partial_{v_1}^2)^{-1}\partial_{v_1}^{m} \hat f, \mu^{1/2})\notag \\
&\quad=(1-\delta\partial_{v_1}^2)\hat {\mathcal T}( \hat  f,\ (1-\delta\partial_{v_1}^2)^{-1}\partial_{v_1}^{m}  \hat f, \mu^{1/2})\notag \\
&\qquad\quad+\delta \sum_{ j=1}^{2}  \sum_{p=0}^{j}\begin{pmatrix}
      2    \\
      j  
\end{pmatrix} \begin{pmatrix}
      j   \\
      p  
\end{pmatrix}\hat {\mathcal T}(\partial_{v_1}^{j- p } \hat  f,\ \partial_{v_1}^{2-j}(1-\delta\partial_{v_1}^2)^{-1}\partial_{v_1}^{m}  \hat f, \ \partial_{v_1}^{ p }\mu^{1/2}).
\label{leib}
 \end{align}
% \end{eqnarray}
 As a result, recalling the definition  \eqref{Fmd} of $\hat F_{m,\delta}$, one has 
 \begin{equation}\label{i1i2}
 \int_{\mathbb Z^3}\left(\int_0^1   \big|\big(\Lambda_{\delta_1}t^{m\varsigma}  (1-\delta\partial_{v_1}^2)^{-1} \hat  {\mathcal T}( \hat {f}, \ \partial_{v_1}^{m} \hat {f},  \  \mu^{1/2})
 ,\   \Lambda_{\delta_1} \hat F_{m,\delta}\big)_{L^2_v}\big| d t\right)^{1/2}d \Sigma(k)
\leq C\big(I_{1}+I_{2}\big)
 \end{equation}
 with
\begin{equation*}
%\label{ }
I_{1}= \int_{\mathbb Z^3}\left(\int_0^1   \big|\big(\Lambda_{\delta_1}    \hat  {\mathcal T}( \hat {f}, \ \hat F_{m,\delta},  \  \mu^{1/2})
 ,\   \Lambda_{\delta_1} \hat F_{m,\delta}\big)_{L^2_v}\big| d t\right)^{1/2}d \Sigma(k),
\end{equation*}
and 
 %\begin{eqnarray*}
 \begin{align*}
 I_{2}=  \sum_{ j=1}^{2}  \sum_{p=0}^{j} \int_{\mathbb Z^3}\delta \bigg[\int_0^1 \big|&\big(\Lambda_{\delta_1}t^{m\varsigma}  (1-\delta\partial_{v_1}^2)^{-1} \hat  {\mathcal T}(\partial_{v_1}^{j- p } \hat  f,  \partial_{v_1}^{2-j}(1-\delta\partial_{v_1}^2)^{-1}\partial_{v_1}^{m}  \hat f, \partial_{v_1}^{ p }\mu^{\frac{1}{2}})
 ,\\
 &\quad \Lambda_{\delta_1} \hat F_{m,\delta}\big)_{L^2_v}\big| d t\bigg]^{\frac{1}{2}}d \Sigma(k).
 \end{align*}
%\end{eqnarray*}
Note that $\Gamma(g,h)=\mathcal T(g,h,\mu^{1/2})$. Thus, using Lemmas \ref{lemtrip} and \ref{lemtl} as well as \eqref{+1234}, we compute to have that for any $\eps>0,$ 
%\begin{equation}
 \begin{align}
	I_{1} &\leq  \eps\int_{\mathbb Z^3} \left(\int_0^1    \norm{(a^{1/2})^w\hat F_{m,\delta}(t,k) }^2_{L_v^2}  d t\right)^{1/2}d \Sigma(k)\notag \\
	&\quad+C\eps^{-1}  \int_{\mathbb Z^3} \sup_{0<t\leq 1}     \norm{ \hat f(t,k)}_{L^2_v}    d \Sigma(k) \int_{\mathbb Z^3} \left(\int_0^1    \norm{(a^{1/2})^w\hat F_{m,\delta}(t,k) }^2_{L_v^2}  d t\right)^{1/2}d \Sigma(k)\notag \\
	&\leq  \inner{\eps+C\eps^{-1}  \eps_0}\int_{\mathbb Z^3} \left(\int_0^1    \norm{(a^{1/2})^w\hat F_{m,\delta}(t,k) }^2_{L_v^2}  d t\right)^{1/2}d \Sigma(k). 
\label{esi1}
	\end{align}
%\end{equation}
Next we treat  the term $I_{2}$.  Notice that  the terms $\mathcal T(g,h,  \partial_{v_1}^{ p } \mu^{1/2})$ for $p\leq 2$  enjoy  the same properties  as those in Lemma \ref{lemtrip}  for  $\Gamma(g,h)=\mathcal T(g,h,    \mu^{1/2})$.  Consequently, using again Lemmas \ref{lemtrip} and \ref{lemtl}, it holds that for any  $\eps>0,$
  \begin{equation*}%\label{i12}
  	\begin{aligned}
  		I_{2}&\leq \eps    \int_{\mathbb Z^3}\ \Big(\int_0^1  \norm{(a^{1/2})^w    (1-\delta\partial_{v_1}^2)^{-1} \Lambda_{\delta_1}^*\Lambda_{\delta_1} \hat  F_{m,\delta}(k)}_{L^2_v}^2  d t\Big)^{1/2} d \Sigma(k) \\
  		&\quad+C\eps^{-1}  \sum_{1\leq j\leq 2} \, \sum_{0\leq  p \leq j} \bigg[  \int_{\mathbb Z^3} \sup_{0<t\leq 1}   t^{ \varsigma j}   \norm{\partial_{v_1}^{j- p } \hat  f}_{L_v^2}    d \Sigma(k) \\
& \qquad  \qquad \qquad \qquad    \times \int_{\mathbb Z^3} \Big(\int_0^1 t^{2{\varsigma (m-j)}}   \norm{(a^{1/2})^w\delta\partial_{v_1}^{2-j} (1-\delta\partial_{v_1}^2)^{-1}\partial_{v_1}^m\hat  f }^2_{L_v^2}  d t\Big)^{1/2}d \Sigma(k)\bigg].
  	\end{aligned}
  \end{equation*}
As for the first term on the right hand side,  we  use   Lemma \ref{lemcomest}  to obtain 
  \begin{eqnarray*}
  	\norm{(a^{1/2})^w    (1-\delta\partial_{v_1}^2)^{-1} \Lambda_{\delta_1}^*\Lambda_{\delta_1} \hat  F_{m,\delta}(k)}_{L^2_v}\leq \norm{(a^{1/2})^w   \hat  F_{m,\delta}(k)}_{L^2_v}.
  \end{eqnarray*}
Meanwhile, for the second term, we notice that  $\delta\partial_{v_1}^{2-j} (1-\delta\partial_{v_1}^2)^{-1}\partial_{v_1}^j$  is the Weyl quantization of  symbol  $\delta\eta_1^{2-j} (1+\delta\abs\eta^2)^{-1}\eta_1^{j}$, which belongs to 
  $S(1, \ \abs{dv}^2+\abs{d\eta}^2) $ uniformly with respect to $\delta$. This with  Lemma \ref{lemcomest} yield that
  \begin{eqnarray*}
  	\norm{(a^{1/2})^w\delta\partial_{v_1}^{2-j} (1-\delta\partial_{v_1}^2)^{-1}\partial_{v_1}^m\hat  f }_{L_v^2}\leq \norm{(a^{1/2})^w\partial_{v_1}^{m-j}\hat  f }_{L_v^2}.
  \end{eqnarray*}
 Then, combining these inequalities gives
   \begin{eqnarray*}
   \begin{aligned}
   	I_{2}&\leq  \eps    \int_{\mathbb Z^3}\ \Big(\int_0^1  \norm{(a^{1/2})^w    \hat  F_{m,\delta}(k)}_{L^2_v}^2  d t\Big)^{1/2} d \Sigma(k)  \\
   	&\qquad+C \eps^{-1} \Big(\sum_{\abs\beta\leq 2}    \int_{\mathbb Z^3} \sup_{0<t\leq 1}   t^{ \varsigma \abs\beta}   \norm{\partial_{v}^{\beta} \hat  f}_{L_v^2}    d \Sigma(k) \Big)\\
   	&\qquad\qquad\qquad\qquad\qquad\quad\times\sum_{1\leq j\leq 2}  \int_{\mathbb Z^3} \Big(\int_0^1 t^{2{\varsigma (m-j)}}   \norm{(a^{1/2})^w \partial_{v_1}^{m-j}\hat  f }^2_{L_v^2}  d t\Big)^{1/2}d \Sigma(k) \\
   	&\leq  \eps    \int_{\mathbb Z^3}\ \Big(\int_0^1  \norm{(a^{1/2})^w    \hat  F_{m,\delta}(k)}_{L^2_v}^2  d t\Big)^{1/2} d \Sigma(k)  +C_\eps   \tilde C_*^{m-2}\com{ (m-1) !}^{\frac{1+2s}{2s}},
\end{aligned}
   \end{eqnarray*}
where the last inequality holds because of the induction assumption specified in  Proposition \ref{progevofv}. Now, we substitute the above estimate and 
  \eqref{esi1} into \eqref{i1i2} to conclude
%\begin{equation}
 	\begin{align}
 		&\int_{\mathbb Z^3}\left(\int_0^1   \big|\big(\Lambda_{\delta_1}t^{m\varsigma}  (1-\delta\partial_{v_1}^2)^{-1} \hat  {\mathcal T}( \hat {f}, \ \partial_{v_1}^{m} \hat {f},  \  \mu^{1/2})
 ,\   \Lambda_{\delta_1} \hat F_{m,\delta}\big)_{L^2_v}\big| d t\right)^{1/2}d \Sigma(k)\notag \\
   	&\leq \inner{\eps+C\eps^{-1}  \eps_0}   \int_{\mathbb Z^3}\ \Big(\int_0^1  \norm{(a^{1/2})^w    \hat  F_{m,\delta}(k)}_{L^2_v}^2  d t\Big)^{1/2} d \Sigma(k)  +C_\eps   \tilde C_*^{m-2}\com{ (m-1) !}^{\frac{1+2s}{2s}}.\label{fite}
 	\end{align}
%\end{equation}
We have completed the estimate on the first term on the right hand side of $I$ given in \eqref{vei}.
       
  It remains to estimate the second term  on the right hand side of $I$.  In fact, following the above argument for treating $I_{2}$ and observing that
\begin{equation*}%\label{anfordis}
\norm{ 	 \partial_{v_1}^{ p } \mu^{1/2}}_{L_v^2} \leq 8^{ p+1 } p !,  \quad \forall\  p \geq 0,
\end{equation*} 
we have 
%\begin{eqnarray*}
 \begin{align*}
 		& \sum_{ p =1}^m \begin{pmatrix}
      m    \\
      p 
\end{pmatrix}\int_{\mathbb Z^3}\left(\int_0^1   \big|\big(\Lambda_{\delta_1}t^{m\varsigma}  (1-\delta\partial_{v_1}^2)^{-1} \hat  {\mathcal T}( \hat {f}, \ \partial_{v_1}^{m- p } \hat {f},  \  \partial_{v_1}^{ p }\mu^{\frac{1}{2}})
 ,\   \Lambda_{\delta_1} \hat F_{m,\delta}\big)_{L^2_v}\big| d t\right)^{\frac{1}{2}}d \Sigma(k)\\
 &\leq  \eps    \int_{\mathbb Z^3}\ \Big(\int_0^1  \norm{(a^{1/2})^w    \hat  F_{m,\delta}(k)}_{L^2_v}^2  d t\Big)^{1/2} d \Sigma(k)  \\
   	&\ +C_\eps \int_{\mathbb Z^3} \sup_{0<t\leq 1}       \norm{  \hat  f}_{L_v^2}    d \Sigma(k)   \sum_{ p =1}^m \begin{pmatrix}
      m    \\
      p 
\end{pmatrix} 8^{ p+1 }  p ! \int_{\mathbb Z^3} \Big(\int_0^1 t^{2{\varsigma (m- p )}}   \norm{(a^{1/2})^w \partial_{v_1}^{m- p }\hat  f }^2_{L_v^2}  d t\Big)^{\frac{1}{2}}d \Sigma(k).
 \end{align*}
 %\end{eqnarray*} 
Moreover, using   \eqref{asonv}  gives 
%\begin{eqnarray*}
\begin{align*}
		&\sum_{ p =1}^m\begin{pmatrix}
      m    \\
      p 
\end{pmatrix} 8^{ p+1 }  p ! \int_{\mathbb Z^3} \Big(\int_0^1 t^{2{\varsigma (m- p )}}   \norm{(a^{1/2})^w \partial_{v_1}^{m- p }\hat  f }^2_{L_v^2}  d t\Big)^{1/ 2}d \Sigma(k)\\
		&\leq  \sum_{ p =1}^{m-1} \frac{m!}{(m- p )!}   8^{ p +1}   \tilde  C_*^{m- p -1}\com{ (m- p -1) !}^{\frac{1+2s}{2s}} +   m! 8^{m+1}\int_{\mathbb Z^3} \Big(\int_0^1    \norm{(a^{1/2})^w  \hat  f }^2_{L_v^2}  d t\Big)^{\frac{1}{2}}d \Sigma(k)  \\
		&\leq \tilde  C_*^{m-2}[(m-1)!]^{\frac{1+2s}{2s}}
\sum_{ p =1}^{m-1} \frac{m \com{ (m- p-1 ) !}^{\frac{1}{2s}}
}{(m-p)[(m-1)!]^{\frac{1}{2s}}}  \tilde C_*^2 ( 8/ \tilde C_*)^{ p +1}+ \eps_0 m! 8^{m+1}\\
&\leq C   \tilde C_*^{m-2}[(m-1)!]^{\frac{1+2s}{2s}},
\end{align*}
%\end{eqnarray*}
where the last inequality holds true since we have chosen the constant $C_*>0$ large enough. Thus, combining these inequalities, we have
 \begin{multline*}
 	 \sum_{ p =1}^m \begin{pmatrix}
      m    \\
      p 
\end{pmatrix}\int_{\mathbb Z^3}\left(\int_0^1   \big|\big(\Lambda_{\delta_1}t^{m\varsigma}  (1-\delta\partial_{v_1}^2)^{-1} \hat  {\mathcal T}( \hat {f}, \ \partial_{v_1}^{m- p } \hat {f},  \  \partial_{v_1}^{ p }\mu^{\frac{1}{2}})
 ,\   \Lambda_{\delta_1} \hat F_{m,\delta}\big)_{L^2_v}\big| d t\right)^{1/ 2}d \Sigma(k)
 \\
 \leq \eps    \int_{\mathbb Z^3}\ \Big(\int_0^1  \norm{(a^{1/2})^w    \hat  F_{m,\delta}(k)}_{L^2_v}^2  d t\Big)^{1/2} d \Sigma(k) + C_\eps  \tilde  C_*^{m-2}[(m-1)!]^{\frac{1+2s}{2s}}.
 \end{multline*}
 This with \eqref{fite} yield the desired upper bound \eqref{esi} of $I$.
 
 \medskip

\noindent\underline{\bf Estimate on  $J$}.     
Recall that $J$ is given in \eqref{vej}.  Using the similar argument for treating $I$ above with slight modifications,  we arrive at
 \begin{multline*}
 	J \leq    \eps    \int_{\mathbb Z^3}\ \Big(\int_0^1  \norm{(a^{1/2})^w    \hat  F_{m,\delta}(k)}_{L^2_v}^2  d t\Big)^{1/2} d \Sigma(k)  \\
 +C_\eps  \sum_{ j=1}^{m-1}\sum_{  p =0}^{j} \begin{pmatrix}
      m    \\
      j 
\end{pmatrix}\begin{pmatrix}
      j    \\
      p  
\end{pmatrix}    8^{ p+1}  p !\Big( \int_{\mathbb Z^3} \sup_{0<t\leq 1}      t^{\varsigma(m-j)} \norm{  \partial_{v_1}^{m-j}  \hat  f}_{L_v^2}    d \Sigma(k) \Big) \\
   	 \times \int_{\mathbb Z^3} \Big(\int_0^1 t^{2{\varsigma (j- p )}}   \norm{(a^{1/2})^w \partial_{v_1}^{j- p }\hat  f }^2_{L_v^2}  d t\Big)^{1/2}d \Sigma(k).
\end{multline*}
Moreover, for any $1\leq j\leq m-1$, it holds that 
    \begin{equation*}
\begin{split}
& \sum_{p=0}^{j}  \begin{pmatrix}
      j    \\
      p  
\end{pmatrix}8^{ p+1} p !\int_{\mathbb Z^3} \Big(\int_0^1 t^{2{\varsigma (j- p )}}   \norm{(a^{1/2})^w \partial_{v_1}^{j- p }\hat  f }^2_{L_v^2}  d t\Big)^{1/ 2}d \Sigma(k)
\\
&\leq  \varepsilon_0 8^{j+1}j!+   \sum_{p=0}^{j-1}  \frac{j!}{(j-p)!}  8^{ p +1} \tilde  C_*^{j- p -1} [(j- p -1)!]^{\frac{1+2s}{2s}} \leq C \tilde C_*^{j-1}  [(j-1)!]^{\frac{1+2s }{2s}},
\end{split}
\end{equation*}
where in the second inequality we have used   the induction assumption \eqref{asonv} specified in  Proposition \ref{progevofv}. 
  As a result, combining the above inequality with the inductive assumption \eqref{asonv} and following the same argument as for deriving \eqref{tecalp1}, we have    
\begin{eqnarray*}
	\begin{aligned}
	&	\sum_{ j=1}^{m-1}\sum_{  p =0}^{j} \begin{pmatrix}
      m    \\
      j 
\end{pmatrix}\begin{pmatrix}
      j    \\
      p  
\end{pmatrix}    8^{p+1 }  p !\Big( \int_{\mathbb Z^3} \sup_{0<t\leq 1}      t^{\varsigma(m-j)} \norm{  \partial_{v_1}^{m-j}  \hat  f}_{L_v^2}    d \Sigma(k) \Big) \\
   &\qquad\qquad\qquad\qquad	 \times \int_{\mathbb Z^3} \Big(\int_0^1 t^{2{\varsigma (j- p )}}   \norm{(a^{1/2})^w \partial_{v_1}^{j- p }\hat  f }^2_{L_v^2}  d t\Big)^{\frac{1}{2}}d \Sigma(k)\\
   &\leq  C \sum_{  j=1}^{ m-1} \frac{m!}{j!(m-j)!}  \tilde C_*^{m-j-1}  [(m-j-1)!]^{\frac{1+2s }{2s}} \tilde C_*^{j-1}  [(j-1)!]^{\frac{1+2s }{2s}}  \leq C \tilde C_*^{m-2}  [(m-1)!]^{\frac{1+2s }{2s}}.
	\end{aligned}
\end{eqnarray*}
Thus we conclude by combining these inequalities that 
\begin{equation}\label{j}
	\begin{aligned}
		J \leq \eps    \int_{\mathbb Z^3}\ \Big(\int_0^1  \norm{(a^{1/2})^w    \hat  F_{m,\delta}(k)}_{L^2_v}^2  d t\Big)^{1/2} d \Sigma(k) +C_\eps \tilde C_*^{m-2}  [(m-1)!]^{\frac{1+2s }{2s}}.
	\end{aligned}
\end{equation}

\medskip

\noindent\underline{\bf Estimate on  $K$}.  Recall $K$ in \eqref{vek}.   The estimate on $K$ is similar to that on $I$ shown before. In fact, we can follow the argument for proving  \eqref{fite} to conclude 
 \begin{multline}\label{k}
	K  \leq   \eps\int_{\mathbb Z^3} \Big(\int_0^1    \norm{(a^{1/2})^w\hat F_{m,\delta}(t,k) }^2_{L_v^2}  d t\Big)^{1/2}d \Sigma(k)  \\
	  + C \eps^{-1} \eps_0 \int_{\mathbb Z^3} \sup_{0<t\leq 1}      \norm{ \hat F_{m,\delta}(t,k)}_{L^2_v}    d \Sigma(k)+C_\eps  \tilde  C_*^{m-2}\com{ (m-1) !}^{\frac{1+2s}{2s}}.
	\end{multline}
The details are omitted for brevity.  

\medskip
Now, we substitute all the above estimates  \eqref{esi}, \eqref{j} and \eqref{k} back to \eqref{ijk} and hence complete the proof of Lemma \ref{lemke}.  
\end{proof}

\begin{lemma}\label{lemke++}
Let all the assumptions of Proposition \ref{progevofv} hold. Then, assuming the induction assumption specified in  Proposition \ref{progevofv}, it holds that for any $\eps>0,$
\begin{multline*}
    	\varlimsup_{\delta_1\rightarrow 0}	\int_{\mathbb Z^3}\left(\int_0^1  \big|\big(\big[\Lambda_{\delta_1}t^{m\varsigma}  (1-\delta\partial_{v_1}^2)^{-1}  \partial_{v_1} ^m, \ \mathcal L\big] \hat f,\   \Lambda_{\delta_1} \hat F_{m,\delta}\big)_{L^2_v}\big| d t\right)^{1/2}d \Sigma(k)\\
    		\leq  \eps \int_{\mathbb Z^3} \Big(\int_0^1    \norm{(a^{1/2})^w\hat F_{m,\delta}(t,k) }^2_{L_v^2}  d t\Big)^{1/2}d \Sigma(k) +C_\eps \tilde C_*^{m-2}  [(m-1)!]^{\frac{1+2s }{2s}}.
    	\end{multline*}
    \end{lemma}
    
\begin{proof}
Observe that we have the following identity
    	\begin{eqnarray*}
    		\begin{aligned}
    			\big[\Lambda_{\delta_1}t^{m\varsigma}  (1-\delta\partial_{v_1}^2)^{-1}  \partial_{v_1} ^m, \ \mathcal L\big] =\big[\Lambda_{\delta_1}, \ \mathcal L\big] t^{m\varsigma}  (1-\delta\partial_{v_1}^2)^{-1}  \partial_{v_1} ^m+\Lambda_{\delta_1}t^{m\varsigma}\big[  (1-\delta\partial_{v_1}^2)^{-1}  \partial_{v_1} ^m, \ \mathcal L\big], 
\end{aligned}
\end{eqnarray*}
where for the commutator in the last term we can further write
\begin{multline*}
    			\big[  (1-\delta\partial_{v_1}^2)^{-1}  \partial_{v_1} ^m, \ \mathcal L\big]\hat f=    (1-\delta\partial_{v_1}^2)^{-1}  \partial_{v_1} ^m  \Gamma (\hat f, \ \mu^{1/2})-  \Gamma ((1-\delta\partial_{v_1}^2)^{-1}  \partial_{v_1} ^m \hat f, \ \mu^{1/2})\\
    			+ (1-\delta\partial_{v_1}^2)^{-1}  \partial_{v_1} ^m  \Gamma ( \mu^{1/2},\ \hat f) 
    			-   \Gamma ( \mu^{1/2},\ (1-\delta\partial_{v_1}^2)^{-1}  \partial_{v_1} ^m \hat f).
\end{multline*}
This enables us to follow the argument in the proof of  Lemma \ref{lemke} so as to conclude
    			\begin{multline*}
    		\int_{\mathbb Z^3}\left(\int_0^1  \big|\big(\Lambda_{\delta_1}t^{m\varsigma}\big[  (1-\delta\partial_{v_1}^2)^{-1}  \partial_{v_1} ^m, \ \mathcal L\big] \hat f,\   \Lambda_{\delta_1} \hat F_{m,\delta}\big)_{L^2_v}\big| d t\right)^{1/2}d \Sigma(k)\\
    		\leq \eps \int_{\mathbb Z^3} \Big(\int_0^1    \norm{(a^{1/2})^w\hat F_{m,\delta}(t,k) }^2_{L_v^2}  d t\Big)^{1/2}d \Sigma(k)       		+C_\eps \tilde C_*^{m-2}  [(m-1)!]^{\frac{1+2s }{2s}}.
\end{multline*}
Meanwhile, similarly as for showing \eqref{lim}, we have   
%repeating  the argument  after \eqref{lim}, 
    	\begin{eqnarray*}
    		\lim_{\delta_1\rightarrow 0}	\int_{\mathbb Z^3}\left(\int_0^1  \big|\big(\big[\Lambda_{\delta_1}, \ \mathcal L\big] t^{m\varsigma}  (1-\delta\partial_{v_1}^2)^{-1}  \partial_{v_1} ^m \hat f,\   \Lambda_{\delta_1} \hat F_{m,\delta}\big)_{L^2_v}\big| d t\right)^{1/2}d \Sigma(k) =0.
\end{eqnarray*}
Combining those estimates gives the conclusion of Lemma \ref{lemke++}. Therefore, the proof is  completed. 
\end{proof}

\begin{lemma} \label{lemla}  
Let all the assumptions of Proposition \ref{progevofv} hold. Then, assuming the induction assumption specified in  Proposition \ref{progevofv}, it holds that for any $\eps>0,$
    	\begin{eqnarray*}
    		\begin{aligned}
    			&\int_{\mathbb Z^3}\left( \int_0^1 mt^{-1} \norm{ \Lambda_{\delta_1} \hat F_{m,\delta}(t)}_{L^2_v}^2d t\right)^{1/2}d \Sigma(k)\\
    			&\leq \eps       \int_{\mathbb Z^3}\left(\int_0^1  \norm{(a^{1/2})^w \hat F_{m,\delta}}_{L^2_v}^2d t\right)^{\frac12}d \Sigma(k) +C_{ \eps}  \tilde C_*^{m-2}     \com{ (m-1) !}^{\frac{1+2s}{2s}}.
    		\end{aligned}
\end{eqnarray*}
\end{lemma}

\begin{proof}
We start with the Sobolev interpolation inequality of the form
\begin{equation*}%\label{interp}
\begin{split}
\norm{ \hat F_{m,\delta}}_{L^2_v}^2\leq  \tau \norm{\comi{D_{v}}^{\frac{s}{1+2s}}  \hat F_{m,\delta}}_{L^2_v}^2+   \tau^{-\frac{1+s}{s}} \norm{\comi{D_{v}}^{\frac{s}{1+2s}-1} \hat F_{m,\delta}}_{L^2_v}^2,
\end{split}
\end{equation*}
for any $\tau>0$. In particular, choosing $\tau=t\eps^2/m$  for $\eps>0$ and recalling the notation $\varsigma=\frac{1+2s}{2s}$, it follows that
%\begin{equation}
\begin{align}
& \int_{\mathbb Z^3}\left( \int_0^1 mt^{-1} \norm{ \Lambda_{\delta_1} \hat F_{m,\delta}}_{L^2_v}^2d t\right)^{\frac{1}{2}}d \Sigma(k) 
  \leq     \eps       \int_{\mathbb Z^3}\left(\int_0^1  \norm{\comi{D_{v}}^{\frac{s}{1+2s}} \hat F_{m,\delta}}_{L^2_v}^2d t\right)^{\frac12}d \Sigma(k)\notag \\
 &\qquad\qquad\qquad+C_{ \eps} m^{\frac{1+2s}{2s}} \int_{\mathbb Z^3}\left(\int_0^1 t^{-2 \varsigma }\norm{\comi{D_{v}}^{\frac{s}{1+2s}-1} \hat F_{m,\delta}}_{L^2_v}^2d t\right)^{\frac12}d \Sigma(k)\notag \\
&\leq   \eps       \int_{\mathbb Z^3}\left(\int_0^1  \norm{(a^{1/2})^w \hat F_{m,\delta}}_{L^2_v}^2d t\right)^{\frac12}d \Sigma(k)\notag \\
 &\quad+C_{ \eps} m^{\frac{1+2s}{2s}} \int_{\mathbb Z^3}\left(\int_0^1 t^{-2 \varsigma }\norm{(a^{1/2})^w\comi{D_{v}}^{-1} \hat F_{m,\delta}}_{L^2_v}^2d t\right)^{\frac12}d \Sigma(k),
\label{interpp2}
\end{align}
%\end{equation}
where  in the last inequality we have used the assertion (iii) in Proposition \ref{estaa}.   Moreover,  in view of \eqref{Fmd},  we use the induction assumption \eqref{asonv} as well as Lemma \ref{lemcomest} to compute
\begin{equation*}
\begin{aligned}
&m^{\frac{1+2s}{2s}} \int_{\mathbb Z^3}\left(\int_0^1 t^{-2 \varsigma }\norm{(a^{1/2})^w \comi{D_{v}}^{-1} \hat F_{m,\delta}}_{L^2_v}^2d t\right)^{\frac12}d \Sigma(k) \\
&\leq  C m^{\frac{1+2s}{2s}} \int_{\mathbb Z^3}\left(\int_0^1 t^{2(m-1)\varsigma}  \norm{(a^{1/2})^w \partial_{v_1}^{m-1} \hat f(t,k)}_{L^2_v}^2d t\right)^{\frac12}d \Sigma(k)  \leq
 C \tilde  C_*^{m-2}     \com{ (m-1) !}^{\frac{1+2s}{2s}}.
\end{aligned}
\end{equation*}
Thus the desired estimate follows by combining those inequalities.  The proof of Lemma  \ref{lemla} is then completed.  	
\end{proof}

\begin{lemma}\label{lemcom}
Let all the assumptions of Proposition \ref{progevofv} hold. Then, assuming the induction assumption specified in  Proposition \ref{progevofv}, it holds that for any $\eps>0,$
\begin{eqnarray*}
    		\begin{aligned}
    			&\int_{\mathbb Z^3}\left(\int_0^1   \big|\big(\big[iv\cdot k,\ \Lambda_{\delta_1}t^{m\varsigma}  (1-\delta\partial_{v_1}^2)^{-1}  \partial_{v_1} ^m\big] \hat f,\   \Lambda_{\delta_1} \hat F_{m,\delta}\big)_{L^2_v}\big| d t\right)^{1/2}d \Sigma(k)\\
    			&\leq  \varepsilon\int_{\mathbb Z^3}   \left(  \int_0^1  t^{2\varsigma m}\norm{ (a^{1/2})^w  \hat F_{m,\delta}}_{L^2_v}^2\ d t\right)^{\frac12}d \Sigma(k) +  C_\varepsilon \tilde  C_*^{m-2}     \com{ (m-1) !}^{\frac{1+2s}{2s}}.
    		\end{aligned}
    	\end{eqnarray*}
    \end{lemma}
    
    \begin{proof}
    	Note that 
\begin{multline*} 
	\com{ iv\cdot k,   \  \Lambda_{\delta_1}t^{m\varsigma}  (1-\delta\partial_{v_1}^2)^{-1}\partial_{v_1}^m} \\
	=-i \Lambda_{\delta_1}t^{m\varsigma}  (1-\delta\partial_{v_1}^2)^{-1} m  k_1 \partial_{v_1}^{m-1}-2i \Lambda_{\delta_1}t^{m\varsigma}  (1-\delta\partial_{v_1}^2)^{-2}(\delta k_1\partial_{v_1}) \partial_{v_1}^m
	\\
	-2i\big  (1+\delta_1  \abs v^2 \big)^{-1-\gamma} \big(1+\delta_1 \abs k^2-\delta_1\Delta_v\big)^{-2} (\delta_1 k\cdot \partial_v)t^{m\varsigma}  (1-\delta\partial_{v_1}^2)^{-1}\partial_{v_1}^m
	.
\end{multline*}
This implies
%\begin{equation}
\begin{align}
\norm{[iv\cdot k,\ \Lambda_{\delta_1}t^{m\varsigma}  (1-\delta\partial_{v_1}^2)^{-1}  \partial_{v_1} ^m] \hat f}_{L_v^2}&\leq C m t^{m\varsigma}\comi k\norm{   (1-\delta\partial_{v_1}^2)^{-1}  \partial_{v_1} ^{m-1}  \hat f}_{L_v^2}\notag\\
	&\leq C m t^{m\varsigma}\comi k^m\norm{       \hat f}_{L_v^2}+C m t^{m\varsigma} \norm{    (1-\delta\partial_{v_1}^2)^{-1}  \partial_{v_1} ^{m}  \hat f}_{L_v^2} \notag\\
	&=C m t^{m\varsigma}\comi k^m\norm{       \hat f}_{L_v^2}+C m \norm{  \hat F_{m,\delta}}_{L_v^2}.
\label{comest}
\end{align}
%\end{equation}
Consequently, via the similar arguments as for deriving \eqref{interp1} and  \eqref{interpp2}, we have
%\begin{eqnarray*}
\begin{align*}
    			&\int_{\mathbb Z^3}\left(\int_0^1   \big|\big(\big[iv\cdot k,\ \Lambda_{\delta_1}t^{m\varsigma}  (1-\delta\partial_{v_1}^2)^{-1}  \partial_{v_1} ^m\big] \hat f,\   \Lambda_{\delta_1} \hat F_{m,\delta}\big)_{L^2_v}\big| d t\right)^{1/2}d \Sigma(k)\\
    			& \leq C  \int_{\mathbb Z^3}  \comi{k}^{m+\frac{s}{1+2s}} \left( \int_0^1  t^{2\varsigma m}\norm{  \hat      f}_{L^2_v}^2\ d t\right)^{\frac12}d \Sigma(k) +\eps \int_{\mathbb Z^3}   \left(  \int_0^1  \norm{  (a^{1/2})^w \hat F_{m,\delta}}_{L^2_v}^2\ d t\right)^{\frac12}d \Sigma(k)\\
    			&\quad+ C \tilde C_0^{m-2}     \com{ (m-1) !}^{\frac{1+2s}{2s}}+  C_\eps   \tilde C_*^{m-2}     \com{ (m-1) !}^{\frac{1+2s}{2s}}\\
    			&\leq  \varepsilon\int_{\mathbb Z^3}   \left(  \int_0^1  t^{2\varsigma m}\norm{ (a^{1/2})^w  \hat F_{m,\delta}}_{L^2_v}^2\ d t\right)^{\frac12}d \Sigma(k) +  C_\eps \tilde  C_*^{m-2}     \com{ (m-1) !}^{\frac{1+2s}{2s}},
\end{align*}
%\end{eqnarray*}
where we have used Proposition \ref{prpind} with the constant $\tilde C_0$ therein and also we have chosen $ \tilde C_*$ such that $ \tilde C_* \geq (\tilde C_0+1)^2$. The proof of Lemma \ref{lemcom} is thus completed.	
\end{proof}

\begin{proof}
[Ending the proof of  Proposition \ref{progevofv}] 
We substitute all the estimates in Lemmas \ref{lemke}-\ref{lemcom} into \eqref{tildereterm} and further let  $\delta_1\rightarrow 0$. It thus follows from the Fatou Lemma that
    \begin{eqnarray*}
    	\begin{aligned}
    			 & \int_{Z^3 } \sup_{0<t\leq 1}  \norm{   \hat F_{m,\delta}(t,k)}_{ L_v^2}d \Sigma(k)   + \int_{\mathbb Z^3}\left(\int_0^1 \norm{(a^{1/2})^w  \hat F_{m,\delta}(t,k)}_{L^2_v}^2d t\right)^{1/2}d \Sigma(k)\\
       		 & \leq \inner{\eps+C \eps^{-1} \eps_0}\int_{\mathbb Z^3} \Big(\int_0^1    \norm{(a^{1/2})^w\hat F_{m,\delta}(t,k) }^2_{L_v^2}  d t\Big)^{1/2}d \Sigma(k)\\  
       		 &\quad+ C \eps^{-1} \eps_0 \int_{\mathbb Z^3} \sup_{0<t\leq 1}      \norm{ \hat F_{m,\delta}(t,k)}_{L^2_v}    d \Sigma(k) 
    		+C_\eps \tilde C_*^{m-2}  [(m-1)!]^{\frac{1+2s }{2s}}.
    		    	\end{aligned}
    \end{eqnarray*}
In the above estimate, we let $\eps>0$ and further $\eps_0>0$ be suitably small so that those integral terms on the right can be absorbed.  Furthermore, leting $\delta\rightarrow 0$, in view of the definition \eqref{Fmd} of $\hat F_{m,\delta}$, it  follows again from the Fatou Lemma that
    \begin{eqnarray*}
    	\begin{aligned}
    			 &\int_{Z^3 } \sup_{0<t\leq 1}  t^{\varsigma m}\norm{\partial_{v_1}^m  \hat f(t,k)}_{ L_v^2}d \Sigma(k)   + \int_{\mathbb Z^3}\left(\int_0^1 t^{2\varsigma m}\norm{(a^{1/2})^w \partial_{v_1}^m \hat f(t,k)}_{L^2_v}^2d t\right)^{1/2}d \Sigma(k) \\
       		 & \leq C \tilde  C_*^{m-2}     \com{ (m-1) !}^{\frac{1+2s}{2s}}.
    	\end{aligned}
    \end{eqnarray*}
Note that the above estimate still holds with $\partial_{v_1}^m$ replaced by $\partial_{v_j}^m, j=2,3$.  Thus for any $\beta\in\mathbb Z_+^3$ with $\abs\beta=m\geq 3$,  using the fact that 
    \begin{eqnarray*}
    	\norm{\partial_v^{\beta}\hat f}_{L_v^2}\leq  C\sum_{1\leq j\leq 3} \norm{\partial_{v_j}^{m}\hat f}_{L_v^2},
    \end{eqnarray*}
it holds that 
     \begin{eqnarray*}
     	 \begin{aligned}
	  &\int_{\mathbb Z^3}	 \sup_{0<t\leq 1} t ^{\varsigma \abs\beta }\norm{  \partial_v^\beta \hat  f(t,k)}_{ L^2_v} d \Sigma(k) +\int_{\mathbb Z^3}	 \inner{\int_0^1 t ^{2\varsigma \abs\beta } \norm{ (a^{1/2})^w\partial_v^\beta \hat f(t,k)}_{L^2_v}^2 d t}^{1/ 2} d \Sigma(k)\\
	&  \leq   C \tilde   C_*^{  \abs\beta-2}\com{ ( \abs\beta-1) !}^{\frac{1+2s}{2s}} \leq  \frac{1}{4} \tilde C_*^{  \abs\beta-1}\com{ ( \abs\beta-1) !}^{\frac{1+2s}{2s}},
\end{aligned}
\end{eqnarray*}
where we have chosen $ \tilde C_*>4C$ in the second line. Thus, we have proved Proposition \ref{progevofv} for $0<t\leq 1$. The estimates for  $1\leq t\leq T$ are similar and here we omit them for brevity.  The proof of  Proposition \ref{progevofv} is therefore completed.  
\end{proof} 
        
\subsection{Low-order regularity}\label{prph2}
 
In this part we are going to show Proposition \ref{lem} for the low-order $H_v^2$-regularity. As mentioned before, the key argument is similar to that for the proof of  Proposition \ref{progevofv} in the previous subsection. The main differences between them arise from the absence of the induction assumption \eqref{asonv}.   Instead   we will prove the desired result for $\abs\beta\leq 2$  starting from the global existence result of low-regularity solutions in Proposition \ref{prop.exis}, in particular the estimate \eqref{+1234}.  
 
To begin with we study the smoothness in $H_v^2$ for solutions to the following linear Cauchy problem with initial data in $L^1_kL^2_v$:
 \begin{equation}\label{lincau}
      \inner{\partial_t+v\cdot\nabla_x-\mathcal L}h=\Gamma(g,h), \quad h(0,x,v)=f_0(x,v),
  \end{equation}
where $g$ is a  given function satisfying certain conditions listed as below.     

\begin{proposition}
 [Linear problem: existence and spatial regulairty] \label{lemexi}
Let $ f_0\in L^1_kL^2_v$. There is a constant $c_0>0$ such that if $g$ satisfies 
\begin{eqnarray*} 
	 \int_{\mathbb Z^3} \sup_{0<t< T}  \norm{    \hat  g(t,k)}_{ L^2_v} d \Sigma(k) 
	 + \int_{\mathbb Z^3}	\left(\int_0^T   \norm{(a^{1/2})^w   \hat  g(t,k)}_{ L^2_v}^2d t\right)^{1/2}d \Sigma(k) 
	 \leq   c_0,
\end{eqnarray*}
for any $T>0$,  then the linear Cauchy problem \eqref{lincau} admits a unique global-in-time mild solution  $h\in L^1_kL^\infty_TL^2_v$ for any $T>0$. Moreover,  there is a constant $C_3>0$, depending only on the parameters $s, \gamma$ in \eqref{kern} and \eqref{angu} but independent of the constant $c_0$ above,    such that it holds for any $T>0$ that
\begin{equation}
\label{1234}
\norm{h}_{L_k^1L_T^\infty L_v^2}	+ \norm{(a^{1/2})^wh}_{L_k^1L_T^2 L_v^2}	 \leq     C_3   \norm{f_0}_{L_k^1L^2_v} 
\end{equation} 
and
\begin{multline}\label{spreg}
 	 \int_{\mathbb Z^3} \comi k^N\sup_{0<t<T} \phi(t)^{N\varsigma} \norm{    \hat  h(t,k)}_{ L^2_v} d \Sigma(k) 
	 \\
	 + \int_{\mathbb Z^3}	\comi k^N \left(\int_0^T  \phi(t)^{2N\varsigma} \norm{(a^{1/2})^w   \hat  h(t,k)}_{ L^2_v}^2d t\right)^{1/2}d \Sigma(k) \leq C_3  \norm{f_0}_{L_k^1L^2_v},
\end{multline}
with any integer $0\leq N\leq 2$. Here we recall  that $\phi(t)=\min\{t,  1\}$ and $\varsigma=\frac{1+2s}{2s}$. 
 \end{proposition}
 
 \begin{proof}
 The existence and uniqueness of solutions satisfying \eqref{1234} follow from the same strategy as in \cite{MR4230064},  where the corresponding results were established  for the nonlinear rather than linear problem.   And the spatial regularity \eqref{spreg} can be achieved   similarly as in Proposition \ref{prplow} by virtue of \eqref{1234}.   We omit the details here for brevity. 
 \end{proof}

\begin{proposition}[Linear problem: $H_v^2$-smoothing effect]\label{lemlin}
There are constants $\varepsilon_0>0$ and $C_4>0$ such that if  $f_0\in L^1_kL^2_v$ with 
%such that 
\begin{equation}\label{smacon}
 \norm{f_0}_{  L^1_kL^2_v} \leq
 \varepsilon_0 
\end{equation}
holds and  
%for some $\eps_0>0$.  Suppose $g$ is given satisfying that a constant $C_4>0$ exists such that 
$g$ satisfies that for any $T>0$
 and for any $ \beta\in\mathbb Z_+^3$ with $\abs\beta\leq 2$,
\begin{multline} \label{inong}
	 \int_{\mathbb Z^3} \sup_{0<t<T} \phi(t) ^{\varsigma \abs\beta }\norm{  \partial_v^\beta \hat  g(t,k)}_{ L^2_v} d \Sigma(k)\\
	 + \int_{\mathbb Z^3}	\left(\int_0^T  \phi(t) ^{2\varsigma \abs\beta }\norm{(a^{1/2})^w  \partial_v^\beta \hat  g(t,k)}_{ L^2_v}^2d t\right)^{1/2}d \Sigma(k) 
	 \leq    C_4 \norm{f_0}_{L_k^1L^2_v},
\end{multline} 
then the global solution $h\in L^1_kL^\infty_TL^2_v$ to \eqref{lincau}  constructed in Proposition \ref{lemexi} satisfies  the estimates \eqref{1234} and \eqref{spreg}, and it further holds that  for any $T>0$ and any $\beta\in\mathbb Z_+^3$ with $\abs\beta\leq 2 $,
\begin{multline} \label{dees}
 		 \int_{\mathbb Z^3} \sup_{0<t< T} \phi(t) ^{\varsigma \abs\beta }\norm{  \partial_v^\beta \hat  h(t,k)}_{ L^2_v} d \Sigma(k)\\
	 + \int_{\mathbb Z^3}	\left(\int_0^T  \phi(t) ^{2\varsigma \abs\beta }\norm{(a^{1/2})^w  \partial_v^\beta \hat  h(t,k)}_{ L^2_v}^2d t\right)^{1/2}d \Sigma(k) 
	 \leq    C_4 \norm{f_0}_{L_k^1L^2_v}.
\end{multline} 
%with the same constant $C_4$ as in \eqref{inong} by increasing $C_4$ if necessary.  
\end{proposition}

As for the proof of Proposition \ref{lemlin} above, in terms of Proposition \ref{lemexi} it suffices to focus on the proof of \eqref{dees} with $|\beta|=1$ and $2$. We proceed it through the following two lemmas.
%to prove Proposition \ref{lemlin}.

\begin{lemma}[$\abs\beta=2$] \label{le2}
Under the same conditions on $f_0$ and $g$ as in Proposition \ref{lemlin}, the estimate \eqref{dees}  holds for any $\beta\in\mathbb Z_+^3$ with  $\abs\beta=2$,  provided that $\eps_0>0$  is small enough.  
 \end{lemma}

\begin{proof}
As in the previous discussions it suffices to  consider the case of $T\leq 1$. In such case, one has $\phi(t)=t.$   We introduce the regularization  $  \hat h_{\delta}$ by setting 
	\begin{equation}\label{hel}
		   \hat h_{\delta}(t,k,v)=t^{2\varsigma}  (1-\delta\partial_{v_1}^2)^{-1}  \partial_{v_1}^2  \hat h(t,k,v), 	\quad 0<\delta\ll 1,
\end{equation}
and  let $\Lambda_{\delta_1}$ be the regularization operator defined by  \eqref{regoper}. 
Observe   
		\begin{multline*}
		\big(\partial_t+iv\cdot k\big) \Lambda_{\delta_1} \hat h_{\delta}-\Lambda_{\delta_1}t^{ 2\varsigma}  (1-\delta\partial_{v_1}^2)^{-1}  \partial_{v_1} ^2  \mathcal L \hat h=\Lambda_{\delta_1}t^{ 2\varsigma}  (1-\delta\partial_{v_1})^{-1}  \partial_{v_1}^2\hat\Gamma \big({\hat g},~{\hat f}\big) \\
		+ 2\varsigma t^{-1} \Lambda_{\delta_1} \hat h_{\delta}+ [iv\cdot k,\ \Lambda_{\delta_1}  t^{2\varsigma}(1-\delta\partial_{v_1}^2)^{-1}  \partial_{v_1}^2]\hat h. 
			\end{multline*}
Then, we perform the similar energy estimates for 	$\Lambda_{\delta_1} \hat h_{\delta}$ as in the previous parts, to get 
 %\begin{equation}
 	\begin{align}
   & \int_{Z^3 } \sup_{0<t\leq 1}  \norm{\Lambda_{\delta_1}  \hat h_{\delta}(t,k)}_{ L_v^2}d \Sigma(k)   + \int_{\mathbb Z^3}\left(\int_0^1 \norm{(a^{1/2})^w \Lambda_{\delta_1} \hat h_{\delta}(t,k)}_{L^2_v}^2d t\right)^{1/2}d \Sigma(k) \notag \\
   & \leq C\int_{\mathbb Z^3}\left(\int_0^1   \big|\big(\Lambda_{\delta_1}t^{ 2\varsigma}  (1-\delta\partial_{v_1}^2)^{-1}  \partial_{v_1}^2  \hat\Gamma \big({\hat g},~{\hat h}\big)  ,\   \Lambda_{\delta_1} \hat h_{\delta}\big)_{L^2_v}\big| d t\right)^{1/2}d \Sigma(k)\notag \\
   &\quad+ C\int_{\mathbb Z^3}\left(\int_0^1  \big|\big(\big[\Lambda_{\delta_1}t^{2\varsigma}  (1-\delta\partial_{v_1})^{-1}  \partial_{v_1}^2, \ \mathcal L\big] \hat h,\   \Lambda_{\delta_1} \hat h_{\delta}\big)_{L^2_v}\big| d t\right)^{1/2}d \Sigma(k)\notag \\
   &\quad + C \int_{\mathbb Z^3}\left( \int_0^1 t^{-1} \norm{ \Lambda_{\delta_1} \hat h_{\delta}(t)}_{L^2_v}^2d t\right)^{1/2}d \Sigma(k)\notag \\
   &\quad+ C\int_{\mathbb Z^3}\left(\int_0^1   \big|\big(\big[iv\cdot k,\ \Lambda_{\delta_1}t^{ 2\varsigma}  (1-\delta\partial_{v_1}^2)^{-1}  \partial_{v_1}^2  \big] \hat h,\   \Lambda_{\delta_1} \hat h_{\delta}\big)_{L^2_v}\big| d t\right)^{1/2}d \Sigma(k).
   \label{eqhde} 
   \end{align}
 %\end{equation} 
In what follows we proceed through five steps to derive the upper bound for  those terms  on the right hand side.

\medskip
\noindent \underline{\it Step 1}.  We begin with the estimate of   the first term  on the right hand side of \eqref{eqhde}. We claim that
%\begin{equation}
	\begin{align}
		&\int_{\mathbb Z^3}\left(\int_0^1   \big|\big(\Lambda_{\delta_1}t^{ 2\varsigma}  (1-\delta\partial_{v_1}^2)^{-1}  \partial_{v_1}^2  \hat\Gamma \big({\hat g},~{\hat h}\big)  ,\   \Lambda_{\delta_1} \hat h_{\delta}\big)_{L^2_v}\big| d t\right)^{1/2}d \Sigma(k)\notag \\
		& \leq   \inner{\eps+CC_4\eps^{-1}\eps_0} \int_{\mathbb Z^3} \Big(\int_0^1    \norm{(a^{1/2})^w\hat h_{\delta}(t,k) }^2_{L_v^2}  d t\Big)^{1/2}d \Sigma(k)  + C  C_3C_4\eps^{-1} \eps_0  \norm{f_0}_{L_k^1L^2_v},\label{frest}
	\end{align}
%\end{equation}
where $C_3$ and $C_4$ are  the constants given in \eqref{spreg} and \eqref{inong}, respectively.   
Indeed, we follow the similar argument as for proving Lemma \ref{lemke} by letting    $m=2$ therein.   Precisely, with the notation $\mathcal T(g,h,w)$ defined by \eqref{matht},  similar to obtain \eqref{ijk}, 
we apply the Leibniz formula to write 
\begin{equation}\label{tildeijk}
	\int_{\mathbb Z^3}\left(\int_0^1   \big|\big(\Lambda_{\delta_1}t^{ 2\varsigma}  (1-\delta\partial_{v_1}^2)^{-1}  \partial_{v_1}^2  \hat\Gamma \big({\hat g},~{\hat h}\big)  ,\   \Lambda_{\delta_1} \hat h_{\delta}\big)_{L^2_v}\big| d t\right)^{1/2}d \Sigma(k)\leq \tilde I +\tilde J+\tilde K,
\end{equation}
with
 %\begin{eqnarray}
	\begin{align}
		\tilde I&=  \int_{\mathbb Z^3}\left(\int_0^1   \big|\big(\Lambda_{\delta_1}t^{2\varsigma}  (1-\delta\partial_{v_1}^2)^{-1} \hat  {\mathcal T}( \hat {g}, \ \partial_{v_1}^{2} \hat {h},  \  \mu^{1/2})
 ,\   \Lambda_{\delta_1} \hat h_{\delta}\big)_{L^2_v}\big| d t\right)^{1/2}d \Sigma(k)\notag\\
 &\quad  + \sum_{ p =1}^2 \begin{pmatrix}
      2    \\
      p  
\end{pmatrix}\int_{\mathbb Z^3}\left(\int_0^1   \big|\big(\Lambda_{\delta_1}t^{2\varsigma}  (1-\delta\partial_{v_1}^2)^{-1} \hat  {\mathcal T}( \hat {g}, \ \partial_{v_1}^{2- p } \hat {h},  \  \partial_{v_1}^{ p }\mu^{\frac{1}{2}})
 ,\   \Lambda_{\delta_1} \hat h_{\delta}\big)_{L^2_v}\big| d t\right)^{\frac{1}{2}}d \Sigma(k),\label{veti}\\
	\tilde J	&=  \int_{\mathbb Z^3}  \sum_{  0\leq p\leq 1}   \Big(\int_0^1   \big|\big(\Lambda_{\delta_1}t^{2\varsigma}  (1-\delta\partial_{v_1}^2)^{-1}\hat  {\mathcal T}(\partial_{v_1} \hat {g}, \ \partial_{v_1}^{1- p } \hat {h},  \ \partial_{v_1}^{ p } \mu^{\frac{1}{2}})
 ,\   \Lambda_{\delta_1} \hat h_{\delta}\big)_{L^2_v}\big| d t\Big)^{\frac{1}{2}}d \Sigma(k),\label{vetj}
 \end{align}
%\end{eqnarray}
and 
\begin{eqnarray}
	\begin{aligned}
 \tilde  K &= \int_{\mathbb Z^3}\left(\int_0^1   \big|\big(\Lambda_{\delta_1}t^{2\varsigma}  (1-\delta\partial_{v_1}^2)^{-1} \hat  {\mathcal T}( \partial_{v_1}^{2} \hat {g}, \ \hat {h},  \  \mu^{1/2})
 ,\   \Lambda_{\delta_1} \hat h_{\delta}\big)_{L^2_v}\big| d t\right)^{1/2}d \Sigma(k).	
 \label{vetk}
 \end{aligned}
\end{eqnarray}
We are going to estimate $\tilde I$, $\tilde J$ and $\tilde K$ as follows.

\smallskip
\noindent\underline{\bf Estimate on  $\tilde I$}. We first consider the first term on the right hand side of  $\tilde I$ in \eqref{veti}.   By using the formula \eqref{leib} for $m=2$, repeating the same argument as to estimate $I_1$ in \eqref{i1i2} and using the estimates \eqref{1234} and \eqref{inong}, we have
%\begin{eqnarray*}
\begin{align*}
		&\int_{\mathbb Z^3}\left(\int_0^1   \big|\big(\Lambda_{\delta_1}t^{ 2\varsigma}  (1-\delta\partial_{v_1}^2)^{-1}    \hat  {\mathcal T}( \hat {g}, \ \partial_{v_1}^{2} \hat {h},  \  \mu^{1/2})  ,\   \Lambda_{\delta_1} \hat h_{\delta}\big)_{L^2_v}\big| d t\right)^{1/2}d \Sigma(k)\\
		&\leq  \eps    \int_{\mathbb Z^3}\ \Big(\int_0^1  \norm{(a^{1/2})^w    \hat  h_{\delta}(t,k)}_{L^2_v}^2  d t\Big)^{1/2} d \Sigma(k)  +C C_4\eps^{-1} \eps_0\int_{\mathbb Z^3} \Big(\int_0^1  \norm{(a^{1/2})^w \hat  h_\delta }^2_{L_v^2}  d t\Big)^{\frac{1}{2}}d \Sigma(k) \\
   	&\quad+C C_3C_4  \eps^{-1} \eps_0\norm{f_0}_{L_k^1L_v^2}   + C  \sum_{0\leq p\leq 1}\int_{\mathbb Z^3}\bigg[\int_0^1 \delta \\
   	&\qquad\quad \times  \big|\big(\Lambda_{\delta_1}t^{ 2\varsigma}  (1-\delta\partial_{v_1}^2)^{-1} \hat {\mathcal T}(\partial_{v_1}^{1-p}  \hat  g,\ \partial_{v_1}(1-\delta\partial_{v_1}^2)^{-1}\partial_{v_1}^{2}  \hat h, \ \partial_{v_1}^{ p }\mu^{\frac{1}{2}})
 ,\   \Lambda_{\delta_1} \hat h_{\delta}\big)_{L^2_v}\big| d t\bigg]^{\frac{1}{2}}d \Sigma(k),
\end{align*}
%\end{eqnarray*}
where $\eps>0$ is an arbitrarily small constant. 
Furthermore, as for the last term on the right hand side of  the above estimate, we use
\begin{multline*}
	 \hat {\mathcal T}(\partial_{v_1}^{1-p}  \hat  g,\ \partial_{v_1}(1-\delta\partial_{v_1}^2)^{-1}\partial_{v_1}^{2}  \hat h, \ \partial_{v_1}^{ p }\mu^{\frac{1}{2}})=\partial_{v_1}\hat {\mathcal T}(\partial_{v_1}^{1-p}  \hat  g,\  (1-\delta\partial_{v_1}^2)^{-1}\partial_{v_1}^{2}  \hat h, \ \partial_{v_1}^{ p }\mu^{\frac{1}{2}})\\
 - \hat {\mathcal T}(\partial_{v_1}^{2-p}  \hat  g,\  (1-\delta\partial_{v_1}^2)^{-1}\partial_{v_1}^{2}  \hat h, \ \partial_{v_1}^{ p }\mu^{\frac{1}{2}})- \hat {\mathcal T}(\partial_{v_1}^{1-p}  \hat  g,\  (1-\delta\partial_{v_1}^2)^{-1}\partial_{v_1}^{2}  \hat h, \ \partial_{v_1}^{ p+1 }\mu^{\frac{1}{2}}),
\end{multline*}
so as to obtain, by observing that $\delta^{1/2} (1-\delta\partial_{v_1}^2)^{-1}\partial_{v_1}$ is uniformly bounded on $L_v^2$ with respect to $\delta$, 
%\begin{eqnarray*}
\begin{align*}
		& \sum_{0\leq p\leq 1}\int_{\mathbb Z^3}\bigg[\int_0^1 \delta \\
   	&\qquad\quad \times  \big|\big(\Lambda_{\delta_1}t^{ 2\varsigma}  (1-\delta\partial_{v_1}^2)^{-1} \hat {\mathcal T}(\partial_{v_1}^{1-p}  \hat  g,\ \partial_{v_1}(1-\delta\partial_{v_1}^2)^{-1}\partial_{v_1}^{2}  \hat h, \ \partial_{v_1}^{ p }\mu^{\frac{1}{2}})
 ,\   \Lambda_{\delta_1} \hat h_{\delta}\big)_{L^2_v}\big| d t\bigg]^{\frac{1}{2}}d \Sigma(k)\\
		&\leq    \eps    \int_{\mathbb Z^3}\ \Big(\int_0^1  \norm{(a^{1/2})^w    \hat  h_{\delta}(t,k)}_{L^2_v}^2  d t\Big)^{1/2} d \Sigma(k)+   C C_3C_4  \eps^{-1} \eps_0 \norm{f_0}_{L_k^1L_v^2} \\
		&\quad  +C  C_4 \eps^{-1}  \eps_0 \int_{\mathbb Z^3} \Big(\int_0^1  t^{2\varsigma}\norm{(a^{1/2})^w \delta^{1/2}(1-\delta\partial_{v_1}^2)^{-1}\partial_{v_1}^{2}  \hat h}^2_{L_v^2}  d t\Big)^{1/2}d \Sigma(k) \\
   	%&\quad \\
   	&\leq  \big(  \eps+CC_4  \eps^{-1}  \eps_0\big)    \int_{\mathbb Z^3}\ \Big(\int_0^1  \norm{(a^{1/2})^w    \hat  h_{\delta}(t,k)}_{L^2_v}^2  d t\Big)^{1/2} d \Sigma(k)+   CC_3C_4 \eps^{-1} \eps_0 \norm{f_0}_{L_k^1L_v^2} ,
\end{align*}
%\end{eqnarray*}
where the last inequality holds true because
\begin{multline*}
		t^{2\varsigma}\norm{(a^{1/2})^w \delta^{1/2}(1-\delta\partial_{v_1}^2)^{-1}\partial_{v_1}^{2}  \hat h}^2_{L_v^2}\\
		=  \norm{(a^{1/2})^w t^{2\varsigma} (1-\delta\partial_{v_1}^2)^{-1}\partial_{v_1}^{2}  \hat h}_{L_v^2}  \times \norm{(a^{1/2})^w \delta (1-\delta\partial_{v_1}^2)^{-1}\partial_{v_1}^{2}  \hat h}_{L_v^2}\\
		\leq    C\norm{(a^{1/2})^w   \hat h_{\delta}}_{L_v^2}   \norm{(a^{1/2})^w   \hat h}_{L_v^2}.
	\end{multline*}
Now, we combine the above estimates to conclude, for any $\eps>0,$ 
%\begin{equation}
	\begin{align}
		&\int_{\mathbb Z^3}\left(\int_0^1   \big|\big(\Lambda_{\delta_1}t^{ 2\varsigma}  (1-\delta\partial_{v_1}^2)^{-1}  \partial_{v_1}^2  \hat  {\mathcal T}( \hat {g}, \ \partial_{v_1}^{2} \hat {h},  \  \mu^{1/2})  ,\   \Lambda_{\delta_1} \hat h_{\delta}\big)_{L^2_v}\big| d t\right)^{1/2}d \Sigma(k)\notag \\
		&\leq   \big( \eps+ CC_4 \eps^{-1}  \eps_0\big)      \int_{\mathbb Z^3}\ \Big(\int_0^1  \norm{(a^{1/2})^w    \hat  h_{\delta}(t,k)}_{L^2_v}^2  d t\Big)^{1/2} d \Sigma(k)   +CC_3C_4 \eps^{-1} \eps_0 \norm{f_0}_{L_k^1L_v^2}.
		\label{tii1e}
	\end{align}
%\end{equation}
Next, we deal with the second term on the right hand side of  $\tilde I$ in \eqref{veti}. By direct computations,  it holds that
%\begin{equation}
\begin{align}
	 &\sum_{ p =1}^2 \begin{pmatrix}
      2    \\
      p  
\end{pmatrix}\int_{\mathbb Z^3}\left(\int_0^1   \big|\big(\Lambda_{\delta_1}t^{2\varsigma}  (1-\delta\partial_{v_1}^2)^{-1} \hat  {\mathcal T}( \hat {g}, \ \partial_{v_1}^{2- p } \hat h,  \  \partial_{v_1}^{ p }\mu^{\frac{1}{2}})
 ,\   \Lambda_{\delta_1} \hat h_{\delta}\big)_{L^2_v}\big| d t\right)^{\frac{1}{2}}d \Sigma(k)\notag\\
& \leq   \eps\int_{\mathbb Z^3} \Big(\int_0^1    \norm{(a^{1/2})^w\hat h_{\delta}(t,k) }^2_{L_v^2}  d t\Big)^{1/2}d \Sigma(k)  + C C_3C_4 \eps^{-1} \eps_0   \norm{f_0}_{L_k^1L^2_v}\notag \\
 &\quad +C \int_{\mathbb Z^3}\left(\int_0^1   \big|\big(\Lambda_{\delta_1}t^{2\varsigma}  (1-\delta\partial_{v_1}^2)^{-1} \hat  {\mathcal T}( \hat {g}, \ \partial_{v_1}  \hat h,  \  \partial_{v_1} \mu^{\frac{1}{2}})
 ,\   \Lambda_{\delta_1} \hat h_{\delta}\big)_{L^2_v}\big| d t\right)^{\frac{1}{2}}d \Sigma(k).
 \label{seti}
 \end{align}
%\end{equation}
To control the last term in the above inequality, we use the Leibniz formula to write, similar as in \eqref{leib}, 
\begin{multline*}
	\hat  {\mathcal T}(\hat {g}, \ \partial_{v_1}  \hat {h},  \   \partial_{v_1} \mu^{\frac{1}{2}})=(1-\delta\partial_{v_1}^2) \hat  {\mathcal T}(\hat {g}, \ (1-\delta\partial_{v_1}^2)^{-1}\partial_{v_1}  \hat {h},  \  \partial_{v_1}  \mu^{\frac{1}{2}})\\+
	 \delta \sum_{ j=1}^{2}  \sum_{p=0}^{j}\begin{pmatrix}
      2    \\
      j  
\end{pmatrix} \begin{pmatrix}
      j    \\
      p  
\end{pmatrix}\hat {\mathcal T}(\partial_{v_1}^{j- p } \hat  g,\ \partial_{v_1}^{2-j}(1-\delta\partial_{v_1}^2)^{-1}\partial_{v_1}   \hat h, \ \partial_{v_1}^{ p+1 }\mu^{1/2}).
\end{multline*}
As a result,   we may repeat the same argument  as for treating \eqref{leib} with $m=1$ by observing that the operator $\delta \partial_{v_1}^{2-j}(1-\delta\partial_{v_1}^2)^{-1}\partial_{v_1}$, $1\leq j\leq 2$, is uniformly bounded on $L_v^2$ with respect to $\delta$. Therefore,  by virtue of \eqref{1234}, \eqref{smacon} and \eqref{inong}, we conclude that  
%\begin{equation}
\begin{align}
&	\int_{\mathbb Z^3}\left(\int_0^1   \big|\big(\Lambda_{\delta_1}t^{2\varsigma}  (1-\delta\partial_{v_1}^2)^{-1} \hat  {\mathcal T}( \hat {g}, \ \partial_{v_1}  \hat h,  \  \partial_{v_1} \mu^{\frac{1}{2}})
 ,\   \Lambda_{\delta_1} \hat h_{\delta}\big)_{L^2_v}\big| d t\right)^{\frac{1}{2}}d \Sigma(k) \notag \\
 &\leq \eps\int_{\mathbb Z^3} \Big(\int_0^1    \norm{(a^{1/2})^w\hat h_{\delta}(t,k) }^2_{L_v^2}  d t\Big)^{1/2}d \Sigma(k) +C C_3C_4 \eps^{-1}\eps_0  \norm{f_0}_{L_k^1L^2_v} \notag \\
 &\quad +CC_4\eps^{-1}\eps_0\int_{\mathbb Z^3} \Big(\int_0^1    \norm{(a^{1/2})^w t^{2\varsigma} (1-\delta\partial_{v_1}^2)^{-1}\partial_{v_1}\hat h (t,k) }^2_{L_v^2}  d t\Big)^{1/2}d \Sigma(k) \notag \\
& \leq \inner{\eps+CC_4\eps^{-1}\eps_0}\int_{\mathbb Z^3} \Big(\int_0^1    \norm{(a^{1/2})^w\hat h_{\delta}(t,k) }^2_{L_v^2}  d t\Big)^{1/2}d \Sigma(k) +C C_3C_4 \eps^{-1}\eps_0  \norm{f_0}_{L_k^1L^2_v},
\label{fe2}
 \end{align}
%\end{equation}
where in the last inequality we have used the fact that 
\begin{equation*}%\label{lies}
	\begin{aligned}
		 &\norm{(a^{1/2})^w t^{2\varsigma} (1-\delta\partial_{v_1}^2)^{-1}\partial_{v_1}\hat h  }^2_{L_v^2}\leq C\norm{(a^{1/2})^w  \hat h_{\delta}  }^2_{L_v^2}+C \norm{(a^{1/2})^w  \hat h  }^2_{L_v^2}.
	\end{aligned}
\end{equation*}
Substituting \eqref{fe2} into \eqref{seti}, we obtain
%\begin{eqnarray*}
\begin{align*}
	 &\sum_{ p =1}^2 \begin{pmatrix}
      2    \\
      p  
\end{pmatrix}\int_{\mathbb Z^3}\left(\int_0^1   \big|\big(\Lambda_{\delta_1}t^{2\varsigma}  (1-\delta\partial_{v_1}^2)^{-1} \hat  {\mathcal T}( \hat {g}, \ \partial_{v_1}^{2- p } \hat h,  \  \partial_{v_1}^{ p }\mu^{\frac{1}{2}})
 ,\   \Lambda_{\delta_1} \hat h_{\delta}\big)_{L^2_v}\big| d t\right)^{\frac{1}{2}}d \Sigma(k)\\
& \leq \inner{\eps+CC_4\eps^{-1}\eps_0}\int_{\mathbb Z^3} \Big(\int_0^1    \norm{(a^{1/2})^w\hat h_{\delta}(t,k) }^2_{L_v^2}  d t\Big)^{1/2}d \Sigma(k) +C C_3C_4 \eps^{-1}\eps_0  \norm{f_0}_{L_k^1L^2_v}.
 \end{align*}
%\end{eqnarray*}
This together with \eqref{tii1e} give the estimate on $\tilde I$ in \eqref{veti} as   
\begin{equation*}%\label{tildei}
	\tilde I\leq \inner{\eps+CC_4\eps^{-1}\eps_0}\int_{\mathbb Z^3} \Big(\int_0^1    \norm{(a^{1/2})^w\hat h_{\delta}(t,k) }^2_{L_v^2}  d t\Big)^{1/2}d \Sigma(k) +C C_3C_4 \eps^{-1}\eps_0  \norm{f_0}_{L_k^1L^2_v},
\end{equation*}
where  $\eps>0$ is an arbitrarily small constant.

\medskip
\noindent\underline{\bf Estimate on  $\tilde J$}. Recall that $\tilde J$ is given in \eqref{vetj}. Following a similar argument as that for proving \eqref{j}, we can verify directly  that
%\begin{equation}
		\begin{align}
		\tilde  J & \leq \eps    \int_{\mathbb Z^3}\ \Big(\int_0^1  \norm{(a^{1/2})^w    \hat  h_{\delta}(t,k)}_{L^2_v}^2  d t\Big)^{1/2} d \Sigma(k) + CC_3C_4 \eps^{-1} \eps_0   \norm{f_0}_{L_k^1L^2_v}\notag \\
		&\quad +\int_{\mathbb Z^3}    \Big(\int_0^1   \big|\big(\Lambda_{\delta_1}t^{2\varsigma}  (1-\delta\partial_{v_1}^2)^{-1}\hat  {\mathcal T}(\partial_{v_1} \hat {g}, \ \partial_{v_1}  \hat {h},  \   \mu^{\frac{1}{2}})
 ,\   \Lambda_{\delta_1} \hat h_{\delta}\big)_{L^2_v}\big| d t\Big)^{\frac{1}{2}}d \Sigma(k).
 \label{j1}
	\end{align}
%\end{equation}
It remains to treat   the last term in \eqref{j1}.   Repeating  the argument for proving  \eqref{fe2} gives
%\begin{eqnarray*}
\begin{align*}
	&	\int_{\mathbb Z^3}    \Big(\int_0^1   \big|\big(\Lambda_{\delta_1}t^{2\varsigma}  (1-\delta\partial_{v_1}^2)^{-1}\hat  {\mathcal T}(\partial_{v_1} \hat {g}, \ \partial_{v_1}  \hat {h},  \   \mu^{\frac{1}{2}})
 ,\   \Lambda_{\delta_1} \hat h_{\delta}\big)_{L^2_v}\big| d t\Big)^{\frac{1}{2}}d \Sigma(k)\\
 &\leq\eps\int_{\mathbb Z^3} \Big(\int_0^1    \norm{(a^{1/2})^w\hat h_{\delta}(t,k) }^2_{L_v^2}  d t\Big)^{1/2}d \Sigma(k) +C C_3C_4 \eps^{-1}\eps_0  \norm{f_0}_{L_k^1L^2_v}\\
 &\quad +CC_4\eps^{-1}\eps_0\int_{\mathbb Z^3} \Big(\int_0^1    \norm{(a^{1/2})^w t^{\varsigma} (1-\delta\partial_{v_1}^2)^{-1}\partial_{v_1}\hat h (t,k) }^2_{L_v^2}  d t\Big)^{1/2}d \Sigma(k)\\
 &\quad + C\int_{\mathbb Z^3}    \Big(\int_0^1  \delta  \big|\big(\Lambda_{\delta_1}t^{2\varsigma}  (1-\delta\partial_{v_1}^2)^{-1} \hat {\mathcal T}(\partial_{v_1}^{3} \hat  g,\   (1-\delta\partial_{v_1}^2)^{-1}\partial_{v_1}   \hat h, \  \mu^{1/2})
 ,\   \Lambda_{\delta_1} \hat h_{\delta}\big)_{L^2_v}\big| d t\Big)^{\frac{1}{2}}d \Sigma(k)\\
&\leq \inner{\eps+CC_4\eps^{-1}\eps_0}\int_{\mathbb Z^3} \Big(\int_0^1    \norm{(a^{1/2})^w\hat h_{\delta}(t,k) }^2_{L_v^2}  d t\Big)^{1/2}d \Sigma(k) +C C_3C_4 \eps^{-1}\eps_0  \norm{f_0}_{L_k^1L^2_v},
\end{align*}
%\end{eqnarray*}
where in the last inequality we have used the following two  estimates that
\begin{eqnarray*}
	\begin{aligned}
		  \norm{(a^{1/2})^w t^{\varsigma} (1-\delta\partial_{v_1}^2)^{-1}\partial_{v_1}\hat h  }^2_{L_v^2}\leq C\norm{(a^{1/2})^w  \hat h_{\delta}  }^2_{L_v^2}+C \norm{(a^{1/2})^w  \hat h  }^2_{L_v^2}, 	\end{aligned}
\end{eqnarray*}
 and that
\begin{multline*}
	\int_{\mathbb Z^3}    \Big(\int_0^1  \delta  \big|\big(\Lambda_{\delta_1}t^{2\varsigma}  (1-\delta\partial_{v_1}^2)^{-1} \hat {\mathcal T}(\partial_{v_1}^{3} \hat  g,\   (1-\delta\partial_{v_1}^2)^{-1}\partial_{v_1}   \hat h, \  \mu^{1/2})
 ,\   \Lambda_{\delta_1} \hat h_{\delta}\big)_{L^2_v}\big| d t\Big)^{\frac{1}{2}}d \Sigma(k)\\
 \leq  \eps    \int_{\mathbb Z^3}\ \Big(\int_0^1  \norm{(a^{1/2})^w    \hat  h_{\delta}(t,k)}_{L^2_v}^2  d t\Big)^{1/2} d \Sigma(k) + C C_3C_4\eps^{-1} \eps_0 \norm{f_0}_{L_k^1L^2_v},
\end{multline*}
which follows from the formula that
\begin{multline*}
	\delta\hat {\mathcal T}(\partial_{v_1}^{3} \hat  g,\   (1-\delta\partial_{v_1}^2)^{-1}\partial_{v_1}   \hat h, \  \mu^{1/2})=\delta^{1/2}	\partial_{v_1}\hat {\mathcal T}(\partial_{v_1}^{2} \hat  g,\  \delta^{1/2} (1-\delta\partial_{v_1}^2)^{-1}\partial_{v_1}   \hat h, \  \mu^{1/2})\\
	-	\hat {\mathcal T}(\partial_{v_1}^{2} \hat  g,\  \delta (1-\delta\partial_{v_1}^2)^{-1}\partial_{v_1}^2   \hat h, \  \mu^{1/2})-	\hat {\mathcal T}(\partial_{v_1}^{2} \hat  g,\   \delta(1-\delta\partial_{v_1}^2)^{-1}\partial_{v_1}   \hat h, \ \partial_{v_1}  \mu^{1/2}).
\end{multline*}
Now, we substitute the above estimate into \eqref{j1} to conclude 
\begin{equation*}
%\label{tildej}
\tilde J\leq 	\inner{\eps+CC_4\eps^{-1}\eps_0}\int_{\mathbb Z^3} \Big(\int_0^1    \norm{(a^{1/2})^w\hat h_{\delta}(t,k) }^2_{L_v^2}  d t\Big)^{1/2}d \Sigma(k) +C C_3C_4 \eps^{-1}\eps_0  \norm{f_0}_{L_k^1L^2_v}.
\end{equation*}

\medskip
\noindent\underline{\bf Estimate on  $\tilde K$}.  Recall that $\tilde K$ is given in \eqref{vetk}.   Similar to \eqref{k}, we have, using the estimates  \eqref{smacon}-\eqref{inong}, 
\begin{eqnarray*}
	\begin{aligned}
	\tilde K  \leq \eps\int_{\mathbb Z^3} \Big(\int_0^1    \norm{(a^{1/2})^w\hat h_{\delta}(t,k) }^2_{L_v^2}  d t\Big)^{1/2}d \Sigma(k)  + CC_3C_4 \eps^{-1} \eps_0   \norm{f_0}_{L_k^1L^2_v}. 
	\end{aligned}
\end{eqnarray*}
Finally, we can substitute all the above estimates on $\tilde I$, $\tilde J$ and $\tilde K$ to \eqref{tildeijk} so as to conclude the desired estimate \eqref{frest}.
 
 \medskip
 \noindent \underline{\it Step 2}.  In this step we treat the second term on the right hand side of \eqref{eqhde}. We claim that  
 %\begin{equation}
 	\begin{align}
 		&\varlimsup_{\delta_1\rightarrow 0}\int_{\mathbb Z^3}\left(\int_0^1  \big|\big(\big[\Lambda_{\delta_1}t^{2\varsigma}  (1-\delta\partial_{v_1})^{-1}  \partial_{v_1}^2, \ \mathcal L\big] \hat h,\   \Lambda_{\delta_1} \hat h_{\delta}\big)_{L^2_v}\big| d t\right)^{1/2}d \Sigma(k)\notag \\
 		& \leq  \eps \int_{\mathbb Z^3} \Big(\int_0^1    \norm{(a^{1/2})^w\hat h_{\delta}(t,k) }^2_{L_v^2}  d t\Big)^{1/2}d \Sigma(k)  + C   C_3\eps^{-1}    \norm{f_0}_{L_k^1L^2_v},\label{lide}
 	\end{align}
%\end{equation}
with the constant $C_3$ given in \eqref{spreg}. Indeed,  the above estimate follows from the similar arguments as those in the proof of Lemma \ref{lemke++} by letting $m=2$ therein  as well as in the previous step 1 with slight modifications.  We omit the details for brevity. 
 
  \medskip
\noindent \underline{\it Step 3.}  As for the third term on the right hand side of \eqref{eqhde}, we use the argument for proving    Lemma \ref{lemla} to get
%\begin{equation}
 \begin{align}
 	&\int_{\mathbb Z^3}\left( \int_0^1 t^{-1} \norm{ \Lambda_{\delta_1} \hat h_{\delta}(t)}_{L^2_v}^2d t\right)^{1/2}d \Sigma(k)\notag \\
    			& \leq \eps       \int_{\mathbb Z^3}\left(\int_0^1  \norm{ (a^{1/2})^w  \hat h_{\delta}}_{L^2_v}^2d t\right)^{\frac12}d \Sigma(k) +
 		\int_{\mathbb Z^3}\left(\int_0^1 t^{-2 \varsigma }\norm{\comi{D_{v}}^{\frac{s}{1+2s}-1} \hat h_{\delta}}_{L^2_v}^2d t\right)^{\frac12}d \Sigma(k)\notag \\
 		  &\leq  \eps \int_{\mathbb Z^3} \Big(\int_0^1    \norm{(a^{1/2})^w\hat h_{\delta}(t,k) }^2_{L_v^2}  d t\Big)^{1/2}d \Sigma(k)  + C   C_3\eps^{-1}    \norm{f_0}_{L_k^1L^2_v},\label{tmin}
 		  \end{align}
 %\end{equation}
where in the last inequality we have used the fact that,  in view of the representation \eqref{hel} of $\hat h_\delta,$  \begin{eqnarray*}
 	\begin{aligned}
 		&\int_{\mathbb Z^3}\left(\int_0^1 t^{-2 \varsigma }\norm{\comi{D_{v}}^{\frac{s}{1+2s}-1} \hat h_{\delta}}_{L^2_v}^2d t\right)^{\frac12}d \Sigma(k)\\
 		&\leq \eps  \int_{\mathbb Z^3}\left(\int_0^1  \norm{(a^{1/2})^w\hat h_{\delta}}_{L^2_v}^2d t\right)^{\frac12}d \Sigma(k)+C\eps^{-1}\int_{\mathbb Z^3}\left(\int_0^1  \norm{(a^{1/2})^w \hat h}_{L^2_v}^2d t\right)^{\frac12}d \Sigma(k).
 	\end{aligned}
 \end{eqnarray*}

\medskip
\noindent\underline{\it Step 4.}  It remains to deal with the last term on the right hand side of \eqref{eqhde}.  Using \eqref{comest} for $m=2$ gives
 \begin{multline*}
 	\norm{[iv\cdot k,\ \Lambda_{\delta_1}t^{2\varsigma}  (1-\delta\partial_{v_1}^2)^{-1}  \partial_{v_1} ^2] \hat h}_{L_v^2}\leq C   t^{2\varsigma}\comi k\norm{   (1-\delta\partial_{v_1}^2)^{-1}  \partial_{v_1}    \hat h}_{L_v^2}\\
 \leq C   t^{2\varsigma}\comi k^2\norm{       \hat h}_{L_v^2}+C  \norm{ \hat h_\delta}_{L_v^2}.
 \end{multline*}
 Then, it holds that
% \begin{equation}
 	\begin{align}
 		&\int_{\mathbb Z^3}\left(\int_0^1  \big|\big(\big[\Lambda_{\delta_1}t^{\varsigma}  (1-\delta\partial_{v_1})^{-1}  \partial_{v_1}, \ \mathcal L\big] \hat h,\   \Lambda_{\delta_1} \hat h_{\delta}\big)_{L^2_v}\big| d t\right)^{1/2}d \Sigma(k)\notag \\
 		& \leq  C \int_{\mathbb Z^3}\comi k^2 \left(\int_0^1  t^{4\varsigma}\norm{ (a^{1/2})^w\hat h (t,v)}_{L^2_v}^2  d t\right)^{\frac{1}{2}}d \Sigma(k)+ C \int_{\mathbb Z^3}\left(\int_0^1  \norm{\hat h_{\delta}(t,k)}_{L^2_v}^2  d t\right)^{\frac{1}{2}}d \Sigma(k)\notag \\
 		&\leq  \eps \int_{\mathbb Z^3} \Big(\int_0^1    \norm{(a^{1/2})^w\hat h_{\delta}(t,k) }^2_{L_v^2}  d t\Big)^{1/2}d \Sigma(k)  + C   C_3\eps^{-1}    \norm{f_0}_{L_k^1L^2_v},  \label{covk}
 	\end{align}
%\end{equation}
where in the last inequality we have used  \eqref{spreg} and \eqref{tmin}. 
 
\medskip
\noindent\underline{\it Step 5.} 
 Finally we substitute the estimates \eqref{frest},  \eqref{lide}, \eqref{tmin} and \eqref{covk} into \eqref{eqhde} and let $\delta_1\rightarrow 0$. Choosing  $\eps>0$ small enough, it follows that  
\begin{multline*}
 	\int_{Z^3 } \sup_{0<t\leq 1}  \norm{   \hat h_{\delta}(t,k)}_{ L_v^2}d \Sigma(k)   \\
 	+ \int_{\mathbb Z^3}\left(\int_0^1 \norm{(a^{1/2})^w \hat h_{\delta}(t,k)}_{L^2_v}^2d t\right)^{1/2}d \Sigma(k)\leq  C  C_3 C_4  \eps_0 \norm{f_0}_{L_k^1L^2_v}+C  C_3   \norm{f_0}_{L_k^1L^2_v}.
 \end{multline*}
 Moreover, letting $\delta\rightarrow 0$, we further obtain that, in view of the definition \eqref{hel} of $\hat h_\delta,$
 \begin{multline*}
 		 \int_{\mathbb Z^3} \sup_{0<t\leq 1} t^{2\varsigma   }\norm{  \partial_{v_1}^2 \hat  h(t,k)}_{ L^2_v} d \Sigma(k) 
	 + \int_{\mathbb Z^3}	\left(\int_0^1 t ^{4\varsigma  }\norm{(a^{1/2})^w  \partial_{v_1}^2 \hat  h(t,k)}_{ L^2_v}^2d t\right)^{\frac12}d \Sigma(k) \\ 
	 \leq  C  C_3 C_4  \eps_0 \norm{f_0}_{L_k^1L^2_v}+C  C_3   \norm{f_0}_{L_k^1L^2_v}.
	 \end{multline*} 
Notice that  the above estimate still holds with $\partial_{v_1}^2$ replaced by $\partial_{v_2}^2$
  or $\partial_{v_3}^2$.  Then, we conclude that for any $\abs\beta=2$,
\begin{multline}\label{tes}
	 \int_{\mathbb Z^3} \sup_{0<t\leq 1} t^{ \varsigma \abs\beta  }\norm{  \partial_{v}^\beta \hat  h(t,k)}_{ L^2_v} d \Sigma(k)
	  + \int_{\mathbb Z^3}	\left(\int_0^1 t ^{2\varsigma\abs\beta  }\norm{(a^{1/2})^w  \partial_{v}^\beta \hat  h(t,k)}_{ L^2_v}^2d t\right)^{\frac12}d \Sigma(k)\\ 
	 \leq  C  C_3 C_4  \eps_0 \norm{f_0}_{L_k^1L^2_v}+C  C_3   \norm{f_0}_{L_k^1L^2_v}.
\end{multline}
Meanwhile, the case of $T\geq 1$ can be treated in a similar way by combining the  above estimate for $h|_{t=1}$. Thus the  desired estimate of Lemma  \ref{le2}  follows,  provided that $\eps_0>0$ is small enough and  $C_4$ is chosen large enough such that $C_4>4(C+1)C_3$ with $C$ the constant in \eqref{tes}.  The proof of Lemma \ref{le2} is completed.
\end{proof}

\begin{lemma}[$\abs\beta=1$]\label{le1}  
Under the same conditions on $f_0$ and $g$ as in Proposition \ref{lemlin}, the estimate \eqref{dees}  holds for any $\beta\in\mathbb Z_+^3$ with $\abs\beta=1$, provided that $\eps_0$  is small enough.  
 \end{lemma}
 
\begin{proof}
The desired estimate is an immediate  consequence of \eqref{tes} and \eqref{1234} by   observing that for $\abs\beta=1$,
    \begin{eqnarray*}
  	\norm{\phi(t)^{\varsigma}\partial_v^{\beta}\hat h}_{L_v^2}^2 \leq  \frac{1}{2}\inner{ 	\norm{\phi(t)^{2\varsigma}\partial_v^{2\beta}\hat h}_{L_v^2}^2+ 	\norm{  \hat h}_{L_v^2}^2}   \end{eqnarray*}
  and
    \begin{eqnarray*}
  	 \norm{\phi(t)^{\varsigma}(a^{1/2})^w\partial_v^{\beta}\hat h}_{L_v^2}^2 \leq C \inner{ 	\norm{\phi(t)^{2\varsigma}(a^{1/2})^w\partial_v^{2\beta}\hat h}_{L_v^2}^2+ 	\norm{  (a^{1/2})^w\hat h}_{L_v^2}^2}.
  \end{eqnarray*} 
 We complete  the proof of Lemma \ref{le1}. 
 \end{proof}

\begin{proof}[Proof of Proposition \ref{lemlin}]
It directly follows from \eqref{1234} and Lemmas \ref{le2}-\ref{le1}. 
 \end{proof}

 With Proposition \ref{lemlin}, the rest is devoted to proving the main result on the $H_v^2$-smoothness of solutions to the Cauchy problem on the  nonlinear Boltzmann equation. 
 
 \begin{proof}[Proof of Proposition \ref{lem}] We consider the following iteration equations with $f^0\equiv 0$,
\begin{equation*}%\label{lin}
 	\partial_{t}f^{n}+v\cdot\partial_{x}f^{n}-\mathcal{L} f^{n+1}
 	=\Gamma(f^{n-1},f^{n}), \quad f^n|_{t=0}=f_0, \quad n\geq 1,
\end{equation*}
where $f_0$ is the initial datum to the Boltzmann equation \eqref{eqforper} satisfying the smallness condition \eqref{123}.   Then, Proposition \ref{lemlin} ensures the existence of   $\{f^n\}_{n\geq 1}$ satisfying the estimate that, for any $T>0$ and any $n\geq 1$,  	
 	  \begin{multline} \label{sees}
 		 \int_{\mathbb Z^3} \sup_{0<t< T} \phi(t) ^{\varsigma \abs\beta }\norm{  \partial_v^\beta \widehat {f^n}(t,k)}_{ L^2_v} d \Sigma(k)\\
	 + \int_{\mathbb Z^3}	\left(\int_0^T  \phi(t) ^{2\varsigma \abs\beta }\norm{(a^{1/2})^w  \partial_v^\beta\widehat {f^n}(t,k)}_{ L^2_v}^2d t\right)^{\frac12}d \Sigma(k) 
	 \leq    C_4 \norm{f_0}_{L_k^1L^2_v} .
\end{multline} 
 	 Consider the difference 
 	\begin{equation*}
 		w^n=f^{n+1}-f^{n},\quad n\geq 1. 
 	\end{equation*}
Then for any $n\geq 1,$
 	\begin{equation*}
 		\partial_{t}w^n+v\cdot\partial_{x}w^n-\mathcal{L} w^n
 	=\Gamma(f^n, w^{n})+\Gamma(w^{n-1},f^{n}),\quad w^n|_{t=0}=0, 
 	\end{equation*}
 	with $w^0=f^1$.  By virtue of \eqref{sees} as well as the smallness of $ \norm{f_0}_{L_k^1L^2_v}$,  we follow the argument in the proof of Proposition \ref{lemlin}  with minor modifications
to obtain, for any $\beta\in\mathbb Z_+^3$  with $\abs\beta\leq 2,$
 %\begin{eqnarray*}
 \begin{align*}
	& \int_{\mathbb Z^3} \sup_{0<t< T} \phi(t) ^{\varsigma \abs\beta }\norm{  \partial_v^\beta \widehat  { w^n} (t,k)}_{ L^2_v} d \Sigma(k)  + \int_{\mathbb Z^3}	\left(\int_0^T  \phi(t) ^{2\varsigma \abs\beta }\norm{(a^{1/2})^w  \partial_v^\beta \widehat  { w^n}(t,k)}_{ L^2_v}^2d t\right)^{\frac12}d \Sigma(k) \\
	&\leq  C \Big( \int_{\mathbb Z^3} \sup_{0<t<T} \phi(t) ^{\varsigma \abs\beta }\norm{  \partial_v^\beta \widehat  {w^{n-1}} (t,k)}_{ L^2_v} d \Sigma(k)    \Big)\\
	 &\qquad\qquad\qquad\qquad\qquad\times \int_{\mathbb Z^3}	\left(\int_0^T  \phi(t) ^{2\varsigma \abs\beta }\norm{(a^{1/2})^w  \partial_v^\beta\widehat  { f^n}(t,k)}_{ L^2_v}^2d t\right)^{\frac12}d \Sigma(k)\\
	 &\leq  C C_4\norm{f_0}_{L_k^1L_v^2} \int_{\mathbb Z^3} \sup_{0<t<T} \phi(t) ^{\varsigma \abs\beta }\norm{  \partial_v^\beta \widehat  {w^{n-1}} (t,k)}_{ L^2_v} d \Sigma(k), 
	 \end{align*}
%\end{eqnarray*} 	
 which implies that, for $\norm{f_0}_{L_k^1L_v^2}$ small enough, 	
 \begin{eqnarray*}
 	 \int_{\mathbb Z^3} \sup_{0<t<T} \phi(t) ^{\varsigma \abs\beta }\norm{  \partial_v^\beta \widehat  { w^n} (t,k)}_{ L^2_v} d \Sigma(k) \leq \big(C C_4\norm{f_0}_{L_k^1L_v^2} \big)^{n+1}\leq 2^{-n-1},\quad n\geq 1. 
 \end{eqnarray*}
 Thus,  combining the above estimates, it holds that
  \begin{multline*}%\label{dif}
  \int_{\mathbb Z^3} \sup_{0<t<T} \phi(t) ^{\varsigma \abs\beta }\norm{  \partial_v^\beta \widehat  { w^n} (t,k)}_{ L^2_v} d \Sigma(k)  \\+ \int_{\mathbb Z^3}	\left(\int_0^T  \phi(t) ^{2\varsigma \abs\beta }\norm{(a^{1/2})^w  \partial_v^\beta \widehat  { w^n}(t,k)}_{ L^2_v}^2d t\right)^{\frac12}d \Sigma(k) 
	 	 \leq  2^{-n-1}. 
	 \end{multline*}	
	This implies that for any $\abs\beta\leq 2$,   $\phi(t) ^{\varsigma \abs\beta } \partial_v^\beta f^n $  and  $\phi(t) ^{\varsigma \abs\beta } (a^{1/2})^w \partial_v^\beta f^n $ are the Cauchy sequences in   $L_k^1L_T^\infty L_v^2$ and $L_k^1L_T^\infty L_v^2$, respectively,  with the limit solving the nonlinear Boltzmann equation \eqref{eqforper} with initial datum $f_0$ and thus equal to $f$ by  the uniqueness of solutions in $L_k^1L_T^\infty L_v^2$.  Moreover, it follows from \eqref{sees} that
	\begin{multline*} 
 		 \int_{\mathbb Z^3} \sup_{0<t<T} \phi(t) ^{\varsigma \abs\beta }\norm{  \partial_v^\beta \hat {f}(t,k)}_{ L^2_v} d \Sigma(k)\\
	 + \int_{\mathbb Z^3}	\left(\int_0^T  \phi(t) ^{2\varsigma \abs\beta }\norm{(a^{1/2})^w  \partial_v^\beta\hat {f}(t,k)}_{ L^2_v}^2d t\right)^{\frac12}d \Sigma(k) 
	 \leq    C_4  \norm{f_0}_{L_k^1L^2_v}.
	 \end{multline*} 
We then have proved  the  desired result in Proposition \ref{lem}, completing the proof.  
 \end{proof}
 
 %%%%%%%%%%%%%%%%%%%%%%%%%%%%%%%%%%%%%%%%%%%%%%%%%%%%%%%%%%%%%%%%%%%%%%%%%%%%%%%%%%%%%%%%%%%%% 
 
 \section{Gevrey smoothing effect in  space and velocity variables }\label{sec5}
We are ready to prove the main result,  
Theorem \ref{maith}, which is just an immediate consequence of Theorems \ref{thm3.1} and  \ref{thm3.2}.  In fact, observe that 
for any $m\in\mathbb Z_+$ and any $\beta \in \mathbb{Z}_+^3$,
\begin{eqnarray*}
	\comi{k}^{m}	  \phi(t) ^{\varsigma (m+\abs\beta) }\norm{  \partial_v^\beta \hat  f(t,k)}_{ L^2_v} \leq \comi{k}^{m}	  \phi(t) ^{\varsigma (m+\abs\beta) }     \norm{    \hat  f(t,k)}_{ L^2_v}^{1/2}\norm{  \partial_v^{2\beta} \hat  f(t,k)}_{ L^2_v} ^{1/2}.
\end{eqnarray*}
Thus, it holds that
	 	\begin{equation*} 
	 \begin{aligned}
	& \int_{\mathbb Z^3} \comi{k}^{m}	 \sup_{0<t<T} \phi(t) ^{\varsigma (m+\abs\beta) }\norm{  \partial_v^\beta \hat  f(t,k)}_{ L^2_v} d \Sigma(k) \\
	& \leq  \left( \int_{\mathbb Z^3} \comi{k}^{2m}	 \sup_{0<t<T} \phi(t) ^{2m\varsigma  }\norm{   \hat  f(t,k)}_{ L^2_v} d \Sigma(k)\right)^{\frac12}  \left(\int_{\mathbb Z^3}  	 \sup_{0<t<T} \phi(t) ^{2\varsigma  \abs\beta }\norm{  \partial_v^{2\beta} \hat  f(t,k)}_{ L^2_v} d \Sigma(k)\right)^{\frac12}\\
	& \leq \left( \tilde C_0^{2m +1}   [(2m)!]   ^{\frac{1+2s}{2s}}\right)^{1/2} \left( \tilde C_*^{2|\beta| +1}   [(2|\beta|)!]   ^{\frac{1+2s}{2s}} \right)^{1/2} \\
	& \leq   \tilde C_0^{m+\frac12}  \tilde C_*^{ |\beta| +\frac12} 2^{  (m+\abs{\beta})\frac{1+2s}{2s} }  ( m !)^{\frac{1+2s}{2s}} ( |\beta| !)   ^{\frac{1+2s}{2s}}  \\
	 &\leq  C^{ m+\abs{\beta} +1}      [ (m+|\beta|) !] ^{\frac{1+2s}{2s}},
	 \end{aligned}
	 \end{equation*}
where in the second inequality we have used    Theorems \ref{thm3.1} and  \ref{thm3.2}, the third inequality follows from the fact that   
	 $(m+n)! \leq 2^{m+n }m! n!$, and the last inequality holds  if  we have chosen 
 $C=2^{   \frac{1+2s}{2s} } (\tilde C_0+ \tilde C_*) $  with the constants $\tilde C_0$ and $\tilde C_*$   given in Theorem \ref{thm3.1} and Theorem \ref{thm3.2}, respectively.  This gives the desired estimate \eqref{thm.gest} in Theorem \ref{maith}.  
 In the end, we briefly explain $f(t,\cdot,\cdot) \in \CG^{\frac{1+2s}{2s}}(\mathbb{T}^3\times\mathbb{R}^3)$ for any positive time $t>0$ by referring to the direct computations as 
	 	\begin{equation*} 
	 \begin{split}
	 &\sup_{0<t< T}  \phi(t) ^{\varsigma (|\alpha|+\abs\beta) }   	\norm{   \partial_x^\alpha\partial_v^\beta    f(t,x,v)}_{L^2_x L^2_v}    \\
	  \leq&\sup_{0<t< T}  \phi(t) ^{\varsigma (|\alpha|+\abs\beta) } \left(\int_{{\mathbb T^3} } \comi{k}^{2|\alpha|} 	\norm{  \partial_v^\beta \hat  f(t,k)}_{  L^2_v}^2 d \Sigma(k)\right) ^{\frac12}\\
	  \leq&\sup_{0<t< T}  \phi(t) ^{\varsigma (|\alpha|+\abs\beta) } \int_{{\mathbb T^3} } \comi{k}^{|\alpha|} 	\norm{  \partial_v^\beta \hat  f(t,k)}_{  L^2_v} d \Sigma(k) \\
	  \leq& \int_{{\mathbb T^3} } \comi{k}^{|\alpha|} \sup_{0<t< T}  \phi(t) ^{\varsigma (|\alpha|+\abs\beta) }	\norm{  \partial_v^\beta \hat  f(t,k)}_{  L^2_v}  ~ \,d\Sigma(k) \\
	  \leq& C^{ |\alpha|+\abs{\beta} +1}      [ (|\alpha|+|\beta|) !]   ^{\frac{1+2s}{2s}}.
	 \end{split}
	 \end{equation*}
	  Then the proof of Theorem \ref{maith} is completed.\qed
  
%%%%%%%%%%%%%%%%%%%%%%%%%%%%%%%%%%%%%%%%%%%%%%%%%%

%\appendix
%\section{Some facts on Symbolic calculus}\label{secapp}
\section{Appendix}\label{secapp}

We recall here  some  notations and basic facts  of  symbolic
calculus, and refer to
\cite[Chapter 18]{MR781536} or \cite{MR2599384}  for detailed discussions on the pseudo-differential
calculus. 

We consider the flat metric $ \abs{dv}^2+\abs{d\eta}^2$,
and let $M$ be an admissible weight function with respect to $ \abs{dv}^2+\abs{d\eta}^2$, that is, the weight function $M$ satisfies the following conditions:
\begin{enumerate}[(a)]
\item (slowly varying condition) there exists a constant $\delta$ such that
\begin{eqnarray*}
\abs{X-Y}\leq \delta, \quad M(X) \approx M(Y),\quad \forall \, X, Y\in \mathbb R^6;
\end{eqnarray*}
\item (temperance)  there exist two constants $C$ and $N$ such that
\begin{eqnarray*}
M(X)/M(Y) \leq C
\comi{X-Y}^N,\quad \forall\,X, Y\in \mathbb R^6.
\end{eqnarray*}
\end{enumerate}
Considering  symbols $q(k, v,\eta)$ as a function of $(v,\eta)$ with
parameter $k$,    we say that  $q\in S(M, \abs{dv}^2+\abs{d\eta}^2)$ uniformly
with respect to $k$, if
\[
\abs{\partial_v^\alpha\partial_\eta^\beta q(k,v,\eta)}\leq C_{\alpha,
\beta} M,\quad \forall\, \alpha, \beta\in\mathbb Z_+^3,~~\forall\,v,\eta\in\mathbb R^3,
\]
with $ C_{\alpha,\beta}$ a constant depending only on $\alpha$ and $\beta$, but
independent of $k$.   The space $S(M, \abs{dv}^2+\abs{d\eta}^2)$ endowed with the semi-norms
\begin{eqnarray*}
 \norm{q}_{N; S(M, \abs{dv}^2+\abs{d\eta}^2)}= \max_{0 \leq \abs\alpha+\abs\beta \leq N} \sup_{(v,\eta)\in \mathbb
 R^6} \abs{M(v,\eta)^{-1}\partial_v^\alpha\partial_\eta^\beta q (v,\eta)},
\end{eqnarray*}
becomes a Fr\'echet space.  
 Let $q\in \mathscr S'(\mathbb R_v^3\times\mathbb R_\eta^3)$ be a tempered distribution and let $t\in\mathbb R$,   the operator ${\rm op}_t q$ is an operator from  $ \mathscr S(\mathbb R_v^3)$ to $ \mathscr S'(\mathbb R_v^3),$  whose Schwartz kernel $K_t$ is defined by the oscillatory integral:    
 \[
    K_t  (z, z')= (2\pi)^{-3} \int_{\mathbb R^3} e^{i (z-z') \cdot\zeta}q((1-t)z+tz', \zeta) d\zeta.
  \]
In particular we denote $q^w={\rm op}_{1/2}q$.
 Here $q^w$ is called  the Weyl quantization of symbol $q$.   Note that if $q\in S(1, \ \abs{dv}^2+\abs{d\eta}^2 )$  and $q$ is real-valued then $q^w$ is bounded and self-adjoint in $L_v^2$. 

  Finally let us  recall some  basic properties of the Wick quantization.   The importance in studying the Wick quantization lies in the fact that  positive symbols give rise to positive operators.   We  refer
the readers to Lerner's
works   \cite{MR1957713,MR2599384} and references therein  for extensive presentations of this
quantization and  its applications in  mathematics and mathematical physics.  

 Let $Y = (v, \eta)$  be a point in $
\mathbb R^{6}$.  The Wick quantization of a symbol $q$  is  given by
\begin{eqnarray*}
	q^{{\rm Wick}} =(2\pi)^{-3}\int_{\mathbb R^6} q(Y) \Pi_{Y}\ dY,
\end{eqnarray*}
where  $\Pi_Y$ is the projector associated to the Gaussian $\varphi_Y$ which is defined by 
\[
   \varphi_Y(z)=\pi^{-3/4}e^{-\frac{1}{2}\abs{z-v}^2}e^{i z \cdot\eta /2},\quad z\in
\mathbb R^3.
\]
 The main property of the Wick quantization is its positivity, i.e.,
\begin{equation}\label{posit}
  q(v,\eta)\geq 0 ~~\,\textrm{for all}~ (v,\eta)\in\mathbb R^{6} ~{\rm implies}
~~\,q^{\rm Wick}\geq 0.
\end{equation}
According to    \cite[Proposition 2.4.3]{MR2599384}, the Wick and Weyl
quantizations of a symbol $q$ are linked by
the following identities
\begin{eqnarray}\label{ww}
  q^{\rm Wick}=\inner{q* 2^{3} e^{-2\pi
\abs{\cdot}^2 }}^w=q^w+r^w
\end{eqnarray}
with
\[
  r(Y)= 2^{3} \int_0^1 \int_{\mathbb R^{6}} (1-\theta)q''(Y+\theta Z )Z^2e^{-2\pi 
\abs{Z}^2}d Zd\theta.
\]
As a result,  $q^{\rm Wick} $ is a bounded operator in $L_v^2$ if $q\in S(1, g)$,  and   $q^{\rm Wick} $ is self-adjoint in $L_v^2$ if $q$ is real-valued. 
We also recall the following composition
formula obtained in the proof of Proposition 3.4 in \cite{MR1957713}
\begin{eqnarray} \label{11082406}
  q_1^{\rm Wick}q_2^{\rm Wick}=\com{q_1q_2- q_1'\cdot q_2'+
  \frac{1}{i}\big \{q_1,~q_2\big\}}^{\rm Wick}+\CT,
\end{eqnarray}
with $\CT$ being a bounded operator in $L^2(\mathbb R^{2n})$,  where $q_1\in L^{\infty}
(\mathbb R^{2n})$ and $q_2$ is a smooth symbol whose derivatives of order $\geq2$
are bounded on $\mathbb R^{6}$. The notation $\set{q_1,q_2}$ denotes the Poisson
bracket defined by
\begin{eqnarray} \label{11051505}
  \big\{q_1,~q_2\big\}=\frac{\partial q_1}{\partial\eta}\cdot\frac{\partial q_2}{\partial
v}-\frac{\partial q_1}{\partial v}\cdot\frac{\partial q_2}{\partial \eta}.
\end{eqnarray}

\medskip
\noindent {\bf Acknowledgements.} 
RJD was partially supported by the NSFC/RGC Joint Research Scheme (N\_CUHK409/19) from RGC in Hong Kong and the Direct Grant (4053452) from CUHK. WXL was supported by NSFC (Nos. 11961160716, 11871054, 11771342) and the Fundamental Research Funds for the Central Universities(No.~2042020kf0210).

\medskip

\noindent{\bf Conflict of Interest:} The authors declare that they have no conflict of interest.

\end{document}